\RequirePackage{fix-cm}
\documentclass[numbook]{svjour3}                     
\smartqed  
\usepackage{graphicx}
\usepackage{amsmath}
\usepackage{amsfonts}
\usepackage{amssymb}
\usepackage{subeqnarray}
\usepackage{bigcenter}
\usepackage{epsfig}
\usepackage{multirow}%
\usepackage{enumerate}
\usepackage{enumitem}
\usepackage{lineno}
\usepackage{subfloat}
\usepackage{subfig}
\usepackage{hyperref}
\usepackage{cleveref}
\crefname{equation}{}{}
\usepackage{tikz}

\def\rhob{\boldsymbol{\rho}}
\def\etab{\boldsymbol{\eta}}

\def\Yb{\boldsymbol{\cal Y}}

\def\bcW{\boldsymbol{\cal W}}
\def\p{\mathrm{p}}
\def\T{\mathrm{T}}
\def\s{\mathrm{s}}

\def \LOCALFIGPATHRP {./IMAGES/}
\def \LOCALFIGPATHLOBB {./IMAGES/}
\def \LOCALFIGPATHCONE {./IMAGES/}

\begin{document}

\title{Energy relaxation approximation for the compressible multicomponent flows in thermal nonequilibrium
}

\titlerunning{FV methods for nonequilibrium multicomponent flows}        

\author{Claude Marmignon \and Fabio Naddei \and Florent Renac}

\authorrunning{Marmignon, Naddei, Renac} 

\institute{C. Marmignon \at DAAA, ONERA, Universit\'e Paris Saclay, F-92322 Ch\^atillon, France
 \and      F. Naddei \at DAAA, ONERA, Universit\'e Paris Saclay, F-92322 Ch\^atillon, France
 \and      F. Renac \at DAAA, ONERA, Universit\'e Paris Saclay, F-92322 Ch\^atillon, France \\ \email{florent.renac@onera.fr}
}

\date{Received: date / Accepted: date}

\maketitle

\begin{abstract}
This work concerns the numerical approximation with a finite volume method of inviscid, nonequilibrium, high-temperature flows in multiple space dimensions. It is devoted to the analysis of the numerical scheme for the approximation of the hyperbolic system in homogeneous form. We derive a general framework for the design of numerical schemes for this model from numerical schemes for the monocomponent compressible Euler equations for a polytropic gas. Under a very simple condition on the adiabatic exponent of the polytropic gas, the scheme for the multicomponent system enjoys the same properties as the one for the monocomponent system: discrete entropy inequality, positivity of the partial densities and internal energies, discrete maximum principle on the mass fractions, and discrete minimum principle on the entropy. Our approach extends the relaxation of energy [Coquel and Perthame, \textit{SIAM J. Numer. Anal.}, 35 (1998), 2223--2249] to the multicomponent Euler system. In the limit of instantaneous relaxation we show that the solution formally converges to a unique and stable equilibrium solution to the multicomponent Euler equations. We then use this framework to design numerical schemes from three schemes for the polytropic Euler system: the Godunov exact Riemann solver [Godunov, Math. Sbornik, 47 (1959), 271--306] and the HLL [Harten et al., SIAM Rev., 25 (1983), 35--61] and pressure relaxation based [Bouchut, Nonlinear stability of finite volume methods for hyperbolic conservation laws and well-balanced schemes for sources, Frontiers in Mathematics, Birkh\"auser, 2004] approximate Riemann solvers. Numerical experiments in one and two space dimensions on flows with discontinuous solutions support the conclusions of our analysis and highlight stability, robustness and convergence of the scheme.

\end{abstract}


%
%
\section{Introduction}

We are here interested in finite volume methods to simulate inviscid hypersonic high-temperature flows. Such simulations are of strong significance in many applications (e.g., hypersonic air vehicles \cite{KNIGHT20128}, reentry vehicles \cite{anderson_89}, meteoroid entry into atmosphere \cite{henneton_etal_15}) and scientific topics (e.g., weakly ionized gases, heat transfer \cite{PRAKASH20118474}, boundary layer stability \cite{Knisely_Zhong19}, shock propagation \cite{Honma_Glass_84}) related to hypersonic flows. For such flows, effects of thermal and chemical nonequilibrium are important and cannot be modeled by the monocomponent compressible Euler equations for a polytropic gas. Real gas models usually include multiple temperatures, chemical reaction rates and vibrational relaxation effects \cite{Giovangigli_book_99,anderson_89}. We here focus on issues related to the numerical treatment of the convective fluxes due to their hyperbolic nature and on the capture of associated features such as strong shocks. We therefore consider thermal nonequilibrium only and neglect chemical nonequilibrium and relaxation of vibration energies that are associated to numerical issues of different nature.

The numerical analysis of hypersonic flows is usually challenging because the characteristic time scales of the chemical reactions and molecular vibrations may be quite different from the characteristic time scale of the flow field. Taking into account the variations in the chemical composition and internal energy modes of a fluid requires to resolve the mass fractions and vibration energies. The thermodynamic properties then depend on these variables which complicates the design of numerical schemes with sought-after properties such as robustness (i.e., that keeps positivity of partial densities and internal and vibration energies), stability from a discrete entropy inequality, maximum principle on the mass fractions, etc. 

The design of numerical schemes for the approximation of the compressible multicomponent Euler equations has been an active field of research over the past decades. 
Park proposed  an implicit time marching associated to central differencing of ionized flows \cite{park85}, while finite volume discretizations have been widely developed with flux splitting techniques \cite{COLELLA_glaz_1985,candler_maccormack_85,liu_vinokur_83,liou_etal_split_RG_90}, Jacobian based methods such as the Roe method \cite{coquel_marmignon_ionized_95,GLAISTER1988361}, the AUSM scheme \cite{gaitonde_12}, relaxation based approximate Riemann solvers (ARS) \cite{soubrie_ionised_2012}, etc. High-order extensions have been proposed with the second-order MUSCL method \cite{Druguet_et_al_05}, ENO and WENO reconstructions \cite{TON1996,DRIKAKIS2003405}, interface capturing schemes \cite{abgrall_96,karni_mutlicomp_96}. Shock fitting techniques have also been addressed in \cite{PRAKASH20118474}. In this work we will consider the design of finite volume schemes based on ARS.

To ensure entropy stability and robustness when using ARS such as the HLL \cite{hll_83}, Roe \cite{Roe_1981}, Rusanov \cite{Rusanov1961}, relaxation \cite{bouchut_04,coquel_etal_relax_fl_sys_12} schemes, etc., one needs an estimation from above of the maximum wave speeds in the Riemann problem. However, fast estimates such as the two-rarefaction approximation \cite[Ch.~9]{toro_book}, the iterative algorithm from \cite{guermond_popov_16}, or the one based on eigenvalues of the Roe linearisation \cite{einfeldt_etal_91} will require time-consuming Newton-Raphson iterations when the equation of state (EOS) differs in the left and right states due to different species compositions. In \cite{dellacherie_03} a relaxation technique is applied to the multicomponent Euler system which allows the use of monocomponent schemes for each component and  associated EOS and the scheme inherits properties from the monocomponent scheme. However, this technique requires to compute as many monocomponent schemes as there are species which can become time consuming. Moreover, the entropy of the relaxation system is proved to be convex for constant mass fractions only which is valid for isolated shocks, but fails for interactions of shocks with material interfaces. Here we consider the energy relaxation technique introduced in \cite{coquel_perthame_98} for the approximation of the monocomponent compressible Euler equations with a general EOS. In this method, one considers a decomposition of the internal energy including the energy for a polytropic gas thus relaxing the general EOS. The method then allows the design of numerical schemes by using classical numerical fluxes for polytropic gases coupled to instantaneous relaxation of the energy.

In this work, we extend this method to our model and show how to define a numerical scheme from a scheme for the polytropic gas dynamics through a simple formula (equation \cref{eq:flux_h_from_H}) which corresponds to a splitting of hyperbolic and relaxation operators. In the limit of instantaneous relaxation we show that the solution of the energy relaxation approximation formally converges to a unique and stable equilibrium solution to the multicomponent Euler equations which justify the splitting. By defining the adiabatic exponent of the polytropic gas as an upper bound of the possible values of adiabatic exponent of the mixture, the scheme for the multicomponent system inherits the properties of the scheme for the monocomponent system: discrete entropy inequality, positivity of the partial densities and internal energies, discrete maximum principle on the mass fractions, and discrete minimum principle on the entropy. An attempt to apply the energy relaxation approximation to the multicomponent Euler system for a fluid mixture in thermal equilibrium has been made in \cite{renac2020entropy}. However, this work did not provide a general framework to build numerical schemes. The closure laws for the fluid mixture indeed prevent the derivation of a strictly convex entropy for the relaxation system which in turn prevents to apply stability theorems \cite{chen_levermore_liu94} to the relaxation process. As a consequence the well-posedness of the instantaneous relaxation process has not been investigated either. On the other hand, the present work successfully addresses this property and may use any polytropic three-point scheme.



The paper is organized as follows. \Cref{sec:HRM_model} presents the multicomponent compressible Euler system in thermal nonequilibrium and the entropy pair. The unstructured finite volume scheme and the three-point scheme are described in \cref{sec:FV-schemes}. We introduce and analyze the relaxation in energy approximation in \cref{sec:relax_flux} that we use in \cref{sec:ex-ES-3pt-schemes} to derive three numerical fluxes for the finite volume scheme. These three schemes are then assessed by numerical experiments in one and two space dimensions in \cref{sec:num_xp} and concluding remarks about this work are given in \cref{sec:conclusions}.

%
%
\section{Model problem}\label{sec:HRM_model}

\subsection{Governing equations and thermodynamic model}

Let $\Omega\subset\mathbb{R}^d$ be a bounded domain in $d$ space dimensions, we consider the multi-species and multi-temperature model for flows in thermal nonequilibrium \cite{park90}. Let the IBVP described by the multicomponent compressible Euler system for a mixture of $n_s$ species
\begin{subequations}\label{eq:HRM_PDEs}
\begin{align}
 \partial_t{\bf u} + \nabla\cdot{\bf f}({\bf u}) &= 0, \quad \mbox{in }\Omega\times(0,\infty), \label{eq:HRM_PDEs-a} \\
 {\bf u}(\cdot,0) &= {\bf u}_{0}(\cdot),\quad\mbox{in }\Omega, \label{eq:HRM_PDEs-b}
\end{align}
\end{subequations}

\noindent with some boundary conditions to be prescribed on $\partial\Omega$ (see \cref{sec:ES_BC}). Here
\begin{equation}\label{eq:HRM_u_f}
 {\bf u} = \begin{pmatrix}\rhob \\ \rho{\bf v} \\ \rho E \\ \rho {\bf e}_v \end{pmatrix}, \quad
 {\bf f}({\bf u}) = \begin{pmatrix} \rhob{\bf v}^\top\\ \rho{\bf v}{\bf v}^\top+\mathrm{p}{\bf I} \\(\rho E+\mathrm{p}){\bf v}^\top \\ \rho {\bf e}_v{\bf v}^\top\end{pmatrix},
\end{equation}

\noindent denote the conserved variables and the convective fluxes with $\rhob=(\rho_1,\dots,\rho_{n_s})^\top$ the vector of densities of the $n_s$ species, while $\rho$, ${\bf v}$ in $\mathbb{R}^d$, and $E$ denote the density, velocity vector, and total specific energy of the mixture, respectively. By $\rho{\bf e}_v=(\rho_1e_1^v,\dots,\rho_{n_d}e_{n_d}^v)^\top$ we denote the  vector of partial vibration energies of the $n_d$ diatomic species that are in thermal nonequilibrium. Each partial vibration energy is linked to the associated vibration temperature $\mathrm{T}_\beta^v$ through
\begin{equation}\label{eq:def_nrj_vib}
e_\beta^v(\T_\beta^v) = r_\beta\frac{\vartheta_\beta^v}{\exp\big(\tfrac{\vartheta_\beta^v}{\mathrm{T}_\beta^v}\big)-1}, \quad r_\beta = \frac{\cal R}{M_\beta},
\end{equation}

\noindent where $\vartheta_\beta^v$ is the characteristic harmonic oscillator temperature, and $r_\beta$ is the gas constant of the species $\beta$ with ${\cal R}$ is the universal gas constant and $M_\beta$ the molecular weight of the species.


The mixture density, pressure and vibration energy are defined from quantities of the individual species through
\begin{equation}\label{eq:def-rho-rhoE-p}
 \rho = \sum_{\alpha=1}^{n_s}\rho_\alpha = \rho \sum_{\alpha=1}^{n_s} Y_\alpha, \quad \mathrm{p} = \sum_{\alpha=1}^{n_s}\mathrm{p}_\alpha, \quad \rho e_v = \sum_{\beta=1}^{n_d} \rho_\beta e_\beta^v,
\end{equation}

\noindent where $Y_\alpha=\tfrac{\rho_\alpha}{\rho}$ denotes the mass fraction of the $\alpha$th species, so we have
\begin{equation}\label{eq:saturation_cond}
 \sum_{\alpha=1}^{n_s}Y_\alpha=1.
\end{equation}

The specific total energy of the mixture reads
\begin{equation}
 E = h_0 + e_t + e_v + e_c, \quad h_0 = \sum_{\alpha=1}^{n_s} Y_\alpha h_\alpha^0, \quad e_t=\sum_{\alpha=1}^{n_s}Y_\alpha e_\alpha^t, \quad e_c=\frac{1}{2}{\bf v}\cdot{\bf v}, 
\end{equation}

\noindent where $h_\alpha^0\geq0$ is the enthalpy of formation of species $\alpha$, $e_\alpha^t=C_{v_\alpha}^t\mathrm{T}$ denotes the internal translation-rotation energy with $C_{v_\alpha}^t=\tfrac{3}{2}r_\alpha$ for a monoatomic species and $C_{v_\alpha}^t=\tfrac{5}{2}r_\alpha$ for diatomic molecules. 
The EOS for the mixture pressure in \cref{eq:def-rho-rhoE-p} is given by the Dalton's law and the partial pressures are assumed to obey polytropic ideal gas EOSs:
\begin{equation}\label{eq:EOS_SG}
 \mathrm{p} = \sum_{\alpha=1}^{n_s}\rho_\alpha r_\alpha\mathrm{T} = \rho r({\bf Y})\mathrm{T}, \quad r({\bf Y}) = {\cal Z}({\bf Y}){\cal R}, \quad {\cal Z}({\bf Y})=\sum_{\alpha=1}^{n_s}\frac{Y_\alpha}{M_\alpha},
\end{equation}

\noindent where ${\bf Y}=(Y_1,\dots,Y_{n_s})^\top$. Note that the pressure may be also written as
\begin{equation}\label{eq:mixture_eos}
 \mathrm{p}({\bf Y},\rho,e_t) = \big(\gamma({\bf Y})-1\big)\rho e_t,
\end{equation}

\noindent with
\begin{equation}\label{eq:equiv-gamma}
 \gamma({\bf Y}) = \frac{r({\bf Y})}{C_{v_t}({\bf Y})} + 1, \quad r({\bf Y}) \overset{\cref{eq:EOS_SG}}{=} \sum_{\alpha=1}^{n_s} Y_\alpha r_\alpha, \quad C_{v_t}({\bf Y}) = \sum_{\alpha=1}^{n_s} Y_\alpha C_{v_\alpha}^t.
\end{equation}

This induces the following bounds on $\gamma({\bf Y})$:
\begin{equation}\label{eq:gamma-bound}
 \frac{7}{5} \leq \gamma({\bf Y}) \leq \frac{5}{3} \quad \forall 0\leq Y_{1\leq \alpha\leq n_s}\leq1.
\end{equation}

System \cref{eq:HRM_PDEs-a} is hyperbolic in the direction ${\bf n}$ in $\mathbb{R}^d$ over the set of states \cite{liu_vinokur_83}
\begin{equation}\label{eq:HRM-set-of-states}
 \Omega^a=\{{\bf u}\in\mathbb{R}^{n_s+n_d+d+1}:\; \rho_{1\leq\alpha\leq n_s} > 0, {\bf v}\in\mathbb{R}^d, e_t>0, e_{1\leq\beta\leq n_d}^v>0\},
\end{equation}

\noindent with eigenvalues $\lambda_1={\bf v}\cdot{\bf n}-c\leq \lambda_2=\dots=\lambda_{n_s+n_d+d}={\bf v}\cdot{\bf n}\leq\lambda_{n_s+n_d+d+1}={\bf v}\cdot{\bf n}+c$, where $\lambda_{1}$ and $\lambda_{n_s+n_d+d+1}$ are associated to genuinely nonlinear fields and $\lambda_{2\leq i\leq n_s+n_d+d}$ to linearly degenerate fields. The frozen sound speed reads
\begin{equation}\label{eq:sound_speed}
 c({\bf Y},e_t) = \sqrt{\gamma({\bf Y})\big(\gamma({\bf Y})-1\big)e_t}. 
\end{equation}

Finally note that we are assuming in \cref{eq:HRM-set-of-states} that the partial densities are positive which would prevent vanishing phases: $\rho_\alpha=0$ for some $\alpha$. When such situation occurs the partial velocities, pressure and energies of the species also vanish and this is equivalent to removing the species in the model \cref{eq:HRM_PDEs} so $\rho_\alpha>0$ in \cref{eq:HRM-set-of-states} is justified and do not exclude vanishing phases.

\subsection{Entropy pair}

Solutions to \cref{eq:HRM_PDEs} should satisfy an entropy inequality
\begin{equation}\label{eq:entropy_ineq_cont}
 \partial_t\eta({\bf u}) + \nabla\cdot{\bf q}({\bf u}) \leq 0
\end{equation}

\noindent for some entropy -- entropy flux  pair $(\eta,{\bf q})$ with $\eta(\cdot)$ a strictly convex function and $\etab'({\bf u})^\top{\bf f}_i'({\bf u})={\bf q}_i'({\bf u})^\top$ for $1\leq i\leq d$. In this section we recall the entropy pair for \cref{eq:HRM_PDEs} derived in \cite{flament_prudhomme_93} and then prove convexity of the entropy.


Following \cite{flament_prudhomme_93}, the entropy for a mixture with internal degrees of freedom in nonequilibrium is the sum of associated entropies defined by their differential forms
\begin{subequations}\label{eq:2nd_principe}
\begin{align}
 \mathrm{T}d\mathrm{s}_\alpha^t &= de_\alpha^t + \mathrm{p}_\alpha d\tau_\alpha \quad \forall 1\leq\alpha\leq n_s, \label{eq:2nd_principe-a} \\ \T_\beta^vd\mathrm{s}_\beta^v &= de_\beta^v \quad \forall 1\leq\beta\leq n_d, \label{eq:2nd_principe-b}
\end{align}
\end{subequations}
\noindent with $\tau_\alpha=\tfrac{1}{\rho_\alpha}$ the covolume of the species $\alpha$. The entropy pair in \cref{eq:entropy_ineq_cont} reads
\begin{equation}\label{eq:phys_entropy_pair}
 \eta({\bf u}) =-\rho\mathrm{s}({\bf u}), \quad {\bf q}({\bf u}) =-\rho\mathrm{s}({\bf u}){\bf v}, \quad \mathrm{s}\equiv \sum_{\alpha=1}^{n_s}Y_\alpha\mathrm{s}_\alpha^t + \sum_{\beta=1}^{n_d}Y_\beta\mathrm{s}_\beta^v.
\end{equation}

Neglecting rotation-vibration coupling and anharmonic contributions, the specific entropies read \cite{flament_prudhomme_93} (up to some additive constants) 
\begin{subequations}\label{eq:def_partial_entropy}
\begin{align}
 \mathrm{s}_\alpha^t(\tau_\alpha,e_\alpha^t) &=C_{v_\alpha}^t\ln(e_\alpha^t) +  r_\alpha\ln(\tau_\alpha) \quad \forall 1\leq\alpha\leq n_s, \label{eq:def_partial_entropy-a}\\ 
 \mathrm{s}_\beta^v(e_\beta^v) &=  r_\beta\ln(e_\beta^v) +  r_\beta\Big(1+\tfrac{e_\beta^v}{r_\beta\vartheta_\beta^v}\Big)\ln\Big(1+\tfrac{r_\beta\vartheta_\beta^v}{e_\beta^v}\Big) \quad \forall 1\leq\beta\leq n_d. \label{eq:def_partial_entropy-b}
\end{align}
\end{subequations}

Note that for smooth solutions, manipulations of \cref{eq:HRM_PDEs} together with \cref{eq:2nd_principe} show that these entropies satisfy the following conservation laws
\begin{equation*}
 \partial_t\Big(\sum_{\alpha=1}^{n_s}\rho_\alpha\mathrm{s}_\alpha^t\Big) + \nabla\cdot\Big(\sum_{\alpha=1}^{n_s}\rho_\alpha\mathrm{s}_\alpha^t{\bf v}\Big) = 0, \quad \partial_t(\rho_\beta\mathrm{s}_\beta^v) + \nabla\cdot(\rho_\beta\mathrm{s}_\beta^v{\bf v}) = 0 \quad \forall 1\leq\beta\leq n_d.
\end{equation*}



%
\begin{proposition}\label{th:convexity_entropy_ms_euler}
 The entropy in \cref{eq:phys_entropy_pair} is a strictly convex and twice differentiable function of ${\bf u}$ in $\Omega^a$.
\end{proposition}

\begin{proof}
Twice differentiability is straightforward from \cref{eq:def_partial_entropy}. To prove the convexity we use the trick introduced in \cite{harten_etal_hyper_98} and also used in \cite{gouasmi_etal_min_pcpe_20} to prove that the Hessian of the entropy $\boldsymbol{\cal H}_\eta$ is congruent to the following strictly convex diagonal matrix:
\begin{align}\label{eq:congruence_hessian}
  \frac{\partial{\bf u}}{\partial{\bf Z}}^\top\boldsymbol{\cal H}_\eta\frac{\partial{\bf u}}{\partial{\bf Z}} &= \frac{\partial{\bf u}}{\partial{\bf Z}}^\top\frac{\partial{\etab'(\bf u})}{\partial{\bf Z}} \nonumber\\
  &= \mbox{diag}\Big((r_\alpha\tau_\alpha)_{1\leq\alpha\leq n_s},(\rho\theta)_{1\leq i\leq d},\rho C_{v_t}({\bf Y})\theta^2, \big(-\rho_\beta \s_\beta^{v''}(e_\beta^v)\big)_{1\leq\beta\leq n_d}\Big),
\end{align}
\noindent\noindent where $\theta=\tfrac{1}{\mathrm{T}}$, $\s_\beta^{v''}(e_\beta^v)=-\tfrac{r_\beta}{e_\beta^v(e_\beta^v+r_\beta\vartheta_\beta^v)}<0$ from \cref{eq:def_partial_entropy-b}, and ${\bf Z}({\bf u})=(\rhob^\top,{\bf v}^\top,\mathrm{T},{\bf e}_v^\top)^\top$ denotes a one-to-one change of variables. Indeed, with some slight abuse in the notation we have
\begin{equation*}
 \frac{\partial{\bf u}}{\partial{\bf Z}} = \begin{pmatrix}
  1 &  &  &  & 0 & 0 & 0 & 0 & 0 & 0\\
	  & \ddots & & &  & \vdots & \vdots & \vdots & \vdots & \vdots \\
	  &  & \ddots &  &  & \vdots & \vdots & \vdots & \vdots & \vdots \\
		&  &   & \ddots & & \vdots & \vdots & \vdots & \vdots & \vdots \\
	0 &  &  & & 1 & 0 & 0 & 0 & 0 & 0 \\
	{\bf v} & \cdots & {\bf v} & \cdots & {\bf v} & \rho{\bf I}_d & 0 & 0 & 0 & 0 \\
	\rho E_1 & \cdots & \rho E_{n_d} & \cdots & \rho E_{n_s} & \rho{\bf v}^\top & \rho C_{v_t}({\bf Y}) & \rho_1 & \cdots & \rho_{n_d} \\
	e_1^v &  & 0 & \cdots & 0 & 0 & 0 & \rho_1 &  & 0 \\
	  & \ddots & & \ddots & \vdots & \vdots & \vdots & & \ddots & \\
	0 &   & e_{n_d}^v & & 0 & 0 & 0 & 0 &  &\rho_{n_d}
\end{pmatrix},
\end{equation*}

\noindent where ${\bf I}_d$ is the identity matrix of size $d$ and $\rho E_\alpha=\partial_{\rho_\alpha}\rho E=C_{v_\alpha}^t\mathrm{T}+h_\alpha^0+\psi_\alpha e_\alpha^v+e_c$ with $\psi_\alpha=1$ if $1\leq\alpha\leq n_d$ and $\psi_\alpha=0$ if $n_d<\alpha\leq n_s$. So $\det(\partial_{{\bf Z}}{\bf u})=\rho^{d+1}C_{v_t}({\bf Y})\Pi_{\beta=1}^{n_d}\rho_\beta>0$. 

Let $g_\alpha^t=e_\alpha^t+\mathrm{p}_\alpha\tau_\alpha-\mathrm{T}\mathrm{s}_\alpha^t$ be the free Gibbs energy of the $\alpha$th species and $g_\beta^v=e_\beta^v-\mathrm{T}_\beta^v\mathrm{s}_\beta^v$, using the differential forms \cref{eq:2nd_principe} we obtain 
\begin{align*}
\T\sum_{\alpha=1}^{n_s} d(\rho_\alpha\mathrm{s}_\alpha^t) &= \sum_{\alpha=1}^{n_s}  d(\rho_\alpha e_\alpha^t) - (e_\alpha^t+\mathrm{p}_\alpha\tau_\alpha-\mathrm{T}\mathrm{s}_\alpha^t)d\rho_\alpha\\
&=  d(\rho E) - d(\rho e_v) - {\bf v}\cdot d(\rho{\bf v}) - \sum_{\alpha=1}^{n_s}  (g_\alpha^t+h_\alpha^0-e_c)d\rho_\alpha, \\
\sum_{\beta=1}^{n_d} d(\rho_\beta\s_\beta^v) &= \sum_{\beta=1}^{n_d} \big(\s_\beta^v(e_\beta^v)-e_\beta^v\theta_\beta^{v}\big)d\rho_\beta + \theta_\beta^{v}d(\rho_\beta e_\beta^v),
\end{align*}

\noindent with $\theta_\beta^v=\tfrac{1}{\T_\beta^v}=\s_\beta^{v'}(e_\beta^v)$, so the entropy variables read
\begin{equation}\label{eq:entropy_var}
  \etab'({\bf u}) := \bigg(\frac{\eta({\bf u})}{\partial{\bf u}}\bigg)^\top  = \begin{pmatrix} C_{v_1}^t+ r_1-\mathrm{s}_1^t + \psi_1\theta_1^v g_1^v+(h_1^0-e_c)\theta \\ \vdots \\ C_{v_{n_s}}^t+ r_{n_s}-\mathrm{s}_{n_s}^t + \psi_{n_s}\theta_{n_s}^v g_{n_s}^v+(h_{n_s}^0-e_c)\theta \\ \theta{\bf v} \\ -\theta \\ \theta-\theta_1^v \\ \vdots \\ \theta-\theta_{n_d}^v \end{pmatrix},
\end{equation}

\noindent and we obtain for $\tfrac{\partial{\etab'(\bf u})}{\partial{\bf Z}}$
\begin{equation*}
\begin{pmatrix}
  \tfrac{r_1}{\rho_1} &  & & & 0 & -\theta{\bf v}^\top & -C_{v_1}^t\theta-(h_1^0-e_c)\theta^2 & e_1^v\s_1^{v''}(e_1^v) & & 0 \\
	  & \ddots & & & & \vdots & \vdots & & \ddots & \\
	  &  & \tfrac{r_{n_d}}{\rho_{n_d}} & & & -\theta{\bf v}^\top & -C_{v_{n_d}}^t\theta-(h_{n_d}^0-e_c)\theta^2 & 0 & & e_{n_d}^v\s_{n_d}^{v''}(e_{n_d}^v) \\
	  &  &  & \ddots &  & \vdots & \vdots & \vdots & \ddots & 0 \\
 	0 & & &  & \tfrac{r_{n_s}}{\rho_{n_s}} & -\theta{\bf v}^\top & -C_{v_{n_s}}^t\theta-(h_{n_s}^0-e_c)\theta^2 & 0 & \cdots & 0 \\
	0 & \cdots & \cdots & \cdots & 0 & \theta{\bf I}_d & -\theta^2{\bf v} & 0 & \cdots & 0 \\
	0 & \cdots & \cdots & \cdots & 0 & 0 & \theta^2 & 0 & \cdots & 0 \\
	0 & \cdots & 0 & \cdots & 0 & 0 & -\theta^2 & -\s_1^{v''}(e_1^v) & & 0 \\
	\vdots & \ddots & \vdots & \ddots & \vdots & \vdots & \vdots & & \ddots & \\
	0 & \cdots & 0 & \cdots & 0 & 0 & -\theta^2 &  0 & & -\s_{n_d}^{v''}(e_{n_d}^v)
\end{pmatrix},
\end{equation*}
\noindent so it may be easily checked that \cref{eq:congruence_hessian} holds true. \qed
\end{proof}

Finally, let $\tau=\tfrac{1}{\rho}$ be the covolume of the mixture. Using $\tfrac{e_\alpha^t}{C_{v_\alpha}^t}=\tfrac{e_t}{C_{v_t}({\bf Y})}=\mathrm{T}$ and $Y_\alpha\tau_\alpha=\tau$, for all $\alpha$, in \cref{eq:def_partial_entropy} the entropy of the mixture in \cref{eq:phys_entropy_pair} may be written as
\begin{subequations}\label{eq:mixture_entropy}
\begin{align}
  \mathrm{s}({\bf Y},\tau,e_t,{\bf e}_v) &= \sum_{\alpha=1}^{n_s} Y_\alpha C_{v_\alpha}^t\ln\Big(\frac{C_{v_\alpha}^t}{C_{v_t}({\bf Y})}e_t\Big) + Y_\alpha r_\alpha\ln\Big(\frac{\tau}{Y_\alpha}\Big) + \s_v({\bf Y},{\bf e}_v)\nonumber \\
	&= C_{v_t}({\bf Y})\ln e_t + r({\bf Y})\ln\tau+K({\bf Y}) + \s_v({\bf Y},{\bf e}_v),  \\
	K({\bf Y}) &= \sum_{\alpha=1}^{n_s}Y_\alpha\Big(C_{v_\alpha}^t\ln\big(\tfrac{C_{v_\alpha}^t}{C_{v_t}({\bf Y})}\big)-r_\alpha\ln Y_\alpha\Big), \\
	\s_v({\bf Y},{\bf e}_v) &= \sum_{\beta=1}^{n_d} Y_\beta \s_\beta^v(e_\beta^v).
\end{align}
\end{subequations}

\section{Finite volume method}\label{sec:FV-schemes}

We consider finite volume schemes for unstructured meshes $\Omega_h\subset\mathbb{R}^d$ of the form
\begin{equation}\label{eq:2D-FV-scheme}
 {\bf U}_\kappa^{n+1} - {\bf U}_\kappa^{n} + \frac{\Delta t^{(n)}}{|\kappa|}\sum_{e\in\partial\kappa} |e|{\bf h}({\bf U}_\kappa^{n},{\bf U}_{\kappa_e^+}^{n},{\bf n}_e) = 0 \quad \forall \kappa\in\Omega_h, n\geq 0,
\end{equation}

\noindent for the discretization of \cref{eq:HRM_PDEs-a}. Here ${\bf U}_{\kappa}^{n+1}$ approximates the averaged solution in the cell $\kappa$ at time $t^{(n+1)}=t^{(n)}+\Delta t^{(n)}$, $\Delta t^{(n)}>0$ is the time step, ${\bf n}_e$ is the unit outward normal vector on the edge $e$ in $\partial\kappa$, and $\kappa_e^+$ the neighboring cell sharing the interface $e$ (see \cref{fig:stencil_2D}). We assume that each element is shape-regular in the sense of \cite{ciarlet2002finite}: the ratio of the radius of the largest inscribed ball to the diameter is bounded by below by a positive constant independent of the mesh. The initial condition for \cref{eq:2D-FV-scheme} reads
\begin{equation*}
 {\bf U}_\kappa^{0} = \frac{1}{|\kappa|}\int_\kappa {\bf u}_0({\bf x}) dV \quad \forall \kappa\in\Omega_h.
\end{equation*}

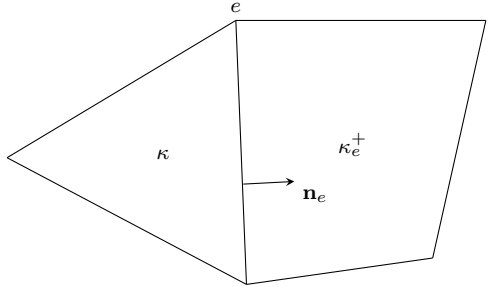
\begin{figure}[ht]
 \begin{center}
  \begin{tikzpicture}[scale=0.7]
   \draw (1.43,1.47) node {$\kappa$};
   \draw (5.,1.65)   node {$\kappa_e^+$};
   \draw (2.8,4.0)    node[above] {$e$};
   
   \draw [>=stealth,->] (2.94,0.90) -- (3.9,0.95) ;
   \draw (3.9,0.95) node[below right] {${\bf n}_e$};
   
   \draw [>=stealth,-] (-1.5,1.4) -- (3.,-1.) ;
   \draw [>=stealth,-] (3.,-1.) -- (2.8,4.0) ;
   \draw [>=stealth,-] (2.8,4.0) -- (-1.5,1.4) ;
   \draw [>=stealth,-] (3.,-1.) -- (6.5,-0.5) ;
   \draw [>=stealth,-] (6.5,-0.5) -- (7.5,4.) ;
   \draw [>=stealth,-] (7.5,4.) -- (2.8,4.0) ;
  \end{tikzpicture}
  \caption{Notations for the mesh for $d=2$.}
  \label{fig:stencil_2D}
 \end{center}
\end{figure}

It is convenient to also consider three-point numerical schemes of the form
\begin{equation}\label{eq:3pt-scheme-a}
  {\bf U}_j^{n+1} - {\bf U}_j^{n} + \frac{\Delta t^{(n)}}{\Delta x}\big({\bf h}({\bf U}_{j}^{n},{\bf U}_{j+1}^{n},{\bf n})-{\bf h}({\bf U}_{j-1}^{n},{\bf U}_{j}^{n},{\bf n})\big) = 0, 
\end{equation}

\noindent where ${\bf U}_{j}^{n}$ approximates the averaged solution in the $j$th cell at time $t^{(n)}$, $\Delta x$ is the space step. In particular we are looking for schemes \cref{eq:3pt-scheme-a} that have the following properties under a given condition on the time step
\begin{equation}\label{eq:3pt-scheme-CFL}
  \frac{\Delta t^{(n)}}{\Delta x}\max_{j\in\mathbb{Z}}|\lambda({\bf U}_{j}^{n})|\leq\frac{1}{2},
\end{equation}

\noindent where $|\lambda(\cdot)|$ corresponds to the maximum absolute value of the wave speeds (and will be defined in \cref{sec:ex-ES-3pt-schemes}): the scheme is

\begin{enumerate}[label=(\roman*)]
 \item {\it consistent} with \cref{eq:HRM_PDEs-a} and {\it conservative} which requires the numerical flux to be consistent:
\begin{equation}\label{eq:consistent_flux}
 {\bf h}({\bf u},{\bf u},{\bf n}) = {\bf f}({\bf u})\cdot{\bf n} \quad \forall {\bf u}\in\Omega^a,
\end{equation}

\noindent and conservative:
\begin{equation}\label{eq:conservative_flux}
   {\bf h}({\bf u}^-,{\bf u}^+,{\bf n}) =-{\bf h}({\bf u}^+,{\bf u}^-,-{\bf n}) \quad \forall {\bf u}^\pm\in\Omega^a;
\end{equation}
 \item {\it Lipschitz continuous} which also requires the numerical flux to be Lipschitz continuous;
 \item {\it entropy stable (ES)} for the pair $(\eta,{\bf q})$ in \cref{eq:entropy_ineq_cont}: it satisfies the inequality
\begin{equation}\label{eq:3pt-scheme-ineq}
 \eta({\bf U}_j^{n+1})-\eta({\bf U}_j^{n}) + \frac{\Delta t^{(n)}}{\Delta x}\big(Q({\bf U}_{j}^{n},{\bf U}_{j+1}^{n},{\bf n})-Q({\bf U}_{j-1}^{n},{\bf U}_{j}^{n},{\bf n})\big) \leq 0, 
\end{equation}
\noindent with some conservative and consistent entropy numerical flux $Q({\bf u},{\bf u},{\bf n}) = {\bf q}({\bf u})\cdot{\bf n}$;

 \item {\it robust}: the solution remains in the set of states \cref{eq:HRM-set-of-states}: ${\bf U}_{j\in\mathbb{Z}}^{n}$ in $\Omega^a$ implies ${\bf U}_{j\in\mathbb{Z}}^{n+1}$ in $\Omega^a$;

 \item and it satisfies a {\it discrete maximum principle on the mass fractions}: 
\begin{equation}\label{eq:3pt-scheme-mass-minmax}
 \min(Y_{\alpha_{j-1}}^{n},Y_{\alpha_{j}}^{n},Y_{\alpha_{j+1}}^{n})\leq Y_{\alpha_j}^{n+1}\leq\max(Y_{\alpha_{j-1}}^{n},Y_{\alpha_{j}}^{n},Y_{\alpha_{j+1}}^{n}) \; \forall 1\leq \alpha\leq n_s,
\end{equation}

 \item together with a {\it minimum principle on the specific entropy} in \cref{eq:phys_entropy_pair} \cite{tadmor86,gouasmi_etal_min_pcpe_20}:
\begin{equation}\label{eq:3pt-scheme-entropy-minmax}
 \s({\bf U}_j^{n+1})\geq\min\big(\s({\bf U}_{j-1}^{n}),\s({\bf U}_j^{n}),\s({\bf U}_{j+1}^{n})\big).
\end{equation}
\end{enumerate}

Then it is a classical matter (see e.g. \cite{perthame_shu_96,tadmor87,tadmor03,godlewski-raviart} and references therein) that the finite volume scheme \cref{eq:2D-FV-scheme} with the same numerical flux enjoys similar properties. Under the following condition on the time step
\begin{equation}\label{eq:CFL-positive-2D-scheme}
  \Delta t^{(n)} \max_{\kappa\in\Omega_h}\frac{|\partial\kappa|}{|\kappa|}\max_{e\in\partial\kappa}|\lambda({\bf U}_{\kappa^\pm}^{n})| \leq \frac{1}{2}, \quad |\partial\kappa|:=\sum_{e\in\partial\kappa}|e|,
\end{equation}

\noindent the scheme is robust and is a convex combination of updates of three-point schemes \cref{eq:3pt-scheme-a}:
\begin{equation}\label{eq:2D-FV-scheme-update-3pt}
 {\bf U}_\kappa^{n+1} = \sum_{e\in\partial\kappa} \frac{|e|}{|\partial\kappa|}\Big({\bf U}_\kappa^{n} - \frac{\Delta t^{(n)}|\partial\kappa|}{|\kappa|}\big({\bf h}({\bf U}_\kappa^{n},{\bf U}_{\kappa_e^+}^{n},{\bf n}_e)-{\bf h}({\bf U}_\kappa^{n},{\bf U}_{\kappa}^{n},{\bf n}_e)\big)\Big),
\end{equation}

\noindent with weights $\tfrac{|e|}{|\partial\kappa|}$. Therefore, the scheme \cref{eq:2D-FV-scheme} also satisfies the discrete minimum and maximum principles together with the entropy inequality
\begin{equation}\label{eq:2D-FV-scheme-ineq}
  \eta({\bf U}_\kappa^{n+1})-\eta({\bf U}_\kappa^{n}) + \frac{\Delta t^{(n)}}{|\kappa|}\sum_{e\in\partial\kappa} |e|Q({\bf U}_\kappa^{n},{\bf U}_{\kappa_e}^{n},{\bf n}_e) \leq 0,
\end{equation}

\noindent consistent with \cref{eq:entropy_ineq_cont}.

In the following we design finite volume schemes \cref{eq:2D-FV-scheme} with the CFL condition \cref{eq:CFL-positive-2D-scheme} by first designing three-point schemes \cref{eq:3pt-scheme-a} that satisfy properties (i) to (vi) with \cref{eq:3pt-scheme-CFL}.

%
%
\section{Energy relaxation approximation}\label{sec:relax_flux}

In this section we derive a general framework that allows the use of standard numerical schemes for the classical gas dynamics with a polytropic ideal gas EOS. The main results are summarized in \cref{th:ES-3-pt-scheme} and show how to build a three-point scheme for \cref{eq:HRM_PDEs-a} that enjoys the properties (i) to (vi) in \cref{sec:FV-schemes} from a three-point scheme for the compressible Euler equations with a polytropic law.

We here extend the energy relaxation approximation for the multicomponent Euler system \cite{coquel_perthame_98} to include the vibration energies (\cref{sec:nrj_relax_sys}) and introduce a convex entropy in \cref{sec:nrj_relax_entropy}. \Cref{sec:nrj_relax_var_pcpe} is devoted to the analysis of solutions to the relaxation system close to equilibrium. In the limit of instantaneous relaxation, we prove that:
\begin{itemize}
 \item solutions to the relaxation system formally converge to a unique and stable equilibrium solution to the multicomponent Euler equations \cref{eq:HRM_PDEs-a} (\cref{th:H-theorem});
 \item this equilibrium corresponds to a global minimum of the relaxation entropy which satisfies a variational principle (\cref{th:var-principle});
 \item small perturbations close to the equilibrium are associated to dissipative processes in \cref{eq:HRM_PDEs-a} (\cref{th:chapman-enskog}).
\end{itemize}
 
These results are then used to infer a numerical scheme for \cref{eq:HRM_PDEs-a} from one for the relaxation system (\cref{sec:FVrelax2FVEuler_ms}) based on a splitting of the hyperbolic and relaxation operators.

\subsection{Energy relaxation system}\label{sec:nrj_relax_sys}

Following the energy relaxation method introduced in \cite{coquel_perthame_98}, we consider the system
\begin{equation}\label{eq:relax-nrj-HRM-sys}
 \partial_t{\bf w}^\epsilon + \nabla\cdot{\bf g}({\bf w}^\epsilon) =-\frac{1}{\epsilon}\big({\bf w}^\epsilon-{\cal M}({\bf w}^\epsilon)\big),
\end{equation}
\noindent and we will denote by \cref{eq:relax-nrj-HRM-sys}$_{\epsilon\rightarrow\infty}$ the system in homogeneous form, i.e., with $\epsilon\rightarrow\infty$. Here
\begin{equation*}
 {\bf w} = \begin{pmatrix}\rhob \\ \rho {\bf v} \\ \rho E_r \\ \rho {\bf e}_v \\ \rho e_s \end{pmatrix}, \;
{\bf g}({\bf w}) = \begin{pmatrix} \rhob{\bf v}^\top \\ \rho {\bf v}{\bf v}^\top+\mathrm{p}_r(\rho,e_r){\bf I} \\ \big(\rho E_r+\mathrm{p}_r(\rho,e_r)\big){\bf v}^\top \\ \rho {\bf e}_v {\bf v}^\top \\ \rho e_s {\bf v}^\top \end{pmatrix}, \;
 {\bf w}-{\cal M}({\bf w}) = \begin{pmatrix} 0 \\ 0 \\ \rho\big(F({\bf Y},e_r)-e_s\big) \\ 0 \\ \rho\big(e_s-F({\bf Y},e_r)\big) \end{pmatrix},
\end{equation*}
\noindent with $\epsilon>0$ the relaxation time scale, and
\begin{equation}\label{eq:EOS-relax}
 \mathrm{p}_r(\rho,e_r) = (\gamma-1)\rho e_r, \quad e_r=E_r-e_c.
\end{equation}

Solutions to \cref{eq:relax-nrj-HRM-sys} satisfy the additional conservation law
\begin{equation}\label{eq:cons_law_mixture_density}
 \partial_t\rho^\epsilon + \nabla\cdot(\rho^\epsilon{\bf v}^\epsilon) = 0,
\end{equation}

\noindent for the mixture density so the variables $\rho$, $\rho{\bf v}$ and $\rho E_r$ are uncoupled from ${\bf Y}$ and ${\bf e}_v$ and coupled to $e_s$ through the relaxation source terms only. This is an important aspect of the model \cref{eq:relax-nrj-HRM-sys} and also allows to interpret $E_r$, $e_r$, and $\p_r$ as the energies and pressure of a polytropic EOS.

From \cref{eq:gamma-bound} we set $\gamma$ as
\begin{equation}\label{eq:gamma-max}
 \gamma>\max_{0\leq Y_{1\leq \alpha\leq n_s}\leq1}\gamma({\bf Y}) = \frac{5}{3},
\end{equation}
\noindent which constitutes the subcharacteristic condition for \cref{eq:relax-nrj-HRM-sys} to relax to an equilibrium as $\epsilon\downarrow0$ \cite{coquel_perthame_98}. The set of states for \cref{eq:relax-nrj-HRM-sys} is 
\begin{equation}\label{eq:relax-set-of-states}
 \Omega^r=\big\{{\bf w}\in\mathbb{R}^{n_s+n_d+d+2}:\; \rho_{1\leq\alpha\leq n_s}>0, {\bf v}\in\mathbb{R}^d, e_r>0, e_{1\leq\beta\leq n_d}^v>0, e_s>0\big\}. 
\end{equation}

Let ${\bf w}=\lim_{\epsilon\downarrow0}{\bf w}^\epsilon$, in this limit, one formally recovers \cref{eq:HRM_PDEs-a} with
\begin{equation}\label{eq:maxwellian-equ}
 {\bf u} = {\cal L}{\bf w}, \quad {\bf w} = {\cal M}({\bf w}), \quad {\bf f}({\bf u}) = {\cal L}{\bf g}({\cal P}({\bf u})),
\end{equation}

\noindent with the operators ${\cal L}:\Omega^r\rightarrow\Omega^a$ and ${\cal P}:\Omega^a\rightarrow\Omega^r$ defined by
\begin{subequations}\label{eq:relax-L-P-operators}
\begin{align}
 {\cal L}{\bf w}&=\big(\rhob^\top, \rho{\bf v}^\top, \rho E_r+\rho e_s + \rho h_0 + \rho e_v, \rho {\bf e}_v^\top\big)^\top,  \\
 {\cal P}({\bf u})&=\big(\rhob^\top, \rho{\bf v}^\top, \rho e_r(\rho,\mathrm{p})+\rho e_c, \rho {\bf e}_v^\top,\rho\big(e_t-e_r(\rho,\mathrm{p})\big)\big)^\top,
\end{align}
\end{subequations}
\noindent with $e_r(\rho,\mathrm{p}_r)$ defined from \cref{eq:EOS-relax}, $\rho=\sum_{\alpha=1}^{n_s}\rho_\alpha$, and $\mathrm{p}=\mathrm{p}({\bf Y},\rho,e_t)$ defined from \cref{eq:mixture_eos}. The equilibrium \cref{eq:maxwellian-equ} corresponds to
\begin{equation}\label{eq:maxwellian-equ-hrm}
 E = E_r + e_s + h_0 + e_v, \quad e_t = e_r + e_s, \quad e_s = F({\bf Y},e_r):=\frac{\gamma-\gamma({\bf Y})}{\gamma({\bf Y})-1}e_r,
\end{equation}
\noindent where the expression for $F$ follows from the consistency relation on the pressure: 
\begin{equation}\label{eq:maxwellian-equ-hrm-press}
 \mathrm{p}\big({\bf Y},\rho,e_r\\+F({\bf Y},e_r)\big)=\mathrm{p}_r(\rho,e_r) \; \overset{\cref{eq:mixture_eos}}{\underset{\cref{eq:EOS-relax}}{\Leftrightarrow}} \; \big(\gamma({\bf Y})-1\big)\rho\big(e_r+F({\bf Y},e_r)\big) = (\gamma-1)\rho e_r.
\end{equation}


%
\subsection{Entropy}\label{sec:nrj_relax_entropy}

Let define the convex function
\begin{equation}\label{eq:entropy-Sr}
 \mathrm{s}_r(\tau,e_r) = -(\tau^{\gamma-1} e_r)^{\frac{1}{\gamma}},
\end{equation}

\noindent with $\tau=\tfrac{1}{\rho}$ the covolume of the mixture, and further introduce
\begin{subequations}\label{eq:functions-zeta-R-E}
\begin{align}
 &\zeta(\Yb,\tau,e_r,e_s,{\bf e}_v) = -\mathrm{s}\Big({\bf Y}(\Yb),{\cal T}\big(\Yb,\mathrm{s}_r(\tau,e_r),e_s\big),{\cal E}(\Yb,e_s)+e_s,{\bf e}_v\Big), \label{eq:functions-zeta-R-E-a} \\ 
&{\cal E}(\Yb,e_s) = \tfrac{\gamma({\bf Y})-1}{\gamma-\gamma({\bf Y})}e_s, \;\;
{\cal T}(\Yb,\mathrm{s}_r,e_s) = \big(\tfrac{\gamma-\gamma({\bf Y})}{\gamma({\bf Y})-1}\tfrac{(-\mathrm{s}_r)^\gamma}{e_s}\big)^{\frac{1}{\gamma-1}}, \;\;
{\bf Y}={\bf Y}(\Yb),  \label{eq:functions-zeta-R-E-b}
\end{align}
\end{subequations}
\noindent where $\mathrm{s}$ is the mixture entropy \cref{eq:mixture_entropy} for \cref{eq:HRM_PDEs-a} and 
\begin{equation}\label{eq:mapping_Y_Yb}
 \Yb=(Y_1,\dots,Y_{n_s-1})^\top
\end{equation}

\noindent the vector of the mass fractions of the $n_s-1$ first species. This particular change of variables will be used only in the proof of \cref{th:var-principle} below where we will clarify the choice for $n_s$. Note that the mapping between ${\bf Y}={\bf Y}(\Yb)$ is obviously one-to-one from \cref{eq:saturation_cond,eq:cons_law_mixture_density} which is always satisfied for ${\bf w}$ in $\Omega^r$ in \cref{eq:relax-set-of-states}, so we may adopt equivalently the notations ${\bf Y}$ or $\Yb$ for the sake of clarity and write
\begin{equation}\label{eq:def2_r}
  r({\bf Y}) = r(\Yb) = r_{n_s} + \sum_{\alpha=1}^{n_s-1} Y_\alpha(r_\alpha-r_{n_s}). 
\end{equation}

In \cref{eq:functions-zeta-R-E-b}, the function ${\cal E}$ solves $e_s=F({\bf Y},e_r)$ for $e_r$ with $F$ defined in \cref{eq:maxwellian-equ-hrm}, while ${\cal T}$ solves $\mathrm{s}_r=\mathrm{s}_r(\tau,{\cal E}(\Yb,e_s))$ for $\tau$ through \cref{eq:entropy-Sr}. Using \cref{eq:mixture_entropy,eq:functions-zeta-R-E}, we easily obtain
\begin{subequations}\label{eq:zeta-delta-zeta}
\begin{align}
 \zeta(\Yb,\tau,e_r,e_s,{\bf e}_v) &= -\mathrm{s}\big({\bf Y}(\Yb),\tau,e_r+e_s,{\bf e}_v\big) + \varsigma(\Yb,e_r,e_s), \\ \varsigma(\Yb,e_r,e_s) &= C_{v_t}(\Yb)\ln\Big(\tfrac{\gamma-\gamma(\Yb)}{\gamma-1}\tfrac{e_r+e_s}{e_s}\big(\tfrac{\gamma(\Yb)-1}{\gamma-\gamma(\Yb)}\tfrac{e_s}{e_r}\big)^{\frac{\gamma(\Yb)-1}{\gamma-1}}\Big),
\end{align}
\end{subequations}
\noindent with partial derivatives
\begin{equation}
 \partial_\tau\zeta=-\tfrac{r(\Yb)}{\tau}, \quad \partial_{e_r}\zeta = -\tfrac{r(\Yb)}{(\gamma-1)e_r}, \quad \partial_{e_s}\zeta = \tfrac{\gamma(\Yb)-\gamma}{\gamma-1}\tfrac{C_{v_t}(\Yb)}{e_s}.
\label{eq:part_deriv_zeta}
\end{equation}

Likewise, the mapping ${\bf w}\mapsto(\Yb,\tau,e_r,e_s,{\bf e}_v)$ is surjective in $\Omega^r$, so we may rewrite $\zeta=\zeta({\bf w})$ as a function of the arguments in \cref{eq:functions-zeta-R-E-a}. This change of variables is also motivated by the following result which will be used to prove convexity of the entropy in \cref{th:entropy-convexity}.

\begin{lemma}\label{th:equiv_convexity}
 Given twice differentiable functions $f({\bf w})=\rho g(\Yb,\tau,e_r,e_s,{\bf e}_v)$, $f$ is strictly convex iff. $g$ is strictly convex in $\Omega^r$.
\end{lemma}

\begin{proof}
 Convexity being invariant under linear maps, the convexity of $f$ is equivalent to that of $f({\bf w})=f_1(\rho_1,\dots,\rho_{n_s-1},\rho,\rho{\bf v}^\top,\rho E_r,\rho {\bf e}_v, \rho e_s)$. Then, it is a classical matter that the convexity of $f_1$ and $f_2$ with $f_1(\rho,{\bf y})=\rho f_2\big(\tfrac{1}{\rho},\tfrac{1}{\rho}{\bf y}\big)$ are equivalent. Since $E_r=e_r-\tfrac{1}{2}{\bf v}\cdot{\bf v}$, the convexity of $f$ is equivalent to the convexity of $\rho f_2(\Yb,\tau,{\bf v},E_r,e_s,{\bf e}_v)=\rho g(\Yb,\tau,e_r,e_s,{\bf e}_v)$ \cite[chap. 2]{godlewski-raviart}. \qed
\end{proof}

\begin{lemma}\label{th:entropy-convexity}
 Under the assumption \cref{eq:gamma-max}, the function $\rho\zeta({\bf w})$ defined by \cref{eq:functions-zeta-R-E} is a strictly convex entropy in $\Omega^r$ for \cref{eq:relax-nrj-HRM-sys}.
\end{lemma}
\begin{proof}
This proof has been moved to \cref{sec:appendix_relax_entropy_convexity} for the sake of clarity. \qed
\end{proof}

\subsection{Properties of the relaxation system close to equilibrium}\label{sec:nrj_relax_var_pcpe}

We first prove the following variational principle which states that the equilibrium \cref{eq:maxwellian-equ-hrm} minimizes the entropy $\zeta$ and constitutes an analogue to the Gibbs Lemma in kinetic theory.

\begin{lemma}\label{th:var-principle}
 Under the assumption \cref{eq:gamma-max}, the function $\zeta$ defined by \cref{eq:functions-zeta-R-E} satisfies the following variational principle:
\begin{multline}\label{eq:var-principle}
 -\mathrm{s}({\bf Y},\tau,e_t,{\bf e}_v) =\!\!\! \min_{e_r+e_s=e_t}\!\!\big\{\zeta\big(\Yb({\bf Y}),\tau,e_r,e_s,{\bf e}_v\big): \\ 0<Y_{1\leq\alpha\leq n_s}\leq1, \tau>0, e_r>0, e_s>0,e_{1\leq\beta\leq n_d}^v>0\big\},
\end{multline}

\noindent and the minimum is reached at a unique global equilibrium which is solution to \cref{eq:maxwellian-equ-hrm}.
\end{lemma}
\begin{proof}
Note that \cref{eq:var-principle} corresponds to the minimization of a strictly convex function (see \cref{th:entropy-convexity}) in the convex set \cref{eq:relax-set-of-states} under the linear constraint $e_r+e_s=e_t$, so we only need to find a local minimum for $\zeta$ which satisfies \cref{eq:var-principle}. We further prove that $\varsigma$ in \cref{eq:zeta-delta-zeta} is positive and vanishes at equilibrium \cref{eq:maxwellian-equ-hrm} that constitutes a global minimum. Let us rewrite $\varsigma$ as $C_{v_t}(\Yb)\ln\big(f(\alpha,x)\big)$ with $f(\alpha,x)=\tfrac{(1-\alpha)(1+x)}{x}\big(\tfrac{\alpha x}{1-\alpha}\big)^\alpha$, $x=\tfrac{e_s}{e_r}>0$, and $\alpha=\tfrac{\gamma(\Yb)-1}{\gamma-1}$ in $(0,1)$ from \cref{eq:gamma-max}. We have $\partial_xf(\alpha,x)=\tfrac{1-\alpha}{x^2}(\alpha x+\alpha-1)$, thus $\partial_xf(\alpha,x)<0$ for $0<x<x_{min}:=\tfrac{1-\alpha}{\alpha}$, $\partial_xf(\alpha,x)>0$ for $x>x_{min}$, and $\partial_xf(\alpha,x_{min})=0$. Since $f(\alpha,x_{min})=1$, $\varsigma$ vanishes at the global minimum $\alpha x_{min}=1-\alpha\Leftrightarrow \tfrac{\gamma(\Yb)-1}{\gamma-1}\tfrac{e_s}{e_r}=1-\tfrac{\gamma(\Yb)-1}{\gamma-1}$ which indeed corresponds to the equilibrium \cref{eq:maxwellian-equ-hrm}: $e_s=F(\Yb,e_r)$. \qed
\end{proof}

The next result concerns the spatially homogeneous system in \cref{eq:relax-nrj-HRM-sys}:
\begin{equation}\label{eq:space-homogeneous-relax-sys}
  \partial_t{\bf w}^\epsilon = -\frac{1}{\epsilon}\big({\bf w}^\epsilon-{\cal M}({\bf w}^\epsilon)\big),
\end{equation}

\noindent and is analogue to the H-theorem for kinetic equations. The result below shows that in the limit of instantaneous relaxation $\epsilon\downarrow0$ the solution to \cref{eq:relax-nrj-HRM-sys} will converge to the equilibrium \cref{eq:maxwellian-equ-hrm}.

\begin{theorem}\label{th:H-theorem}
The vector of variables ${\bf u}=(\rhob^\top,\rho{\bf v}^\top,\rho E,\rho{\bf e}_v^\top)^\top$ with $E=E_r + e_s + h_0 + e_v$ is a constant of \cref{eq:space-homogeneous-relax-sys} and the entropy $\rho\zeta$ decreases in time and reaches a unique minimum which corresponds to the equilibrium \cref{eq:maxwellian-equ-hrm}. This equilibrium is stable in the sense of Lyapunov.
\end{theorem}

\begin{proof}
From \cref{eq:space-homogeneous-relax-sys}, we directly obtain that $\rhob$, $\rho{\bf v}$, and $\rho{\bf e}_v$ are constant so $\partial_t(\rho e_v)=0$ and $\partial_t{\bf Y}=0$. Then summing the $\rho E_r$ and $\rho e_s$ equations gives $\partial_t(\rho E_r+\rho e_s)=0$ so ${\bf u}$ in \cref{th:H-theorem} is constant.

Then, for smooth solutions of \cref{eq:space-homogeneous-relax-sys} we get
\begin{align*}
 \partial_t\zeta({\bf Y}^\epsilon,\tau^\epsilon,e_r^\epsilon,e_s^\epsilon) &= -\tau^{\epsilon^2}\partial_\tau\zeta\partial_t\rho^\epsilon + \partial_{e_r}\zeta\partial_te_r^\epsilon + \partial_{e_s}\zeta\partial_te_s^\epsilon \\
 &=  \frac{1}{\epsilon}\big(e_s^\epsilon-F(\Yb^\epsilon,e_r^\epsilon)\big)(\partial_{e_r}\zeta-\partial_{e_s}\zeta)^\epsilon \\
  & \overset{\cref{eq:part_deriv_zeta}}{=} -\frac{1}{\epsilon}\frac{r(\Yb^\epsilon)}{(\gamma-1)e_r^\epsilon e_s^\epsilon}\big(e_s^\epsilon-F(\Yb^\epsilon,e_r^\epsilon)\big)^2 \leq 0,
\end{align*}
\noindent so $\partial_t\zeta\leq0$ and $\partial_t\zeta=0$ iff. $e_s^\epsilon=F(\Yb^\epsilon,e_r^\epsilon)$ which corresponds to the equilibrium \cref{eq:maxwellian-equ-hrm} which in turn corresponds to the global minimum of $\zeta$ from \cref{th:var-principle}. We therefore conclude that the system is stable by applying the Lyapunov stability criterion with the Lyapunov function ${\bf w}\mapsto\zeta({\bf w}+{\bf w}^0)-\zeta({\bf w}^0)$ where ${\bf w}^0={\cal P}({\bf u})$ corresponds to the equilibrium \cref{eq:maxwellian-equ-hrm} with ${\cal P}$ defined in \cref{eq:relax-L-P-operators}. Finally note that the partial energies are given explicitly by $e_\alpha^t=C_{v_\alpha}^te_t/C_{v_t}({\bf Y})$ which confirms that the equilibrium corresponds to a unique state. \qed
\end{proof}

The last result describes the first-order asymptotic analysis of small perturbations in the relaxation process in the neighborhood of the equilibrium \cref{eq:maxwellian-equ-hrm} by performing a formal Chapman-Enskog expansion \cite{chen_levermore_liu94}. This result extends \cite[Prop.~2.4]{coquel_perthame_98} to multicomponent flows and allows to understand the relaxation process close to equilibrium as a viscous perturbation to \cref{eq:HRM_PDEs-a} and to prove well-posedness of \cref{eq:relax-nrj-HRM-sys} and consistency with \cref{eq:HRM_PDEs-a} when $\epsilon\downarrow0$.

\begin{theorem}\label{th:chapman-enskog}
In the limit $\epsilon\downarrow0$, small perturbations to the equilibrium \cref{eq:maxwellian-equ-hrm} obey the following first order asymptotic expansion in $\epsilon$:
\begin{equation*}
 \partial_t{\bf u}^\epsilon + \nabla\cdot{\bf f}({\bf u}^\epsilon) - \nabla\cdot{\bf f}_v({\bf u}^\epsilon,\nabla{\bf u}^\epsilon)=0, \quad {\bf f}_v({\bf u}^\epsilon,\nabla{\bf u}^\epsilon) = \mu({\bf Y}^\epsilon,\rho^\epsilon,e_t^\epsilon) \begin{pmatrix} 0 \\ (\nabla\cdot{\bf v}^\epsilon){\bf I}_d \\ (\nabla\cdot{\bf v}^\epsilon){\bf v}^{\epsilon\top} \\ 0 \end{pmatrix},
\end{equation*}

\noindent with $\mu({\bf Y}^\epsilon,\rho^\epsilon,e_t^\epsilon) = \epsilon\tfrac{(\gamma-\gamma({\bf Y}^\epsilon))(\gamma({\bf Y}^\epsilon)-1)^2}{\gamma-1}\rho^\epsilon e_t^\epsilon$ positive under \cref{eq:gamma-max}.
\end{theorem}

\begin{proof}
Let consider perturbations to the equilibrium expanded in the form
\begin{equation}\label{eq:chapman-enskog-expansion}
  e_r^\epsilon = e_r^0 + \epsilon e_r^1 + \epsilon^2 e_r^2 + \dots, \quad e_s^\epsilon = e_s^0 + \epsilon e_s^1 + \epsilon^2 e_s^2 + \dots.
\end{equation}

First observe that smooth solutions to \cref{eq:relax-nrj-HRM-sys} with \cref{eq:EOS-relax} satisfy
\begin{subequations}
  \begin{align}
  \partial_t{\bf Y}^\epsilon + {\bf v}^\epsilon\cdot\nabla{\bf Y}^\epsilon &= 0, \\
  \partial_te_r^\epsilon + {\bf v}^\epsilon\cdot\nabla e_r^\epsilon + (\gamma-1)e_r^\epsilon\nabla\cdot{\bf v}^\epsilon &= \frac{e_s^\epsilon- F({\bf Y}^\epsilon,e_r^\epsilon)}{\epsilon}, \label{eq:relax-nrj-HRM-sys-smooth-sol-b}\\
  \partial_te_s^\epsilon + {\bf v}^\epsilon\cdot\nabla e_s^\epsilon &= -\frac{e_s^\epsilon- F({\bf Y}^\epsilon,e_r^\epsilon)}{\epsilon}, \label{eq:relax-nrj-HRM-sys-smooth-sol-c}
\end{align}
\end{subequations}

\noindent from which we deduce
\begin{equation} \label{eq:relax-nrj-HRM-sys-smooth-sol-d}
  \partial_tF({\bf Y}^\epsilon,e_r^\epsilon)+ {\bf v}^\epsilon\cdot\nabla F({\bf Y}^\epsilon,e_r^\epsilon) = \frac{\gamma-\gamma({\bf Y}^\epsilon)}{\gamma({\bf Y}^\epsilon)-1}\Big(\frac{e_s^\epsilon-F({\bf Y}^\epsilon,e_r^\epsilon)}{\epsilon}-(\gamma-1)e_r^\epsilon\nabla\cdot{\bf v}^\epsilon\Big).
\end{equation}

\noindent Plugging \cref{eq:chapman-enskog-expansion} into either \cref{eq:relax-nrj-HRM-sys-smooth-sol-b}, or \cref{eq:relax-nrj-HRM-sys-smooth-sol-c}, the order ${\cal O}(\epsilon^{-1})$ imposes
\begin{equation*}
 e_s^0 = F({\bf Y}^\epsilon,e_r^0) = \frac{\gamma-\gamma({\bf Y}^\epsilon)}{\gamma({\bf Y}^\epsilon)-1}e_r^0, 
\end{equation*}

\noindent while the constraint $e_t^\epsilon=e_r^\epsilon+e_s^\epsilon=e_r^0+F({\bf Y}^\epsilon,e_r^0)$ in \cref{eq:var-principle} gives
\begin{equation*}
 e_r^0 = \frac{\gamma({\bf Y}^\epsilon)-1}{\gamma-1}e_t^\epsilon, \quad e_s^0 = \frac{\gamma-\gamma({\bf Y}^\epsilon)}{\gamma-1}e_t^\epsilon, \quad e_r^k+e_s^k=0 \quad \forall k\geq 1.
\end{equation*}

Plugging again \cref{eq:chapman-enskog-expansion} into \cref{eq:relax-nrj-HRM-sys-smooth-sol-c,eq:relax-nrj-HRM-sys-smooth-sol-d}, we obtain at leading order
\begin{align*}
  \partial_te_s^0  + {\bf v}^\epsilon\cdot\nabla e_s^0 &= -\big(e_s^1-F({\bf Y}^\epsilon,e_r^1)\big), \\
  \partial_tF({\bf Y}^\epsilon,e_r^0)+ {\bf v}^\epsilon\cdot\nabla F({\bf Y}^\epsilon,e_r^0) &= \frac{\gamma-\gamma({\bf Y}^\epsilon)}{\gamma({\bf Y}^\epsilon)-1}\Big(e_s^1-F({\bf Y}^\epsilon,e_r^1)-(\gamma-1)e_r^0\nabla\cdot{\bf v}^\epsilon\Big),
\end{align*}

\noindent and since $e_s^0=F({\bf Y}^\epsilon,e_r^0)$, we get
\begin{equation*}
  -\big(e_s^1-F({\bf Y}^\epsilon,e_r^1)\big) = \frac{\gamma-\gamma({\bf Y}^\epsilon)}{\gamma({\bf Y}^\epsilon)-1}\Big(e_s^1-F({\bf Y}^\epsilon,e_r^1)-(\gamma-1)e_r^0\nabla\cdot{\bf v}^\epsilon\Big),
\end{equation*}

\noindent and using the above expressions for $e_r^0$ and $e_s^0$ gives
\begin{equation*}
  e_r^1=-e_s^1=-\big(\gamma-\gamma({\bf Y}^\epsilon)\big)\Big(\frac{\gamma({\bf Y}^\epsilon)-1}{\gamma-1}\Big)^2e_t^\epsilon\nabla\cdot{\bf v}^\epsilon.
\end{equation*}

Finally, in \cref{eq:relax-nrj-HRM-sys} consider the momentum equation and add up the equations for $\rho E_r$, $\rho e_s$, $\rho e_v$ together with an equation for $\rho h_0=\sum_\alpha\rho_\alpha h_\alpha^0$. We then obtain up to order ${\cal O}(\epsilon)$
\begin{align*}
  \partial_t\rho^\epsilon {\bf v}^\epsilon  + \nabla\cdot\big(\rho^\epsilon{\bf v}^\epsilon{\bf v}^{\epsilon\top} + p_r(\rho^\epsilon,e_r^0+\epsilon e_r^1) \big) &= 0, \\
  \partial_t\rho^\epsilon E^\epsilon  + \nabla\cdot\Big(\big(\rho^\epsilon E^\epsilon + p_r(\rho^\epsilon,e_r^0+\epsilon e_r^1) \big){\bf v}^\epsilon\Big) &= 0, 
\end{align*}

\noindent and we conclude by observing that $p_r(\rho^\epsilon,e_r^0)=\p({\bf Y}^\epsilon,\rho^\epsilon,e_t^\epsilon)$ from \cref{eq:maxwellian-equ-hrm-press} and by using the expression for $e_r^1$. \qed
\end{proof}

\subsection{General framework for the design of three-point schemes}\label{sec:FVrelax2FVEuler_ms}

We now clarify the form of the numerical flux for \cref{eq:HRM_PDEs-a} that we deduce from a
numerical flux for \cref{eq:relax-nrj-HRM-sys} in homogeneous form. The former flux will satisfy the properties (i) to (vi) in \cref{sec:FV-schemes} providing that the latter satisfies similar properties. The three-point scheme for \cref{eq:relax-nrj-HRM-sys}$_{\epsilon\rightarrow\infty}$ reads
\begin{equation}\label{eq:3pt-scheme-relax}
 {\bf W}_j^{n+1}-{\bf W}_j^{n} + \tfrac{\Delta t^{(n)}}{\Delta x}\big({\bf H}({\bf W}_{j}^{n},{\bf W}_{j+1}^{n},{\bf n})-{\bf H}({\bf W}_{j-1}^{n},{\bf W}_{j}^{n},{\bf n})\big) = 0, 
\end{equation}
\noindent with ${\bf H}({\bf w},{\bf w},{\bf n})={\bf g}({\bf w})\cdot{\bf n}$. We assume that under some CFL condition on the time step (see \cref{sec:ex-ES-3pt-schemes}), \cref{eq:3pt-scheme-relax} enjoys the properties (i) to (vi) in \cref{sec:FV-schemes}. In particular we have
\begin{equation}\label{eq:3pt-scheme-equal-relax}
 \rho\zeta({\bf W}_j^{n+1}) \leq \rho\zeta({\bf W}_j^{n}) - \tfrac{\Delta t^{(n)}}{\Delta x}\big(Z({\bf W}_{j}^{n},{\bf W}_{j+1}^{n},{\bf n})-Z({\bf W}_{j-1}^{n},{\bf W}_{j}^{n},{\bf n})\big), 
\end{equation}
\noindent with $Z({\bf w},{\bf w},{\bf n})=\rho\zeta{\bf v}\cdot{\bf n}$. Then, from \cref{eq:3pt-scheme-relax} we may design a scheme for \cref{eq:HRM_PDEs-a} as stated in the theorem below.

\begin{theorem}\label{th:ES-3-pt-scheme}
  Consider the three-point numerical scheme \cref{eq:3pt-scheme-relax} for \cref{eq:relax-nrj-HRM-sys}$_{\epsilon\rightarrow\infty}$, i.e., $\epsilon\rightarrow\infty$ in \cref{eq:relax-nrj-HRM-sys},  with Lipschitz, consistent and conservative numerical flux. Assume that it satisfies \cref{eq:3pt-scheme-equal-relax} with a consistent numerical flux, some maximum principle on the mass fractions
\begin{equation}\label{eq:relax-mass-minmax}
 \min(Y_{\alpha_{j-1}}^{n},Y_{\alpha_{j}}^{n},Y_{\alpha_{j+1}}^{n})\leq Y_{\alpha_j}^{n+1}\leq\max(Y_{\alpha_{j-1}}^{n},Y_{\alpha_{j}}^{n},Y_{\alpha_{j+1}}^{n}) \quad \forall 1\leq \alpha\leq n_s,
\end{equation}

\noindent and the specific entropy
\begin{equation}\label{eq:relax-entropy-minmax}
 \zeta({\bf W}_j^{n+1})\leq\max\big(\zeta({\bf W}_{j-1}^{n}),\zeta({\bf W}_j^{n}),\zeta({\bf W}_{j+1}^{n})\big),
\end{equation}

\noindent and is robust, ${\bf W}_{j\in\mathbb{Z}}^{n\geq0}\in\Omega^r$. If \cref{eq:gamma-max} holds, the three-point numerical scheme \cref{eq:3pt-scheme-a} with the Lipschitz, consistent and conservative numerical flux 
 \begin{align}\label{eq:flux_h_from_H}
   {\bf h}({\bf u}^-,{\bf u}^+,{\bf n}) &= {\cal L}{\bf H}\big({\cal P}({\bf u}^-),{\cal P}({\bf u}^+),{\bf n}\big), \\ 
   h_X({\bf u}^-,{\bf u}^+,{\bf n}) &= H_X\big({\cal P}({\bf u}^-),{\cal P}({\bf u}^+),{\bf n}\big), \quad X\in\{\rhob,\rho{\bf v},\rho{\bf e}_v\}, \nonumber \\
   h_{\rho E}({\bf u}^-,{\bf u}^+,{\bf n}) &= H_{\rho E_r}({\cal P}({\bf u}^-),{\cal P}({\bf u}^+),{\bf n}\big)+H_{\rho e_s}({\cal P}({\bf u}^-),{\cal P}({\bf u}^+),{\bf n}\big) \nonumber \\ 
	& +\sum_{\alpha=1}^{n_s}h_\alpha^0H_{\rho_\alpha}({\cal P}({\bf u}^-),{\cal P}({\bf u}^+),{\bf n}\big) +\sum_{\beta=1}^{n_d}H_{\rho e_\beta^v}({\cal P}({\bf u}^-),{\cal P}({\bf u}^+),{\bf n}\big),  \nonumber
\end{align}
\noindent with ${\cal L}$ defined  in \cref{eq:relax-L-P-operators}, is ES for the pair $(\eta,{\bf q})$ in \cref{eq:entropy_ineq_cont}, satisfies \cref{eq:3pt-scheme-ineq} with $Q({\bf u}^-,{\bf u}^+,{\bf n})=Z\big({\cal P}({\bf u}^-),{\cal P}({\bf u}^+),{\bf n}\big)$, the minimum and maximum principles \cref{eq:3pt-scheme-mass-minmax,eq:3pt-scheme-entropy-minmax}, and is robust, ${\bf U}_{j\in\mathbb{Z}}^{n\geq0}\in\Omega^a$.
\end{theorem}

\begin{proof}
  By consistency of ${\bf H}$: ${\bf h}({\bf u},{\bf u},{\bf n})\overset{\cref{eq:flux_h_from_H}}{=}{\cal L}{\bf H}\big({\cal P}({\bf u}),{\cal P}({\bf u}),{\bf n}\big)={\cal L}{\bf g}({\cal P}({\bf u}))\cdot{\bf n}\overset{\cref{eq:maxwellian-equ}}{=}{\bf f}({\bf u})\cdot{\bf n}$, while Lipschitz continuity and conservation are direct since ${\cal L}$ is linear.

Then, let ${\bf W}_{j}^{n}={\cal P}({\bf U}_{j}^{n})$ so $\rho\zeta({\bf W}_{j}^{n})=\eta({\bf U}_{j}^{n})$ and $Z({\bf W}_{j}^{n},{\bf W}_{j+1}^{n},{\bf n})=Q({\bf U}_{j}^{n},{\bf U}_{j+1}^{n},{\bf n})$, and define ${\bf W}_{j}^{n+1}$ from \cref{eq:3pt-scheme-relax} and ${\bf U}_{j}^{n+1}={\cal L}{\bf W}_{j}^{n+1}$. If ${\bf U}_{j}^{n}\in\Omega^a$, then ${\bf W}_{j}^{n}={\cal P}({\bf U}_{j}^{n})\in\Omega^r$ since $e_r=\tfrac{\gamma({\bf Y})-1}{\gamma-1}e_t$ and $e_s=\tfrac{\gamma-\gamma({\bf Y})}{\gamma-1}e_t$, and ${\bf W}_{j}^{n+1}\in\Omega^r$ so ${\bf U}_{j}^{n+1}={\cal L}{\bf W}_{j}^{n+1}\in\Omega^a$ by \cref{eq:maxwellian-equ}. Now, by the variational principle \cref{eq:var-principle} we have 
%
\begin{eqnarray*}
\eta({\bf U}_j^{n+1}) &\overset{\cref{eq:var-principle}}{\leq}& \rho\zeta({\bf W}_j^{n+1}) \\
 &\overset{\cref{eq:3pt-scheme-equal-relax}}{\leq}& \rho\zeta({\bf W}_j^{n}) - \tfrac{\Delta t^{(n)}}{\Delta x}\big(Z({\bf W}_{j}^{n},{\bf W}_{j+1}^{n},{\bf n})-Z({\bf W}_{j-1}^{n},{\bf W}_{j}^{n},{\bf n})\big) \\
 &=& \rho\zeta\big({\cal P}({\bf U}_j^{n})\big) - \tfrac{\Delta t^{(n)}}{\Delta x}\Big(Z\big({\cal P}({\bf U}_{j}^{n}),{\cal P}({\bf U}_{j+1}^{n}),{\bf n}\big)-Z\big({\cal P}({\bf U}_{j-1}^{n}),{\cal P}({\bf U}_{j}^{n}),{\bf n}\big)\Big) \\
 &=& \eta({\bf U}_j^{n}) - \tfrac{\Delta t^{(n)}}{\Delta x}\big(Q({\bf U}_{j}^{n},{\bf U}_{j+1}^{n},{\bf n})-Q({\bf U}_{j-1}^{n},{\bf U}_{j}^{n},{\bf n})\big).
\end{eqnarray*}

Finally, pluging $\zeta({\bf W}_{j}^{n})=-\s({\bf U}_{j}^{n})$ into \cref{eq:relax-entropy-minmax} for all $j$ and using \cref{eq:var-principle} we obtain \cref{eq:3pt-scheme-entropy-minmax}, while \cref{eq:3pt-scheme-mass-minmax} holds because \cref{eq:relax-mass-minmax} and the components associated to $\rhob$ in \cref{eq:flux_h_from_H} remain unaffected. \qed
\end{proof}

Since the pressure in \cref{eq:relax-nrj-HRM-sys} obeys a polytropic ideal gas EOS and the variables $({\bf Y},{\bf e}_v,e_s)$ are purely advected, one may use many methods for \cref{eq:3pt-scheme-relax} such as, e.g., the Godunov \cite{godunov_59}, Rusanov \cite{Rusanov1961}, HLL \cite{hll_83}, or Roe \cite{Roe_1981} schemes, though the latter method does not guaranty robustness \cite{einfeldt_etal_91}. In the next section we will consider some of these schemes.

In the definition of the numerical flux \cref{eq:flux_h_from_H}, the ${\cal L}$ operator consists in adding up some components of ${\bf H}$ to build the numerical flux for the total energy, $\rho E$, while the ${\cal P}$ operators consist in taking data at equilibrium. This last operation is equivalent to applying instantaneous relaxation, i.e., to consider \cref{eq:relax-nrj-HRM-sys}$_{\epsilon\rightarrow\infty}$, through a splitting of hyperbolic and relaxation operators \cite{coquel_perthame_98}. Note that instantaneous relaxation is here justified by the analysis in \cref{sec:nrj_relax_var_pcpe}. This approach is also in agreement with the numerical flux we will consider that uses discrete projections onto Maxwellian equibria \cite{bouchut_04}.

\begin{remark}
 We note that \cref{th:ES-3-pt-scheme} may be directly applied to the multidimensional schemes \cref{eq:2D-FV-scheme} instead of the three-point scheme \cref{eq:3pt-scheme-a}. This may allow to use gueninely multi-dimensional schemes possibly with a less restrictive constraint on the time step. In contrast considering \cref{eq:3pt-scheme-a} with the CFL condition \cref{eq:CFL-positive-2D-scheme} would allow to encompass more general schemes such as the ARS we will consider in \cref{sec:ex-ES-3pt-schemes}.
\end{remark}

%
%
\section{Examples of three-point schemes}\label{sec:ex-ES-3pt-schemes}

In this section we consider examples of three-point schemes \cref{eq:3pt-scheme-relax} for the homogeneous energy relaxation system \cref{eq:relax-nrj-HRM-sys}$_{\epsilon\rightarrow\infty}$ to illustrate \cref{th:ES-3-pt-scheme}. Such schemes define schemes \cref{eq:3pt-scheme-a} for \cref{eq:HRM_PDEs-a} through \cref{eq:flux_h_from_H}. As already noticed, other numerical schemes may be used since we use a simple polytropic EOS in \cref{eq:relax-nrj-HRM-sys}. 

The schemes we consider use Riemann type solvers with numerical flux in \cref{eq:3pt-scheme-relax} of the form
\begin{equation}\label{eq:relax_flux}
 {\bf H}({\bf w}^-,{\bf w}^+,{\bf n}) = {\bf g}\big(\bcW(0;{\bf w}^-,{\bf w}^+,{\bf n})\big)\cdot{\bf n},
\end{equation}

\noindent where $\bcW(\cdot;{\bf w}_L,{\bf w}_R,{\bf n})$ is used to approximate the solution to the Riemann problem \cref{eq:relax-nrj-HRM-sys}$_{\epsilon\rightarrow\infty}$ with initial data ${\bf w}_0(x) = {\bf w}_L$ if $x:={\bf x}\cdot{\bf n}<0$ and ${\bf w}_0(x) = {\bf w}_R$ if ${\bf x}\cdot{\bf n}>0$. We then build three-point schemes for \cref{eq:HRM_PDEs-a} by simply applying \cref{eq:flux_h_from_H}.

\subsection{The Godunov method}\label{sec:Godunov}

As noticed in \cite{coquel_perthame_98} it is possible to apply the exact Riemann solver \cite{godunov_59} for polytropic gas to \cref{eq:3pt-scheme-relax} where $\bcW$ corresponds to the exact entropy weak solution to the Riemann problem. Consider the compressible Euler equations
\begin{equation}\label{eq:Euler_eqns}
 \partial_t\tilde{\bf w} + \nabla\cdot\tilde{\bf g}(\tilde{\bf w}) = 0, \quad \tilde{\bf w}=\begin{pmatrix}\rho \\ \rho{\bf v} \\ \rho E_r\end{pmatrix}, \quad \tilde{\bf g}(\tilde{\bf w})=\begin{pmatrix}\rho{\bf v}^\top \\ \rho{\bf v}{\bf v}^\top+\mathrm{p}_r{\bf I} \\ (\rho E_r+\mathrm{p}_r){\bf v}^\top\end{pmatrix},
\end{equation}

\noindent with $\mathrm{p}_r$ defined from \cref{eq:EOS-relax} and $\gamma$ satisfying \cref{eq:gamma-max}.

Any variable $\psi$ in $\{{\bf Y},{\bf e}_v,e_s\}$ is uncoupled from the $\tilde{\bf w}$ variables and is only purely transoprted in \cref{eq:relax-nrj-HRM-sys}$_{\epsilon\rightarrow\infty}$. Noting that the intermediate states are $(\psi_L,\psi_L,\psi_R,\psi_R)$, the entropy weak solution is made of the Riemann solution for the Euler equations with variables $\tilde{\bf w}$ and fluxes $\tilde{\bf g}(\tilde{\bf w})$ plus the states for $\psi$ \cite[Lemma~4.6]{coquel_perthame_98}. The Godunov method is thus ES and guaranties robustness of \cref{eq:3pt-scheme-relax} as well as the minimum and maximum principles \cref{eq:relax-mass-minmax,eq:relax-entropy-minmax} under some standard CFL condition.

Let $\rho$, $u={\bf v}\cdot{\bf n}$ and $\p_r$ be the solution to the Riemann problem for \cref{eq:Euler_eqns} with initial data $\tilde{\bf w}_0(x) = \big(\rho_L,u_L,\p({\bf Y}_L,\rho_L,e_{t_L})\big)^\top$ if $x:={\bf x}\cdot{\bf n}<0$ and $\tilde{\bf w}_0(x) = \big(\rho_R,u_R,\p({\bf Y}_R,\rho_R,e_{t_R})\big)^\top$ if ${\bf x}\cdot{\bf n}>0$ and let $u^\star$ be the velocity in the star region. Note that $\p_{r_X}=\p_X$ given by \cref{eq:mixture_eos} for $X=L,R$ since data are taken at equilibrium in \cref{eq:flux_h_from_H}. Then the numerical flux for \cref{eq:3pt-scheme-a} reads
\begin{equation}\label{eq:godunov_flux}
 {\bf h}^{God}({\bf u}_L,{\bf u}_R,{\bf n}) = \begin{pmatrix} \rho(\epsilon_L{\bf Y}_L + \epsilon_R{\bf Y}_R) u \\ \rho u(u{\bf n}+\epsilon_L{\bf v}_L^\perp+\epsilon_R{\bf v}_R^\perp) + \p_r{\bf n} \\ \rho\big(\epsilon_LE_L^\star+ \epsilon_RE_R^\star\big)u + p_ru \\ \rho(\epsilon_Le_{v_L}+\epsilon_Re_{v_R})u \end{pmatrix},
\end{equation}

\noindent where $E_X^\star=h_0({\bf Y}_X)+e_r+\tfrac{\gamma-\gamma({\bf Y}_X)}{\gamma({\bf Y}_X)-1}e_{r}+e_{v_X}+\tfrac{{\bf v}_X^\perp\cdot{\bf v}_X^\perp+u^2}{2}$, $e_r=\tfrac{\p_r}{(\gamma-1)\rho}$, ${\bf v}_X^\perp={\bf v}_X-({\bf v}_X\cdot{\bf n}){\bf n}$, $X=L,R$, $\epsilon_L=1$ if $u^\star>0$ and $0$ else, and $\epsilon_R=1-\epsilon_L$. Finally, the condition on the time step is \cref{eq:3pt-scheme-CFL} with $\lambda({\bf u})=|{\bf v}\cdot{\bf n}|+c_\gamma(\rho,\p)$ with $c_\gamma(\rho,\p)=\sqrt{\gamma\p/\rho}$ and $\p$ defined from \cref{eq:mixture_eos}.

\subsection{The HLL numerical flux}\label{sec:HLL_flux}

The HLL approximate Riemann solver \cite{hll_83} for \cref{eq:3pt-scheme-relax} reads
\begin{equation}\label{eq:HLL_ARS-xt}
 \bcW^{hll}({\bf w}_L,{\bf w}_R,{\bf n}) = \left\{ \begin{array}{ll}  {\bf w}_L, & \tfrac{x}{t} < S_L, \\
 \frac{S_R{\bf w}_R-S_L{\bf w}_L+{\bf g}({\bf w}_L)-{\bf g}({\bf w}_R)}{S_R-S_L}, & S_L<\tfrac{x}{t}<S_R,\\
 {\bf w}_R, & S_R < \tfrac{x}{t}, \end{array} \right.
\end{equation}

\noindent and is ES \cite{hll_83} and robust \cite{einfeldt_etal_91} under the CFL condition \cref{eq:3pt-scheme-CFL}, with $\lambda=\max(|S_L|,|S_R|)$, providing that $S_L$ (resp. $S_R$) is a lower (resp. upper) bound of the speed of the leftmost (resp. rightmost) wave in the exact Riemann solution. Applying \cref{eq:flux_h_from_H}, the numerical flux for \cref{eq:3pt-scheme-a} reads
\begin{equation}\label{eq:hll_flux}
 {\bf h}^{hll}({\bf u}_L,{\bf u}_R,{\bf n}) = \left\{ \begin{array}{ll}  {\bf f}({\bf u}_L)\cdot{\bf n}, & \tfrac{x}{t} < S_L, \\
 \frac{S_R{\bf f}({\bf u}_L)\cdot{\bf n}-S_L{\bf f}({\bf u}_R)\cdot{\bf n}+S_LS_R({\bf u}_R-{\bf u}_L)}{S_R-S_L}, & S_L<\tfrac{x}{t}<S_R,\\
 {\bf f}({\bf u}_R)\cdot{\bf n}, & S_R < \tfrac{x}{t}, \end{array} \right.
\end{equation}

\noindent and we evaluate the wave speeds from the two-rarefaction approximation \cite[Ch.~9]{toro_book}:
\begin{equation*}
S_L={\bf v}_L\cdot{\bf n}-c_L, \quad S_R={\bf v}_R\cdot{\bf n}+c_R, \quad c_X=c_\gamma(\rho_X,\p_X)\sqrt{1+\frac{\gamma+1}{2\gamma}\Big(\frac{\p_{tr}^\star}{\p_X}-1\Big)^+},
\end{equation*}

\noindent where $(\cdot)^+=\max(\cdot,0)$ denotes the positive part, $c_\gamma(\rho,\p)=\sqrt{\gamma\p/\rho}$, $\p_X=\p({\bf Y}_X,\rho_X,e_{t_X})$ given by \cref{eq:mixture_eos}, $\gamma$ satisfying \cref{eq:gamma-max}, and
\begin{equation*}
\p_{tr}^\star = \Bigg(\frac{c_\gamma(\rho_L,\p_L) + c_\gamma(\rho_R,\p_R) + \frac{\gamma-1}{2}({\bf v}_L-{\bf v}_R)\cdot{\bf n}}{c_\gamma(\rho_L,\p_L)\p_L^{-\frac{\gamma-1}{2\gamma}}+c_\gamma(\rho_R,\p_R)\p_R^{-\frac{\gamma-1}{2\gamma}}}\Bigg)^{\tfrac{2\gamma}{\gamma-1}}.
\end{equation*}

Note that the two-rarefaction approximation holds for the compressible Euler equations with a polytropic EOS for an adiabatic exponent $1<\gamma\leq\tfrac{5}{3}$ \cite{guermond_popov_16}. The strict inequality in \cref{eq:gamma-max} may thus prevent the bound estimates with this approach. However, the analysis in \cite[Lemma~4.2]{guermond_popov_16} shows that this may occur only for moderate shock strengths so the scheme remains ES for strong shocks as expected in practice. For instance, we use $\gamma=1.01\times\tfrac{5}{3}$ in the numerical experiments of \cref{sec:num_xp} for which the above estimates are valid when either $\p_{tr}^\star\leq\p_X$, or $\p_{tr}^\star\gtrsim 1.05\p_X$.



%
\subsection{Pressure relaxation-based numerical flux}\label{sec:bouchut_flux}

We now consider the numerical flux based on relaxation of pressure \cite[Prop.~2.21]{bouchut_04}. The approximate Riemann solver for \cref{eq:3pt-scheme-relax} reads
\begin{equation}\label{eq:bouchut-relax-ARS-xt}
 \bcW^{r}(\tfrac{x}{t};{\bf w}_L,{\bf w}_R,{\bf n}) = \left\{ \begin{array}{ll}  {\bf w}_L, & \tfrac{x}{t} < S_L, \\
 {\bf w}_L^\star, & S_L<\tfrac{x}{t}<u^\star,\\
 {\bf w}_R^\star, & u^\star<\tfrac{x}{t}<S_R,\\
 {\bf w}_R, & S_R < \tfrac{x}{t}, \end{array} \right.
\end{equation}
\noindent where ${\bf w}_X^\star=(\rho_X^\star {\bf Y}_X^\top,\rho_X^\star {\bf v}_X^{\star\top},\rho_X^\star E_{r,X}^\star, \rho_X^\star {\bf e}_{v,X}^\top, \rho_X^\star e_{s,X})^\top$, for $X=L,R$, and
\begin{subequations}\label{eq:sol-PR-relax}
\begin{align}
  {\bf v}_L^\star &= {\bf v}_L^\perp+u^\star{\bf n}, \quad {\bf v}_R^\star = {\bf v}_R^\perp+u^\star{\bf n},  \label{eq:sol-PR-relax-a} \\
  u^\star  &= \frac{a_Lu_L+a_Ru_R+\mathrm{p}_{r,L}-\mathrm{p}_{r,R}}{a_L+a_R}, \quad \mathrm{p}^\star = \frac{a_R\mathrm{p}_{r,L}+a_L\mathrm{p}_{r,R}+a_La_R(u_L-u_R)}{a_L+a_R},  \\
  \frac{1}{\rho_L^\star} &= \frac{1}{\rho_L} + \frac{u^\star-u_L}{a_L},\quad \frac{1}{\rho_R^\star} = \frac{1}{\rho_R} + \frac{u_R-u^\star}{a_R}, \\
  E_{r,L}^\star &= E_{r,L} - \frac{\mathrm{p}^\star u^\star - \mathrm{p}_{r,L}u_L}{a_L},\quad E_{r,R}^\star = E_{r,R} - \frac{\mathrm{p}_{r,R}u_R-\mathrm{p}^\star u^\star}{a_R},
\end{align}
\end{subequations}
\noindent where ${\bf v}_X^\perp={\bf v}_X-u_X{\bf n}$, $u_X={\bf v}_X\cdot{\bf n}$, $\rho_X=\sum_{\alpha=1}^{n_s}\rho_{\alpha_X}$, ${\bf Y}_X=\tfrac{1}{\rho_X}\rhob_X$, and $\mathrm{p}_{r,X}=\mathrm{p}_r(\rho_X,e_{r,X})$ defined by \cref{eq:EOS-relax}.

The wave speeds in \cref{eq:bouchut-relax-ARS-xt} are evaluated from $S_L=u_L-a_L/\rho_L$ and $S_R=u_R+a_R/\rho_R$ where the approximate Lagrangian sound speeds \cite{bouchut_04} are defined by
\begin{subequations}\label{eq:wave-estimate}
\begin{align}
 \left\{ \begin{array}{rcl} \tfrac{a_L}{\rho_L} &=& c_\gamma(\rho_L,\mathrm{p}_{r,L}) + \tfrac{\gamma+1}{2}\Big(\tfrac{\mathrm{p}_{r,R}-\mathrm{p}_{r,L}}{\rho_Rc_\gamma(\rho_R,\mathrm{p}_{r,R})}+u_L-u_R\Big)^+ \\
\tfrac{a_R}{\rho_R} &=& c_\gamma(\rho_R,\mathrm{p}_{r,R}) + \tfrac{\gamma+1}{2}\Big(\tfrac{\mathrm{p}_{r,L}-\mathrm{p}_{r,R}}{a_L}+u_L-u_R\Big)^+
\end{array}\right., & \quad \mbox{if } \mathrm{p}_{r,R} \geq \mathrm{p}_{r,L}, \\
 \left\{ \begin{array}{rcl} \tfrac{a_R}{\rho_R} &=& c_\gamma(\rho_R,\mathrm{p}_{r,R}) + \tfrac{\gamma+1}{2}\Big(\tfrac{\mathrm{p}_{r,L}-\mathrm{p}_{r,R}}{\rho_Lc_\gamma(\rho_L,\mathrm{p}_{r,L})}+u_L-u_R\Big)^+ \\
\tfrac{a_L}{\rho_L} &=& c_\gamma(\rho_L,\mathrm{p}_{r,L}) + \tfrac{\gamma+1}{2}\Big(\tfrac{\mathrm{p}_{r,R}-\mathrm{p}_{r,L}}{a_R}+u_L-u_R\Big)^+
\end{array}\right., & \quad \mbox{else,}  
\end{align}
\end{subequations}

\noindent with $\gamma$ defined from \cref{eq:gamma-max}.

This numerical scheme is based on a relaxation approximation using evolution equations for a relaxation pressure in place of $\mathrm{p}_r$ and for $a$ in \cref{eq:sol-PR-relax} in place of the Lagrangian sound speed $\rho c_\gamma(\rho,\mathrm{p}_r)$. The Riemann solution contains only linearly degenerate fields and \cref{eq:bouchut-relax-ARS-xt} follows from projection of the initial data onto an equilibrium manifold. We refer to \cite[Sec.~ 2.4]{bouchut_04} or \cite{coquel_etal_relax_fl_sys_12} for complete introductions and in-depth analyses. In particular, the analysis in \cite[Sec.~ 2.4]{bouchut_04} proves the ES, robustness and the minimum principle on entropy by reversing the roles of energy conservation and entropy inequality \cite{coquel_etal_01}. This technique also applies to the entropy $\rho\zeta({\bf w})$ and we may consider $\rho E_r=\rho e_r({\bf Y},\tau,\zeta,e_s,{\bf e}_v) + \tfrac{1}{2}{\bf v}\cdot{\bf v}$ as an entropy for the system defined by conservation laws for $(\rhob^\top,\rho{\bf v}^\top,\rho\zeta,\rho {\bf e}_v,\rho e_s)^\top$. Indeed, the convexity of $E_r({\bf Y},\tau,\zeta,e_s,{\bf e}_v)$ is equivalent to the convexity of $\zeta({\bf Y},\tau,e_r,e_s,{\bf e}_v)$ since from \cref{eq:zeta-delta-zeta} $\partial_{e_r}\zeta=-\tfrac{r({\bf Y})}{(\gamma-1)e_r}<0$ \cite[chap. 2]{godlewski-raviart}. 

The Bouchut scheme guaranties positivity of $\rho$ and $e_r$ under the CFL condition \cref{eq:3pt-scheme-CFL} with $\lambda = \max(|S_L|,|S_R|)$. Positivity of $\rhob=\rho{\bf Y}$, $e_{s}$ and $e_{v}$ then follows by cell-averaging the Riemann solution \cref{eq:sol-PR-relax}. Likewise, the discrete minimum maximum principle \cref{eq:relax-mass-minmax} holds for the same reason. Applying \cref{eq:flux_h_from_H}, the numerical flux for \cref{eq:3pt-scheme-a} reads
\begin{equation}\label{eq:bouchut-flux}
 {\bf h}^{r}({\bf u}_L,{\bf u}_R,{\bf n}) = \left\{ \begin{array}{ll}
 {\bf f}({\bf u}_L)\cdot{\bf n},      & \tfrac{x}{t} < S_L, \\
 {\bf f}({\bf u}_L^\star)\cdot{\bf n}, & S_L<\tfrac{x}{t}<u^\star,\\
 {\bf f}({\bf u}_R^\star)\cdot{\bf n}, & u^\star<\tfrac{x}{t}<S_R,\\
 {\bf f}({\bf u}_R)\cdot{\bf n},      & S_R < \tfrac{x}{t}, \end{array} \right.
\end{equation}

\noindent with ${\bf u}_X^\star=(\rho_X^\star {\bf Y}_X^\top,\rho_X^\star {\bf v}_X^{\star\top},\rho_X^\star E_{X}^\star, \rho_X^\star {\bf e}_{v,X}^\top)^\top$, for $X=L,R$, and
\begin{align*}
  u^\star  &= \frac{a_Lu_L+a_Ru_R+\mathrm{p}_{L}-\mathrm{p}_{R}}{a_L+a_R}, \quad \mathrm{p}^\star = \frac{a_R\mathrm{p}_{L}+a_L\mathrm{p}_{R}+a_La_R(u_L-u_R)}{a_L+a_R},  \\
  E_{L}^\star &= E_{L} - \frac{\mathrm{p}^\star u^\star - \mathrm{p}_{L}u_L}{a_L},\quad E_{R}^\star = E_{R} - \frac{\mathrm{p}_{R}u_R-\mathrm{p}^\star u^\star}{a_R},
\end{align*}

\noindent $\p_X=\p({\bf Y}_X,\rho_X,e_{t_X})$ for $X=L,R$, the other quantities being defined in \cref{eq:sol-PR-relax} and the wave speed estimates are defined from \cref{eq:wave-estimate} with $\p_r=\p$ from \cref{eq:flux_h_from_H}.


%

%
\subsection{Wall boundary conditions}\label{sec:ES_BC}

Let consider the case of an impermeability condition, ${\bf v}\cdot{\bf n}=0$, at a wall $\Gamma_w\subset\partial\Omega_h$ which is commonly imposed through the use of mirror state ${\bf u}^+=(\rhob,-\rho{\bf v}^\top,\rho E,\rho {\bf e}_v^\top)^\top$. For elements $\kappa$ adjacent to a wall, we modify \cref{eq:2D-FV-scheme} in the following way
\begin{equation*}
 {\bf U}_\kappa^{n+1} = {\bf U}_\kappa^{n} - \frac{\Delta t^{(n)}}{|\kappa|}\Big(\sum_{e\in\partial\kappa\backslash\Gamma_w} |e|{\bf h}({\bf U}_\kappa^{n},{\bf U}_{\kappa_e^+}^{n},{\bf n}_e) + \sum_{e\in\partial\kappa\cap\Gamma_w} |e|{\bf h}^{r}(0,{\bf U}_\kappa^{n},{\bf U}_{\kappa}^{n^+},{\bf n}_e) \Big),
\end{equation*}

\noindent where ${\bf h}$ corresponds to one of the above numerical fluxes, ${\bf h}^{r}$ is pressure relaxation-based flux \cref{eq:bouchut-flux}, and the exponent $^+$ denotes the mirror state. The above scheme still can be written as a convex combination of updates of three-point schemes \cref{eq:3pt-scheme-a} as in \cref{eq:2D-FV-scheme-update-3pt}, so the entropy inequality \cref{eq:2D-FV-scheme-ineq} holds.

Using the mirror state we have from \cref{eq:sol-PR-relax} that $a_L=a_R=a$, $\p_L=\p_R=\p$ given by \cref{eq:mixture_eos}, ${\bf Y}_L={\bf Y}_R={\bf Y}$ so the left and right states have the same thermodynamics. We thus obtain ${\bf h}^{r}(0,{\bf u},{\bf u}^+,{\bf n})=(0,0,\p^\star{\bf n},0,0)^\top$ with $\p^\star=\p+a{\bf v}\cdot{\bf n}$, and $a=\sqrt{\gamma({\bf Y})\rho\p}+(\gamma({\bf Y})+1)\rho({\bf v}\cdot{\bf n})^+$. This boundary condition is consistent with the impermeability condition and enforces the pressure through the characteristic associated to the eigenvalue ${\bf v}\cdot{\bf n}+c({\bf Y},e_t)$.

Note that from \cref{th:ES-3-pt-scheme} the entropy flux vanishes at wall boundary interfaces since by $Q({\bf u},{\bf u}^+,{\bf n})=Z\big({\cal P}({\bf u}),{\cal P}({\bf u}^+),{\bf n}\big)=\eta({\bf u}){\bf v}\cdot{\bf n}$ evaluated at $\bcW^{r}\big(0;{\cal P}({\bf u}),{\cal P}({\bf u}^+),{\bf n}\big)$ for which ${\bf v}\cdot{\bf n}=u^\star=0$. Assuming either compactly supported solutions, or using ES boundary conditions from \cite{Svard_Ozcan_14} at far-field boundaries, we end with the following global estimate on the entropy:
\begin{equation*}
  \sum_{\kappa\in\Omega_h} |\kappa|\eta({\bf U}_\kappa^{n+1}) \leq \sum_{\kappa\in\Omega_h} |\kappa|\eta({\bf U}_\kappa^{n}) + C,
\end{equation*}

\noindent where $C$ is a constant that depends on boundary data. Using the strict convexity of the entropy $\eta({\bf u})$, one may use Dafermos' argument to prove $L^2$ stability of the solution \cite{Dafermos2016} (see e.g. \cite[Th.~2.6]{Svard_Ozcan_14}).


%
%
\section{Numerical experiments}\label{sec:num_xp}

In this section we present numerical experiments, obtained with the CFD code {\it Aghora} developed at ONERA \cite{renac_etal15}, on problems in one and two space dimensions in order to illustrate the performance of the schemes derived in this work. We use $\gamma=1.01\times\tfrac{5}{3}$ in \cref{eq:godunov_flux} to ensure the inequality in \cref{eq:gamma-max}, while we set $\gamma=\tfrac{5}{3}$ in \cref{eq:wave-estimate} and increase the wave speed estimates $S_{X=L,R}$ by a factor $1.01$ in \cref{eq:hll_flux,eq:bouchut-flux}. The time step is evaluated through \cref{eq:3pt-scheme-CFL}. For 2D simulations, we impose the freestream values at supersonic inlets and extrapolate variables at supersonic outlets, while we apply the impermeability boundary condition in \cref{sec:ES_BC} at walls. Steady computations are obtained by using local time stepping until the $\ell^2$ norm of the vector of residuals has decreased by a factor $10^{10}$. Additional results obtained for a monocomponent perfect gas with an equivalent adiabatic exponent are also reported for the sake of comparison: we use either the Roe solver \cite{Roe_1981} with entropy fix \cite{harten_hyman_83} (referred to as ROE-PG), or the HLL solver \cite{hll_83} with the two-rarefaction approximation \cite[Ch.~9]{toro_book} for computing the wave speeds (referred to as HLL-PG).

\subsection{One-dimensional shock-tube problems}

We first consider the convection of a material interface separating air ($\rho_L=3.607655$, $Y_{1,L}=1-Y_{2,L}=1$, $e_{v_1,L}=1.8070291$, $\gamma_1=1.4$) in thermal disequilibrium from helium ($\rho_R=0.5$, $Y_{1,R}=1-Y_{2,R}=0$, $e_{v_1,R}=0$, $\gamma_2=\tfrac{5}{3}$) in a flow with pressure $\p_L=\p_R=1$ and velocity $u_L=u_R=1$. Results are shown in \cref{fig:solution_RP_contact} and highlight convergence of the three schemes with some more smearing of the contact by the HLL scheme as expected.

\begin{figure}
\begin{bigcenter}
\subfloat{\epsfig{figure=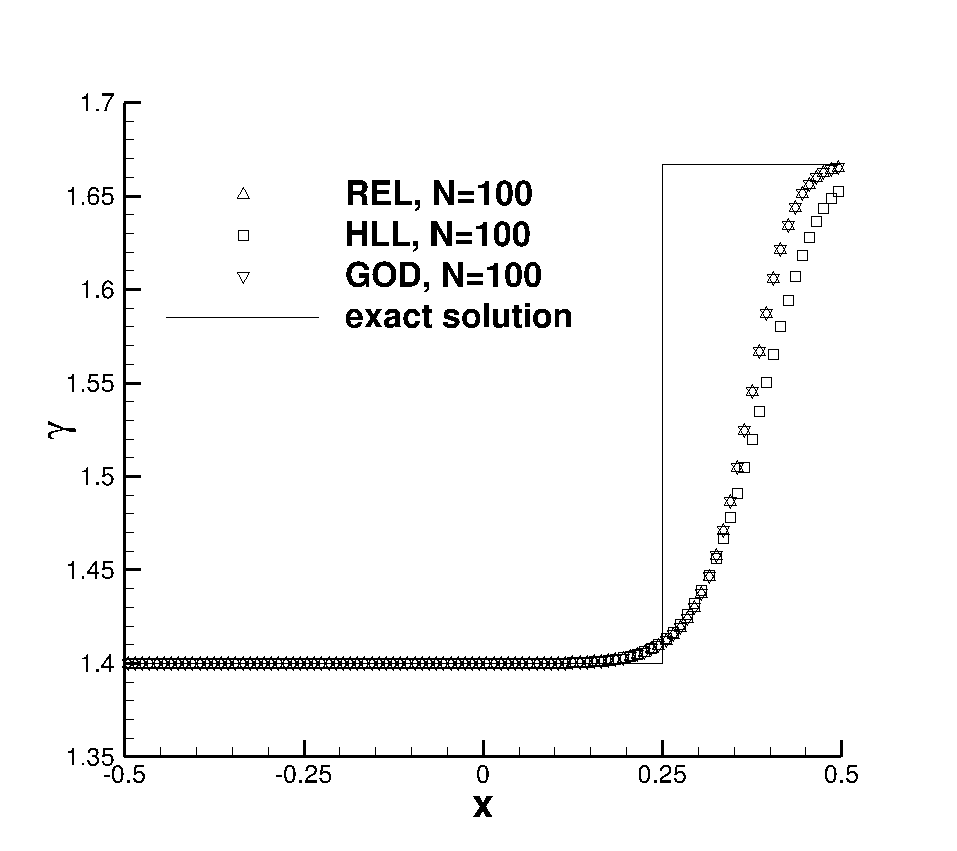,width=3.5cm} \hspace{-0.45cm}}
\subfloat{\epsfig{figure=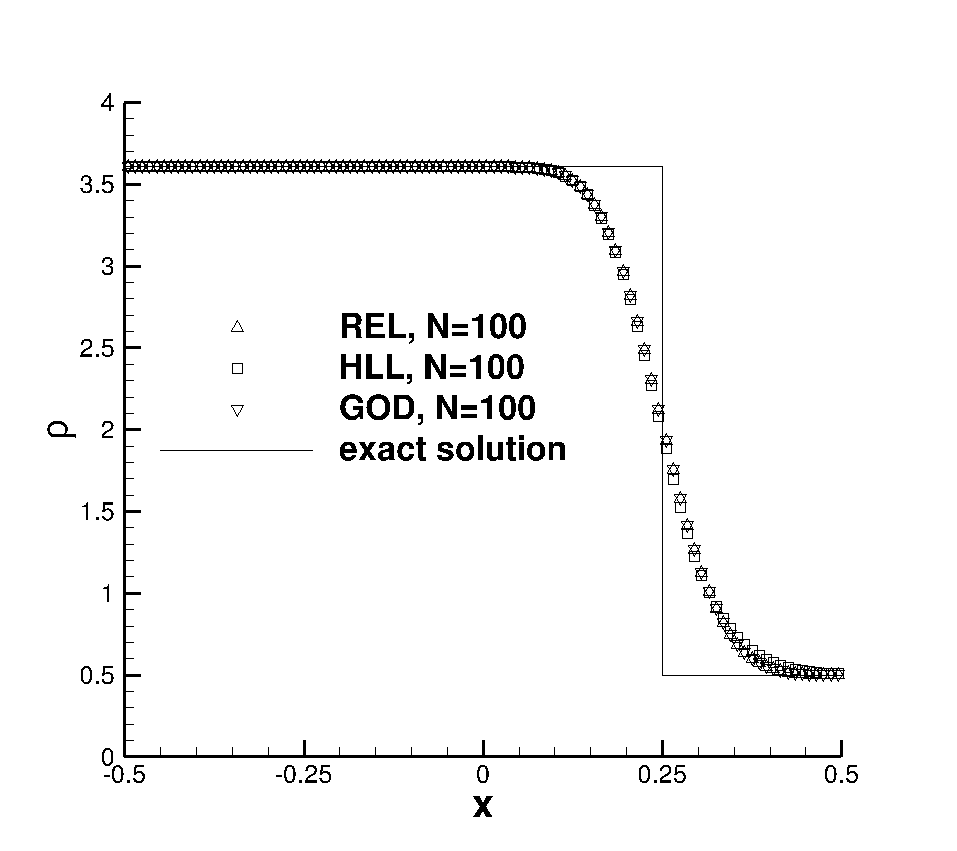,width=3.5cm} \hspace{-0.45cm}}
\subfloat{\epsfig{figure=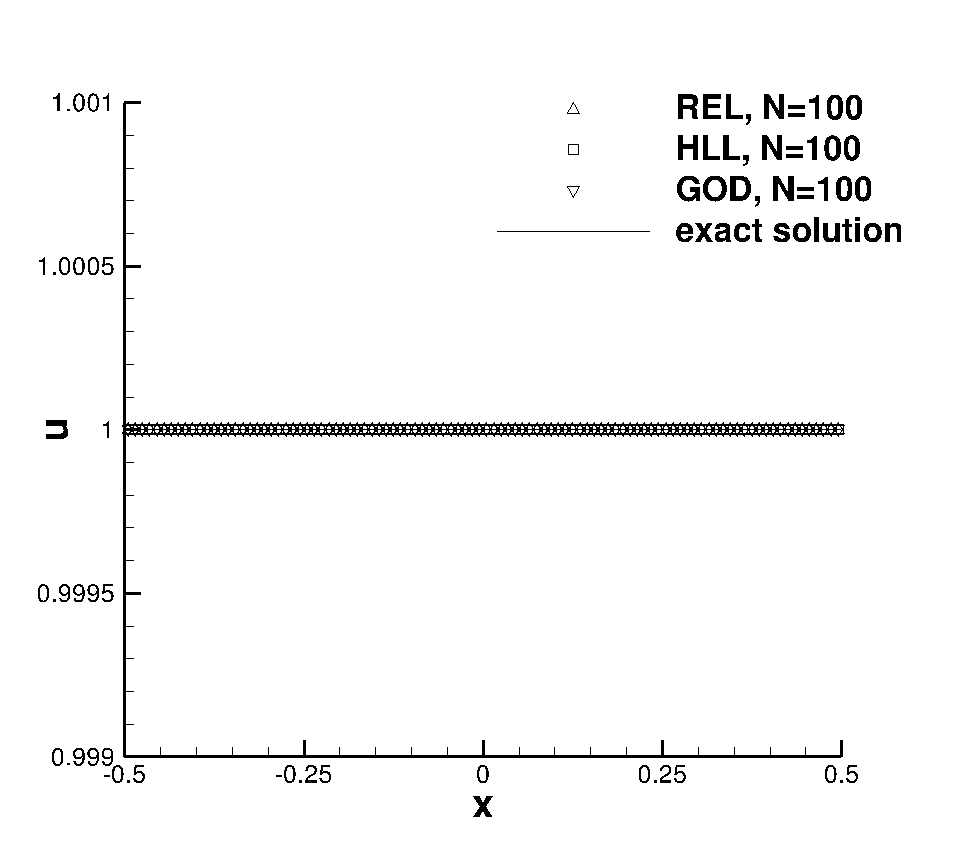,width=3.5cm} \hspace{-0.45cm}}
\subfloat{\epsfig{figure=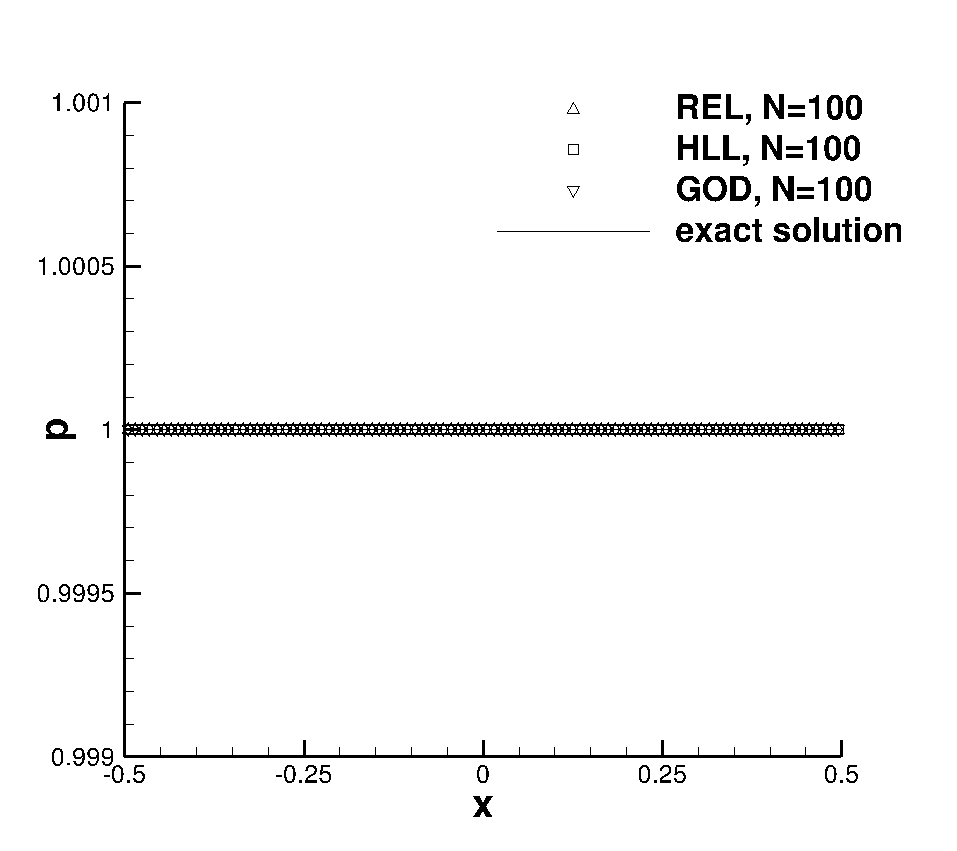,width=3.5cm}}\\ 
\setcounter{subfigure}{0}
\subfloat[$\gamma({\bf Y})$]{\epsfig{figure=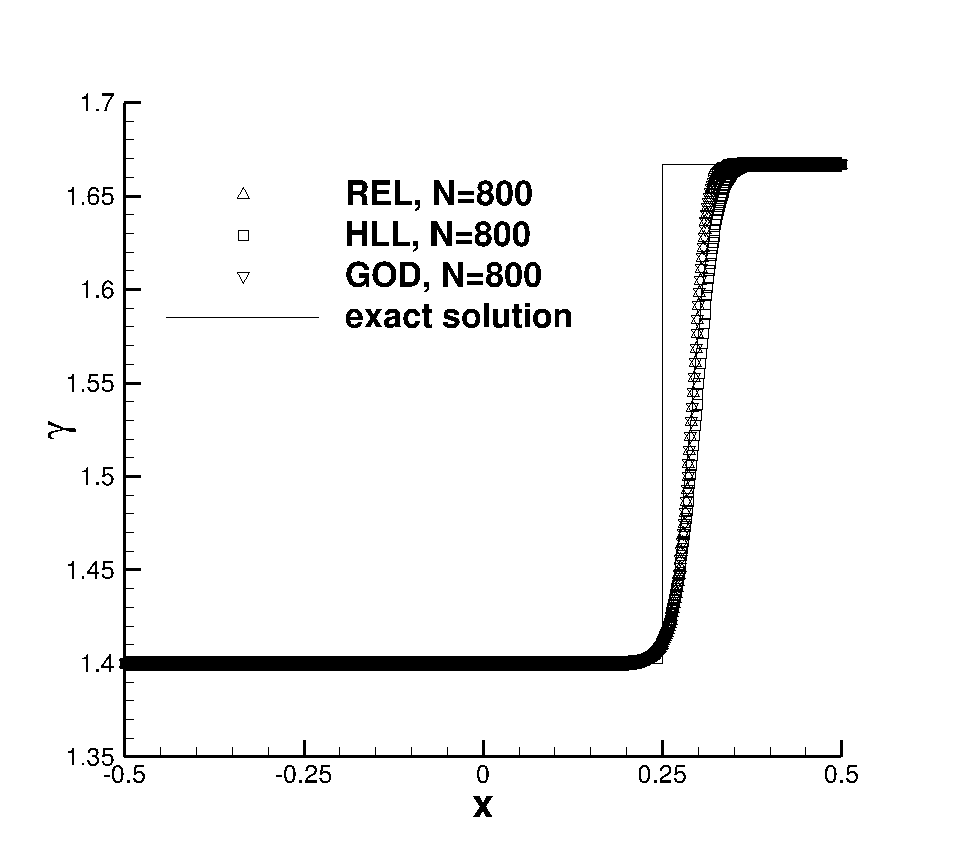,width=3.5cm} \hspace{-0.45cm}}
\subfloat[$\rho$]{\epsfig{figure=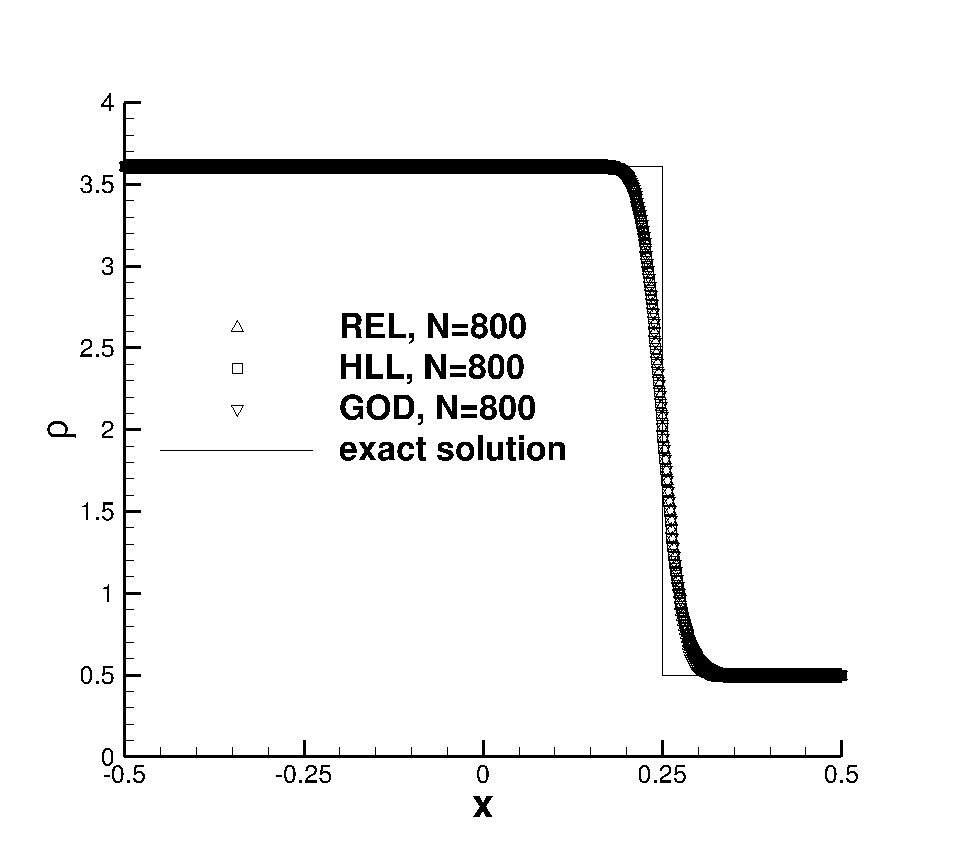,width=3.5cm} \hspace{-0.45cm}}
\subfloat[$u$]{\epsfig{figure=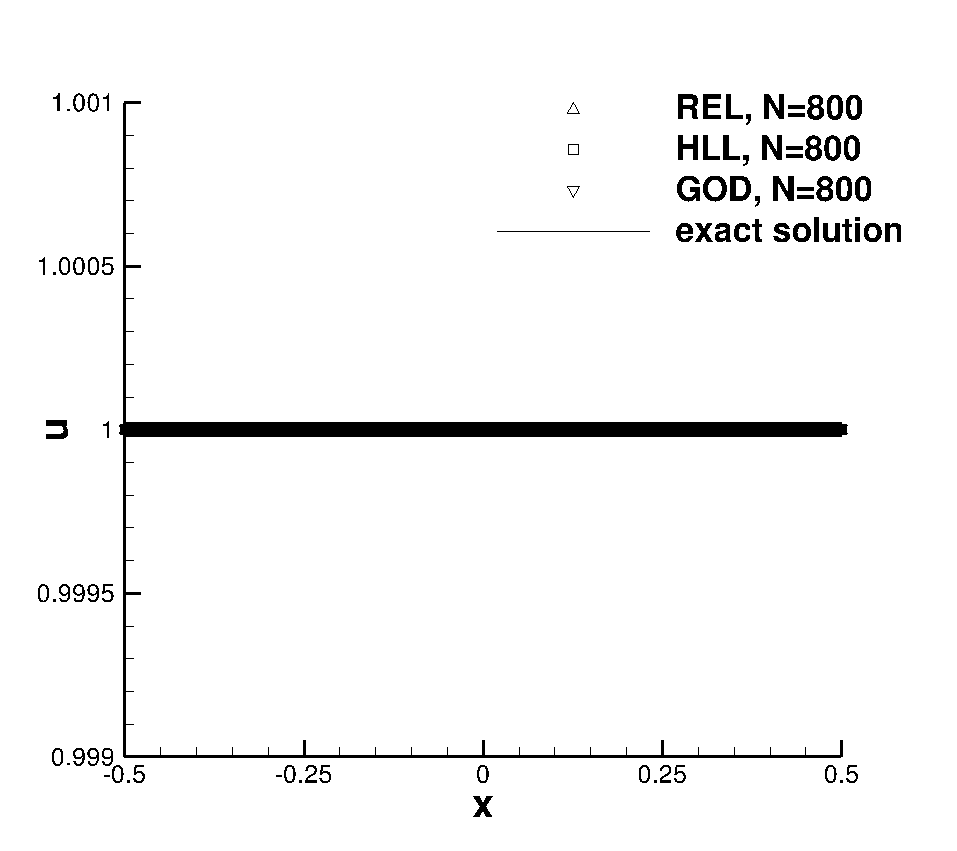,width=3.5cm} \hspace{-0.45cm}}
\subfloat[$\p$]{\epsfig{figure=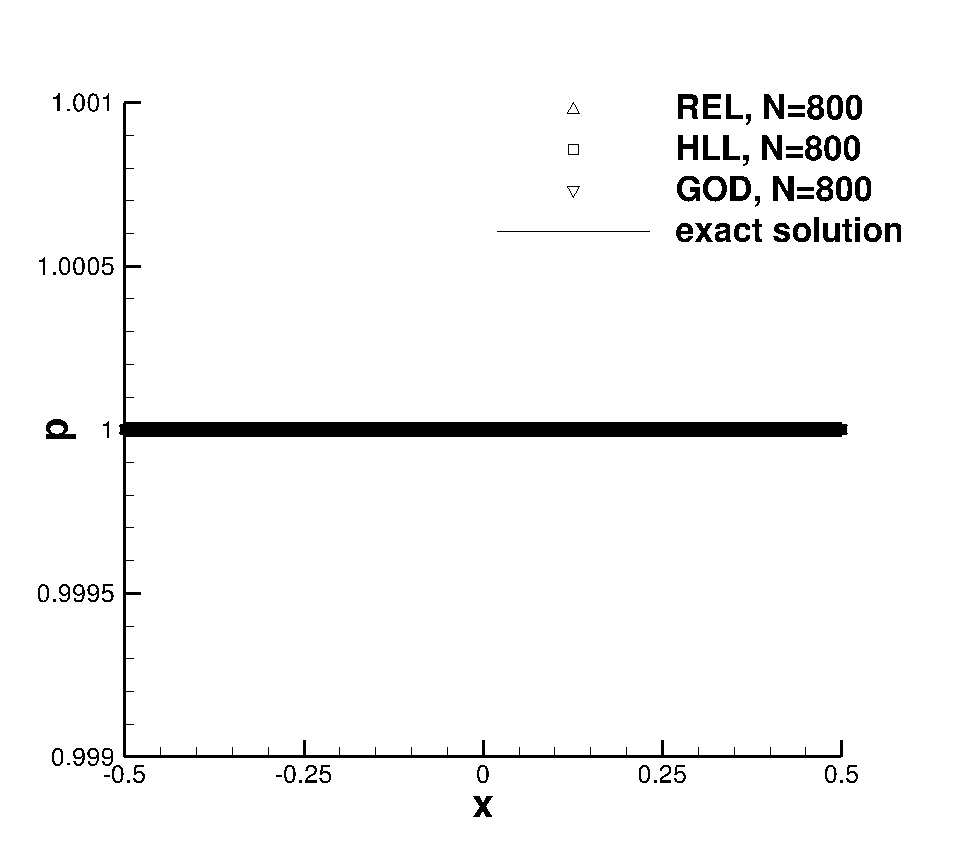,width=3.5cm}} 
\caption{Convection of a material interface: results obtained with the numerical fluxes \cref{eq:bouchut-flux} (REL), \cref{eq:hll_flux} (HLL), and \cref{eq:godunov_flux} (GOD), on two grids with $N=100$ and $N=800$ elements.}
\label{fig:solution_RP_contact}
\end{bigcenter}
\end{figure}

We now consider a shock tube problem adapted from \cite{liou_etal_split_RG_90} initially separating regions with large pressure and temperature ratios: $u_L=u_R=0$, $\p_L=100\p_R=100$bars, and $\T_L=30\T_R=9000$K. We consider air in thermal equilibrium with a 5 species model with a uniform composition $Y_{N_2}=0.7543$, $Y_{O_2}=0.2283$, $Y_{NO}=0.01026$, $Y_{N}=6.5\times10^{-7}$, and $Y_{O}=0.00713$. We neglect the enthalpies of formation so the gas is a perfect gas with an equivalent adiabatic exponent $\gamma({\bf Y})=1.402$ and we compare our results to the Roe solver for a perfect gas with an adiabatic exponent of $1.402$ (ROE-PG). Results in \cref{fig:solution_RP_air_liou} show that all solvers provide similar results and converge to the entropy weak solution. We stress that in spite of the crude assumption $\gamma>\tfrac{5}{3}$ in the numerical fluxes from \cref{sec:ex-ES-3pt-schemes} they offer similar accuracy as the Roe solver.

\begin{figure}
\begin{center}
\subfloat{\epsfig{figure=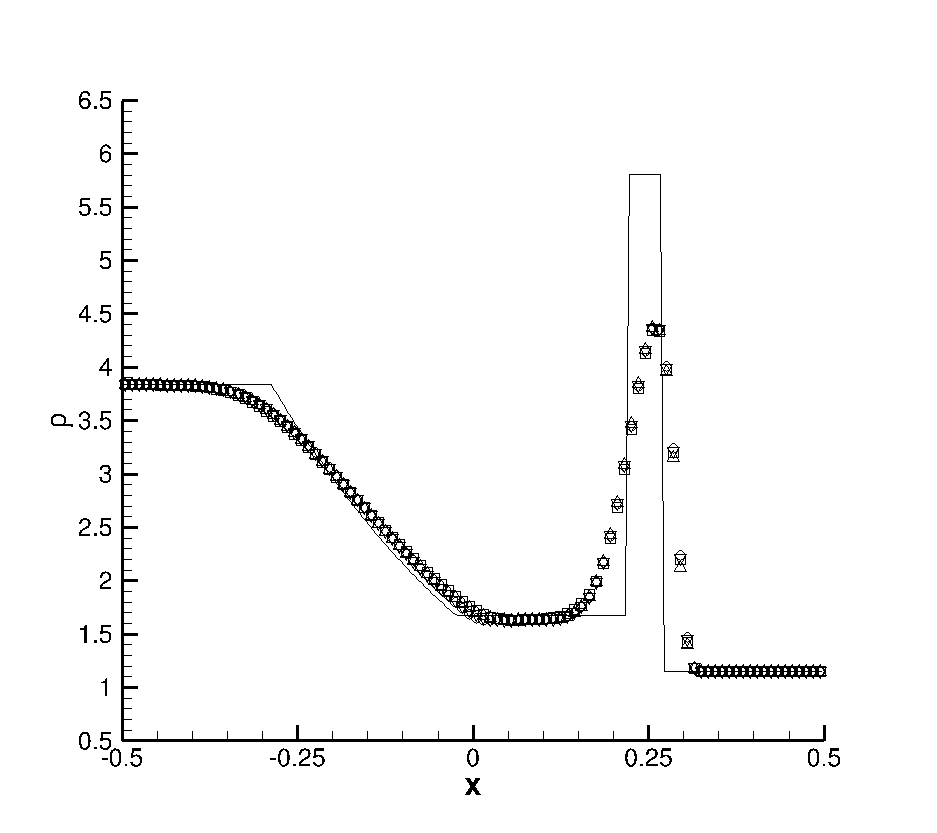,width=3.5cm} \hspace{-0.45cm}}
\subfloat{\epsfig{figure=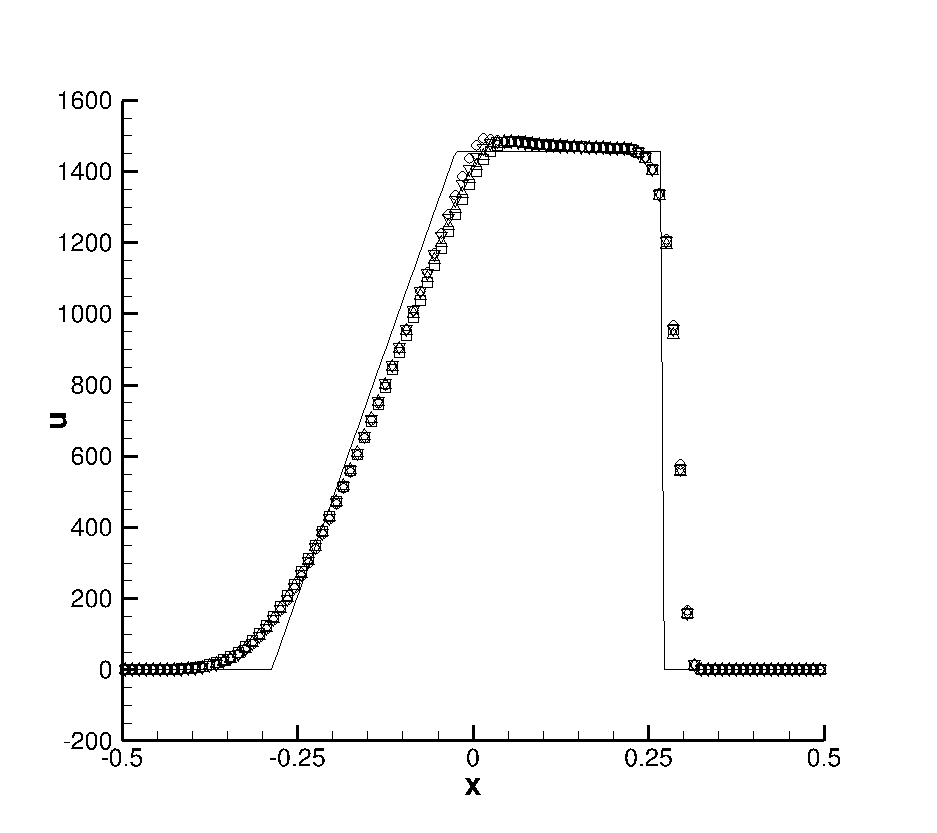,width=3.5cm} \hspace{-0.45cm}}
\subfloat{\epsfig{figure=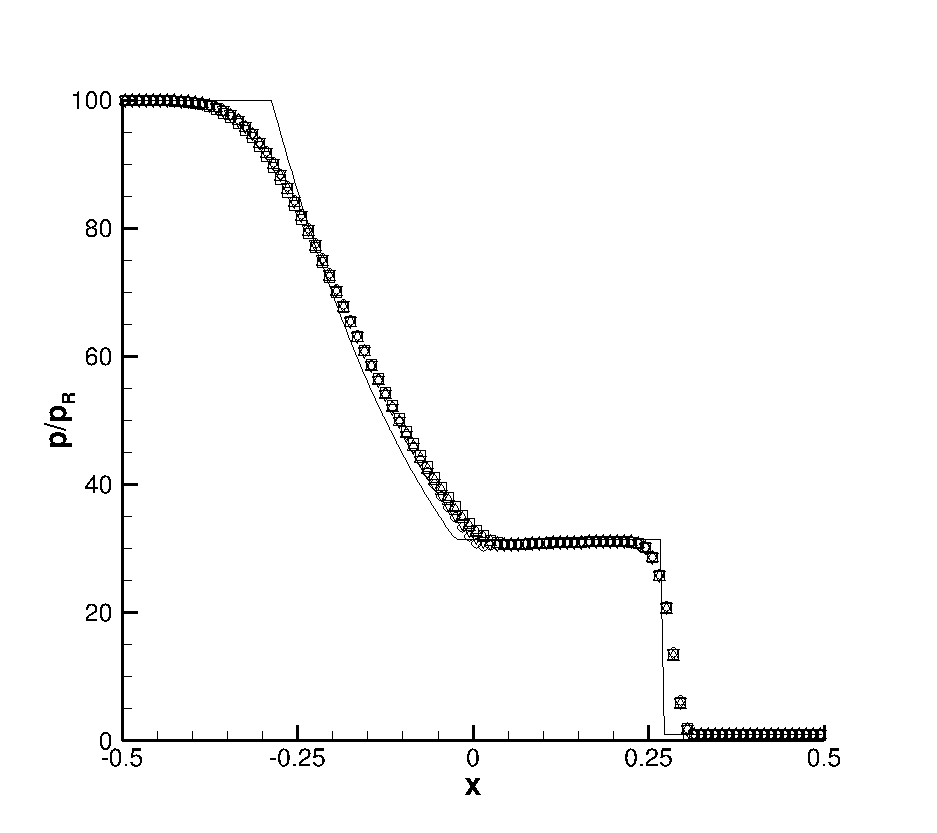,width=3.5cm} \hspace{-0.45cm}}
\subfloat{\epsfig{figure=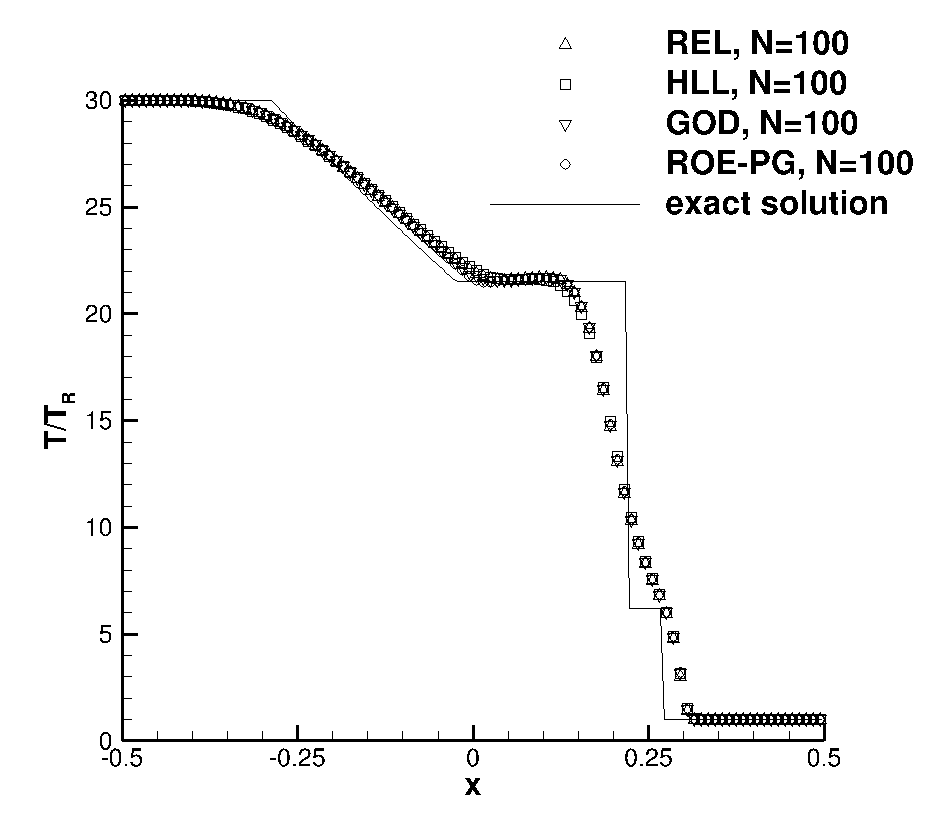,width=3.5cm}}\\ 
\setcounter{subfigure}{0}
\subfloat[$\rho$]{\epsfig{figure=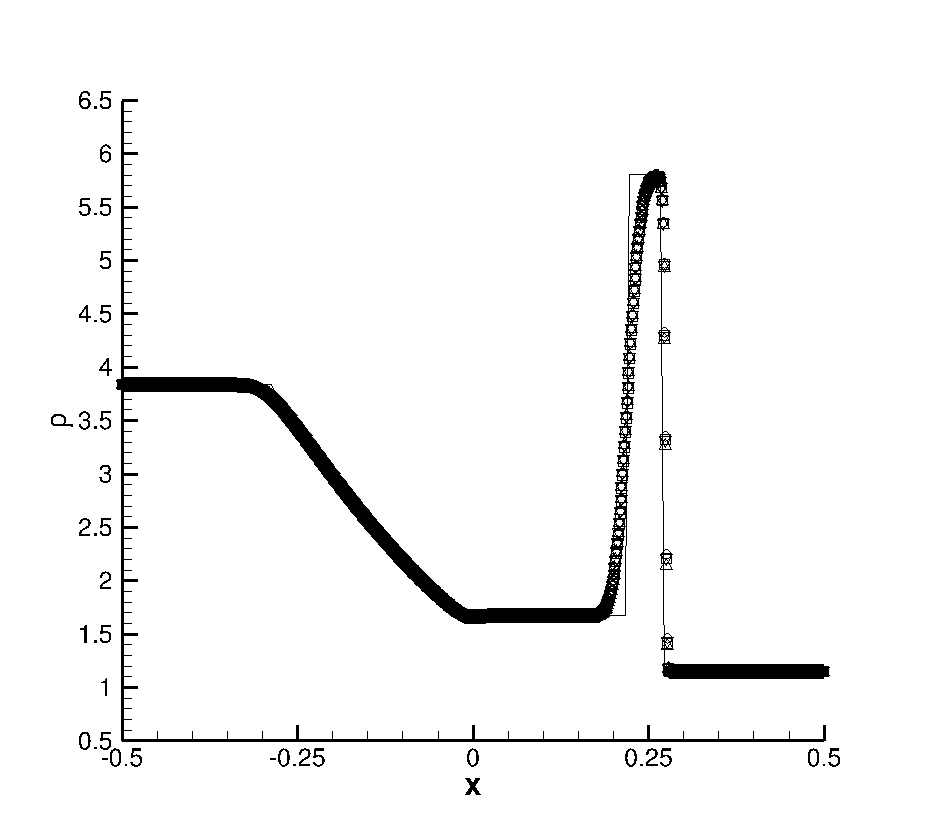,width=3.5cm} \hspace{-0.45cm}}
\subfloat[$u$]{\epsfig{figure=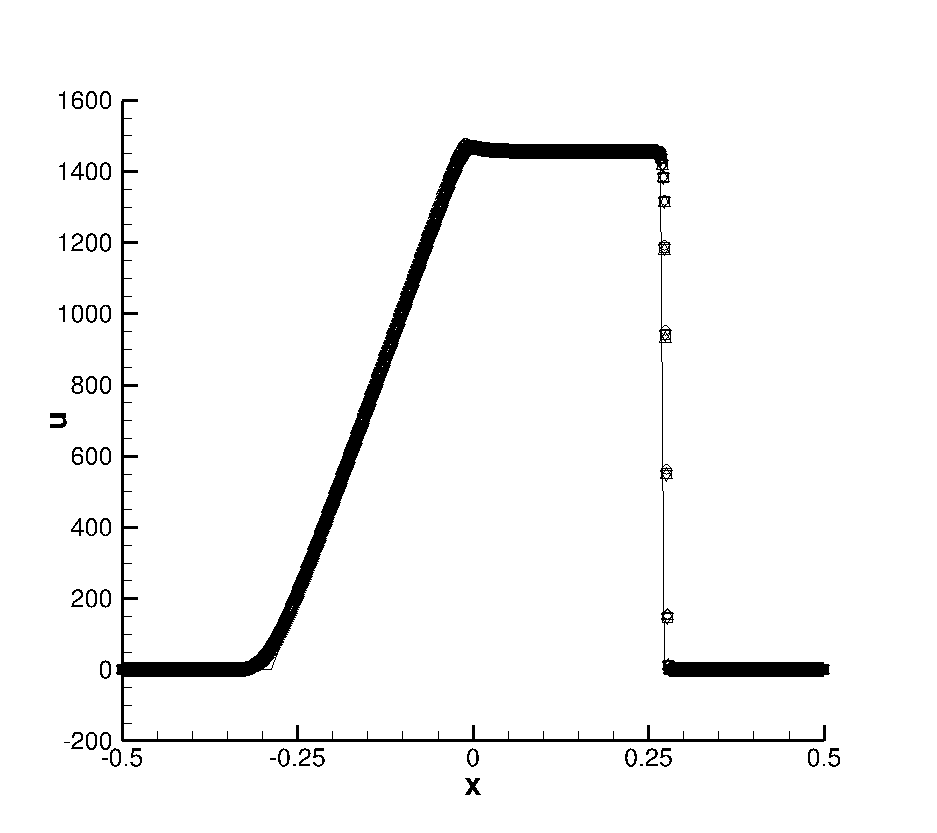,width=3.5cm} \hspace{-0.45cm}}
\subfloat[$\tfrac{\p}{\p_R}$]{\epsfig{figure=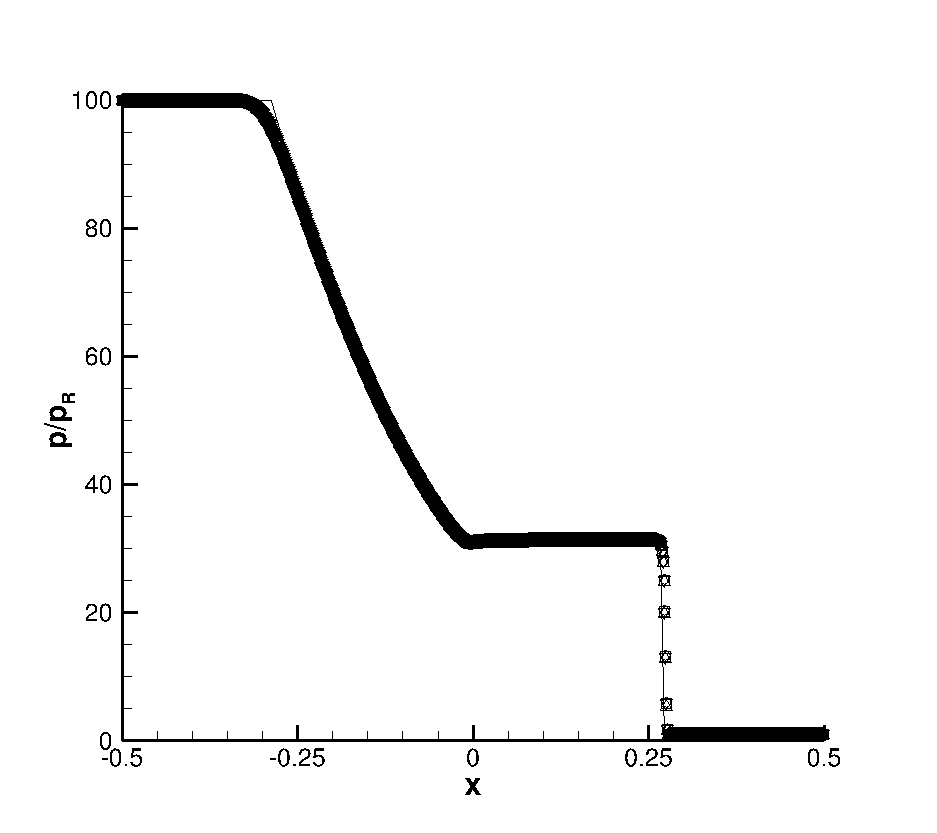,width=3.5cm} \hspace{-0.45cm}}
\subfloat[$\tfrac{\T}{\T_R}$]{\epsfig{figure=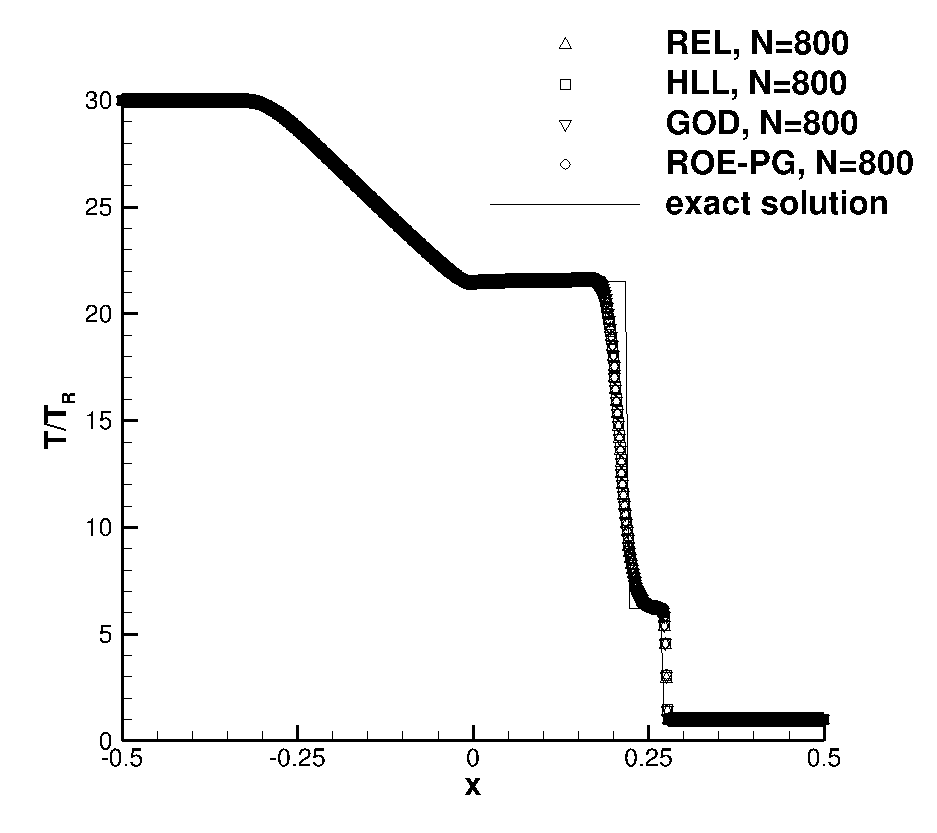,width=3.5cm}} 
\caption{Shock tube filled with air: results at time $t=1.5\times10^{-4}$ obtained with the numerical fluxes \cref{eq:bouchut-flux} (REL), \cref{eq:hll_flux} (HLL), \cref{eq:godunov_flux} (GOD), and the Roe solver for a monocomponent perfect gas (ROE-PG), on two grids with $N=100$ and $N=800$ elements.}
\label{fig:solution_RP_air_liou}
\end{center}
\end{figure}

\subsection{Hypersonic flow over a sphere}

We now consider the 2D hypersonic flow over a $\tfrac{1}{4}$ inch diameter sphere with the freestream conditions of Lobb's experiments \cite{LOBB1964519}. The freestream Mach number is $M_\infty=\tfrac{u_\infty}{c_\infty}=15.3$ with $\rho_\infty=7.83\times10^{-3}$kg/m$^3$ and $\T_\infty=293$K. The upstream flow is made of nitrogen and oxigen with $Y_{N_2}=0.79$, $Y_{O_2}=0.21$ which are uniform in the flow domain since we do not consider chemical reactions or molecular relaxation. The freestream vibration temperatures are taken at $\T_\infty$ for both species. A symmetry condition is imposed at the bottom boundary.

\begin{figure}
\begin{center}
\subfloat[]{\epsfig{figure=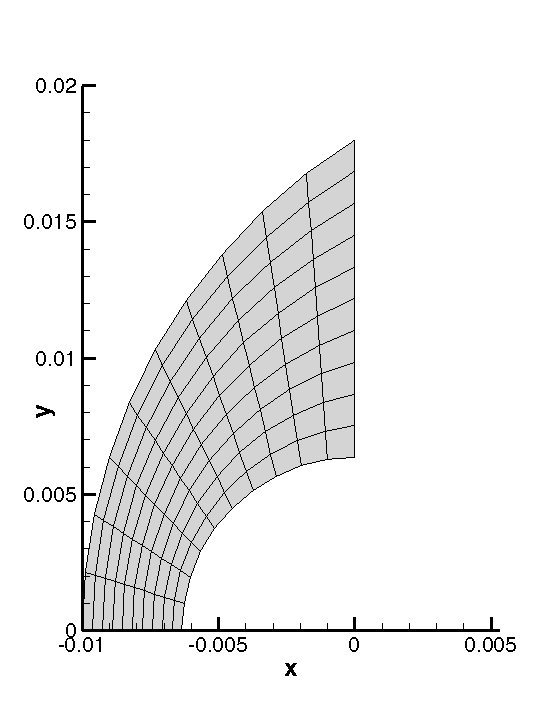,height=4.0cm}} \quad
\subfloat[]{\epsfig{figure=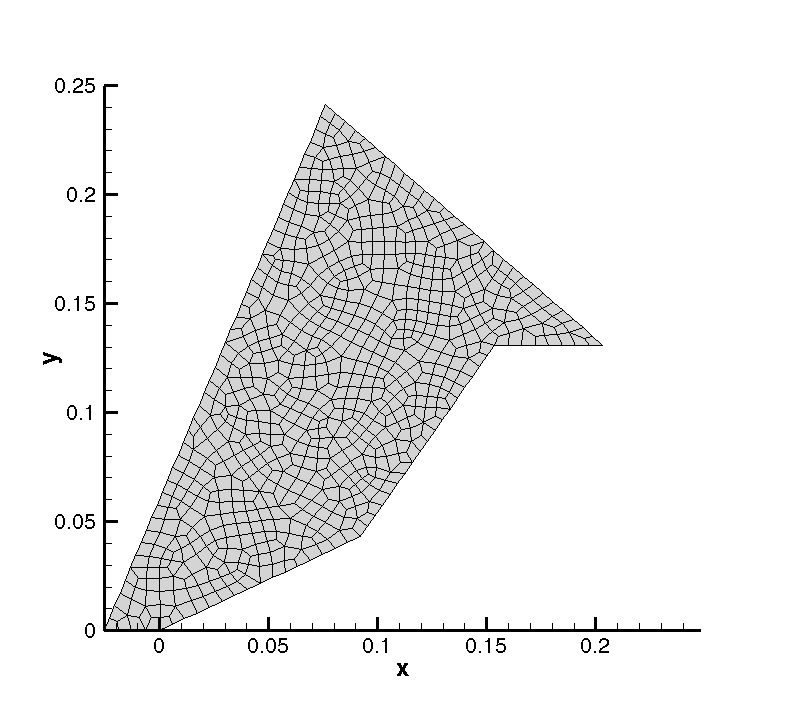, height=4.0cm}}
\caption{Exemples of meshes for the 2D simulations: hypersonic flows over (a) a sphere with $N=100$ elements, (b) a double cone with $N=749$ elements.}
\label{fig:2D_meshes}
\end{center}
\end{figure}

\Cref{fig:solution_lobb_M15} displays the contours of Mach number and translation-rotation temperature on two different grids with the three different schemes. Neglecting chemical reactions overestimates the shock distance to the sphere and prevents comparison to Lobb's experiments. 
As we do not consider chemical reactions or molecular relaxation, a partial validation of the current results can be obtained by comparison with simulations of an equivalent monocomponent perfect gas with adiabatic exponent $\gamma = 1.4$. The simulations of the considered gas mixture using the numerical flux \cref{eq:hll_flux} and of the equivalent perfect gas using the HLL flux for polytropic gas dynamics (HLL-PG) are reported in \cref{fig:solution_lobb_hll_hllms}. As expected, while some differences can be identified for underresolved simulations, the results are almost perfectly overlapping for sufficiently fine resolutions.

We are however interested in comparing results obtained with the different schemes and analysing their convergence under grid refinement. To this end we compare the convergence of shock distance from the sphere in \cref{fig:shock_position_lobb_M15}. We use grids with  $N=20\times20$, $40\times40$, $80\times80$ and  $160\times160$ elements for the simulation (see \cref{fig:2D_meshes}), while the reference distance $x_{ref}$ is evaluated  with the Godunov numerical flux on a fine mesh with $N=320\times320$. The results confirm convergence of the shock position and highlight close values obtained with the three different schemes.

\begin{figure}
\begin{center}
\subfloat{\epsfig{figure=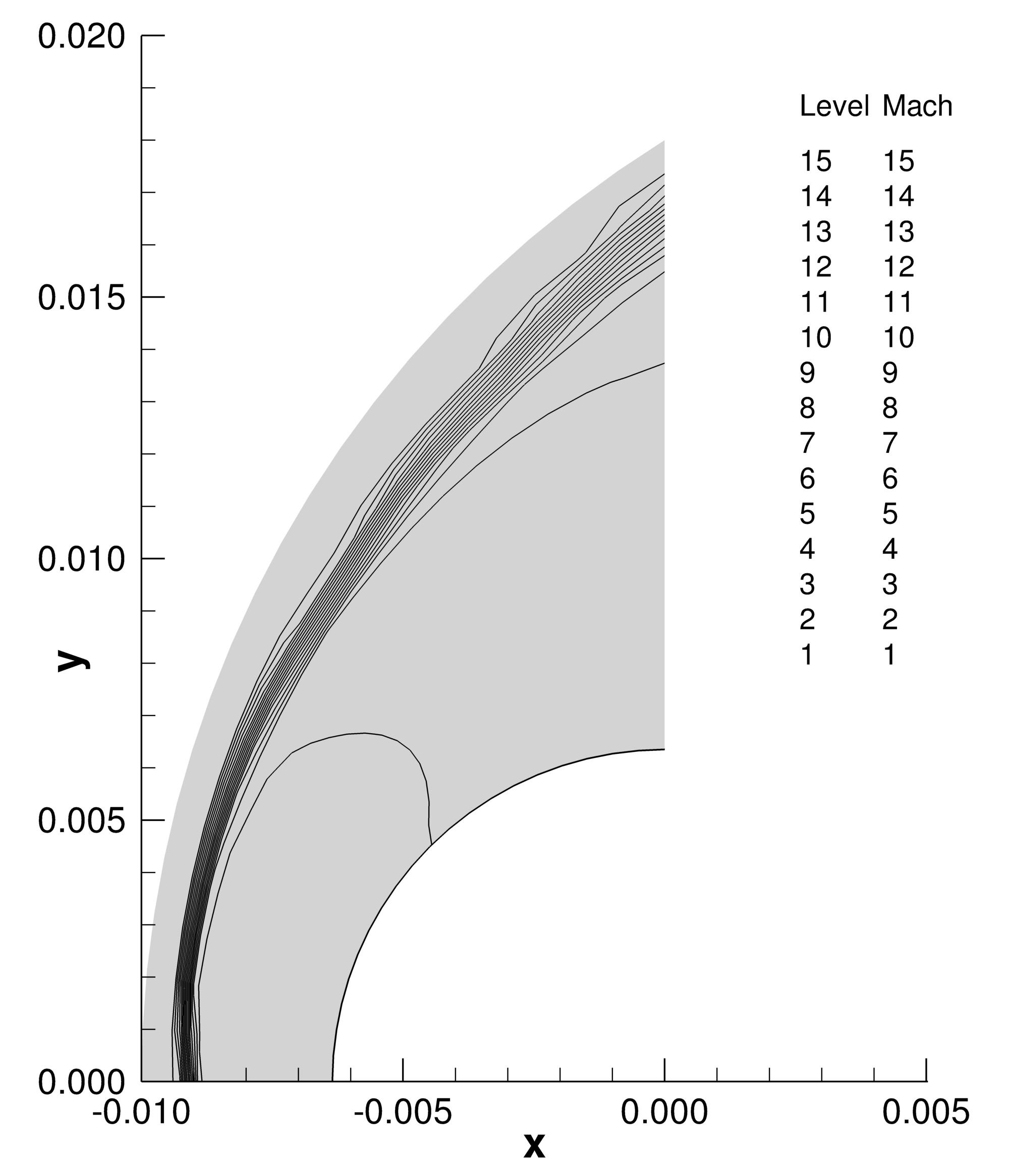,width=3.1cm} \hspace{-0.16cm}}
\subfloat{\epsfig{figure=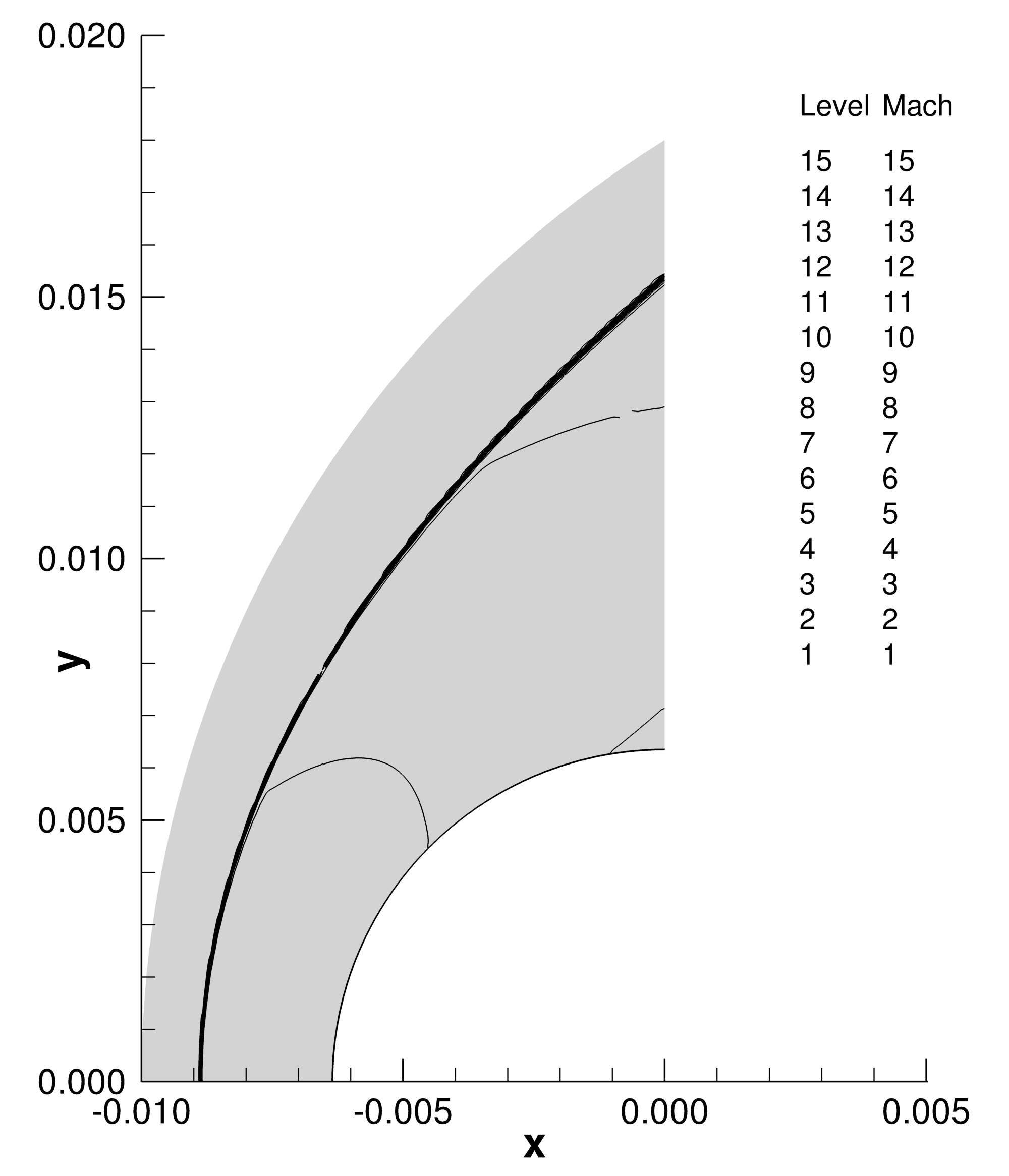,width=3.1cm} \hspace{-0.16cm}}
\subfloat{\epsfig{figure=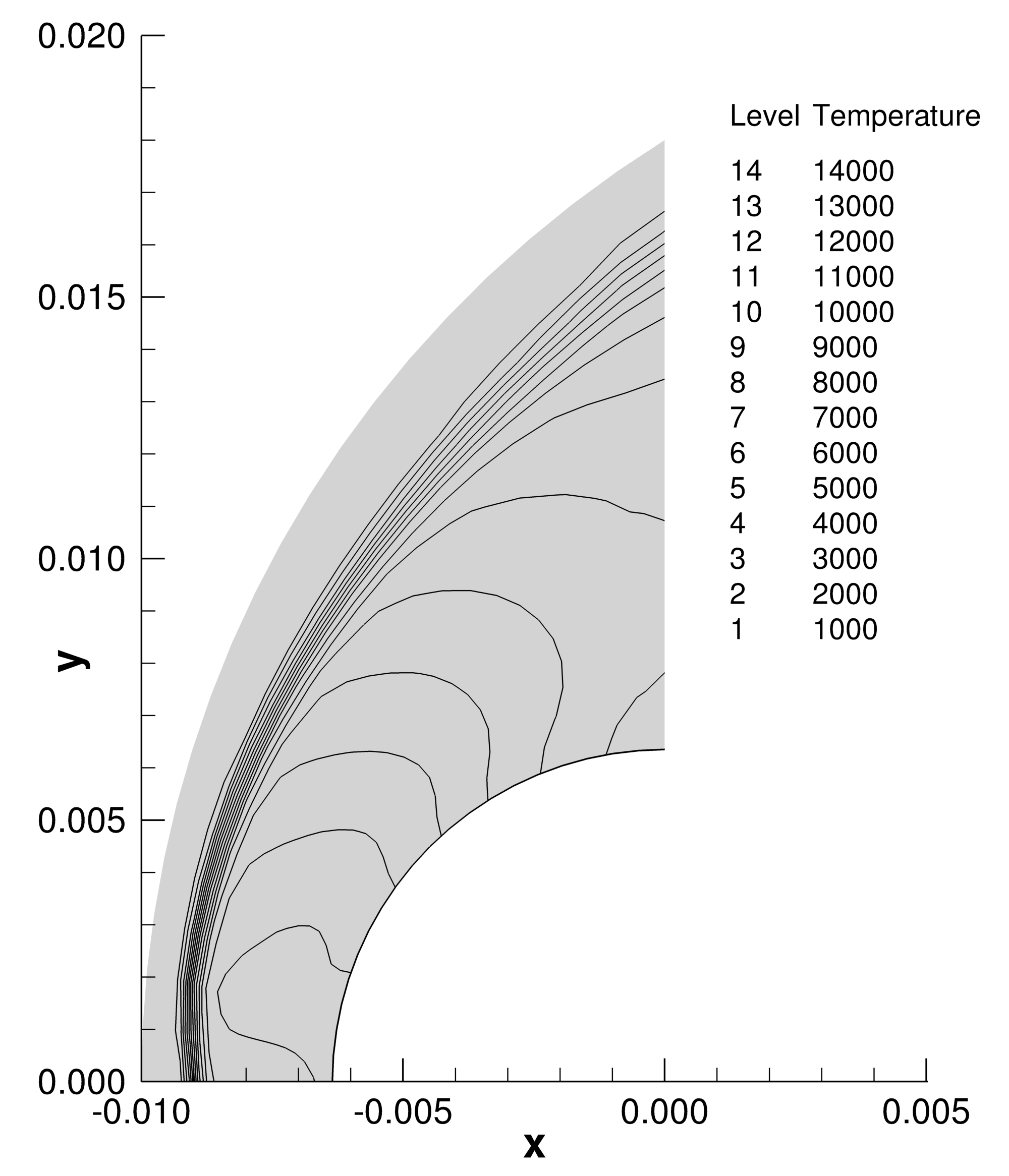,width=3.1cm} \hspace{-0.16cm}}
\subfloat{\epsfig{figure=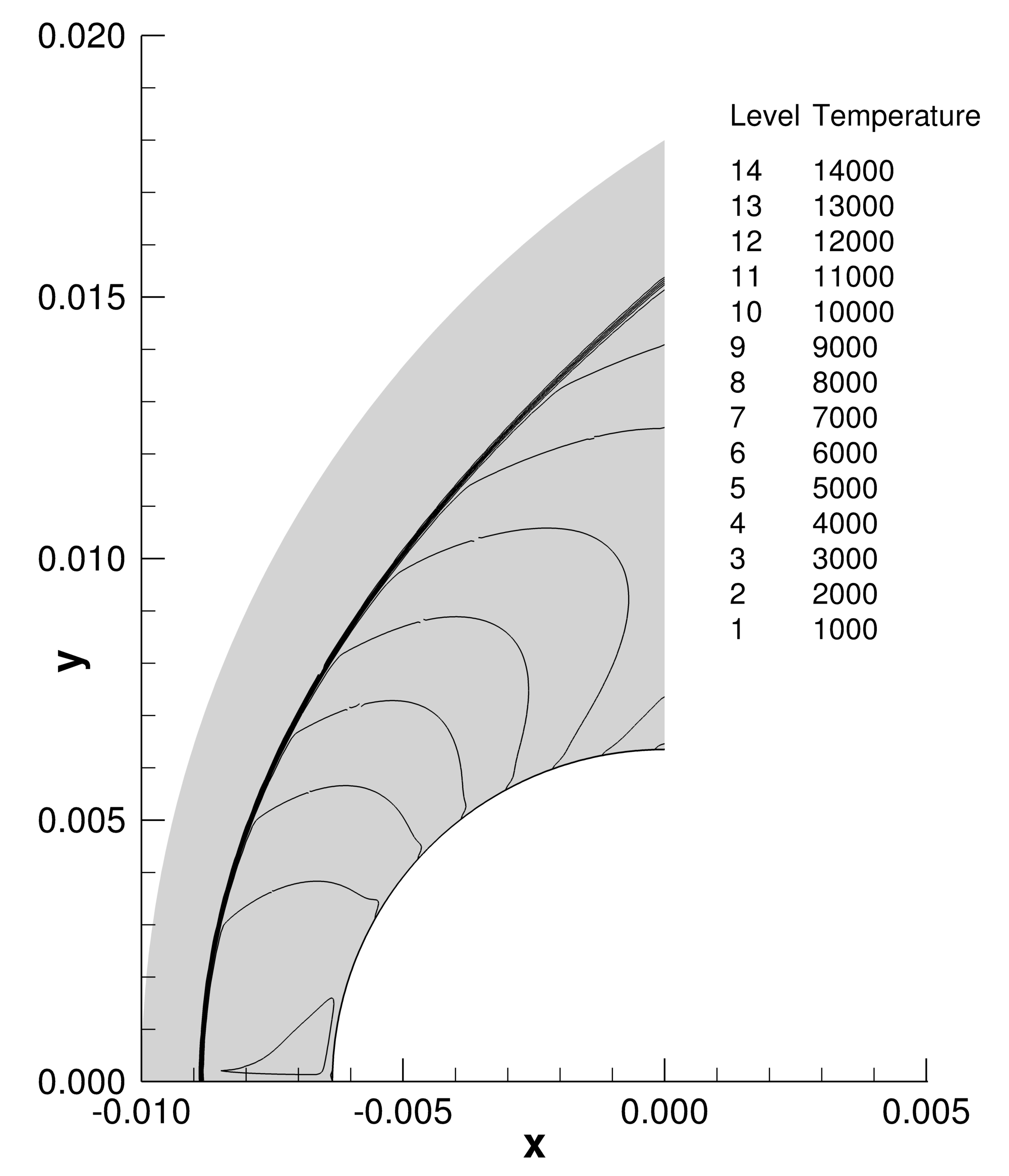,width=3.1cm}}\\
%
\subfloat{\epsfig{figure=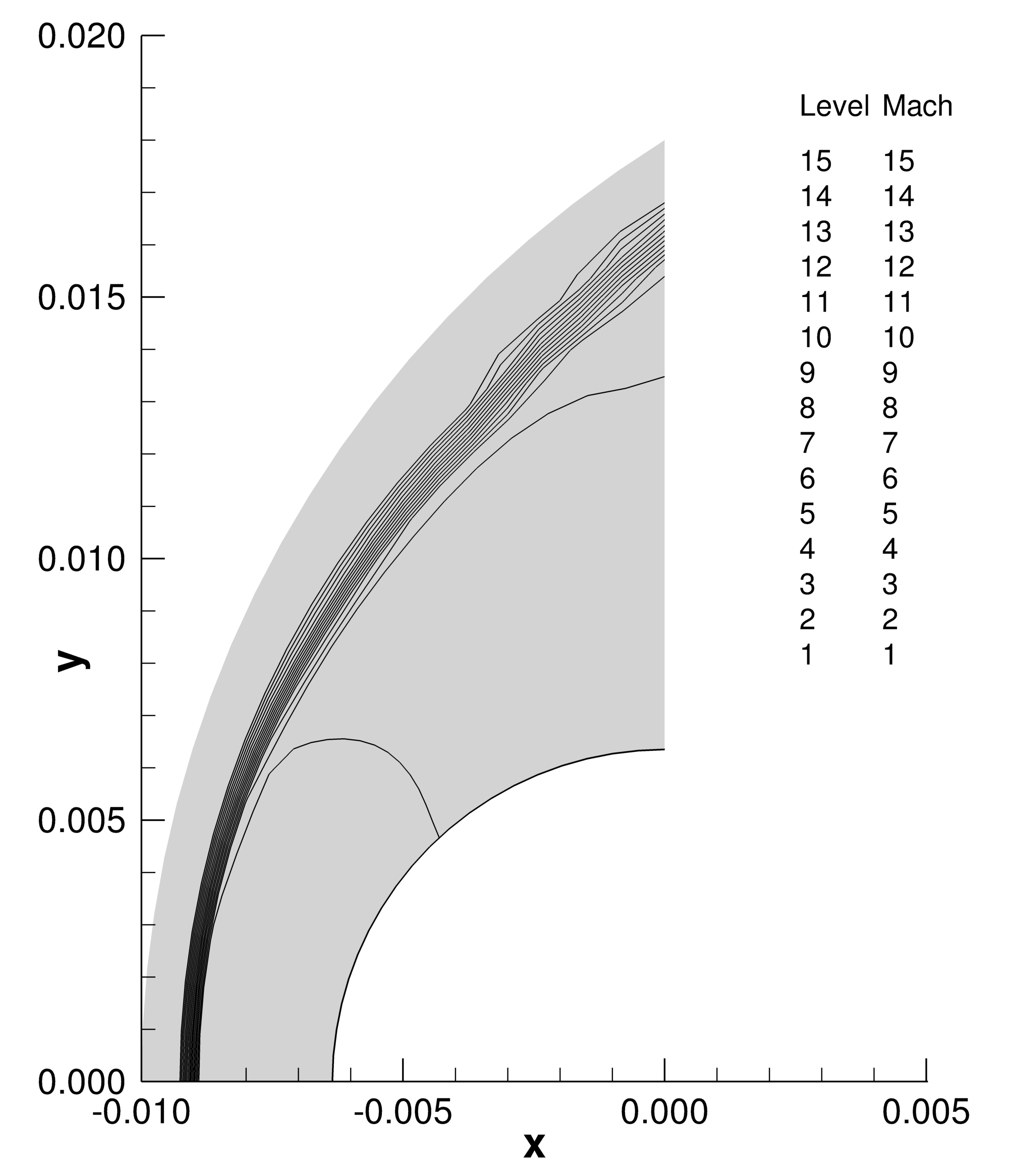,width=3.1cm} \hspace{-0.16cm}}
\subfloat{\epsfig{figure=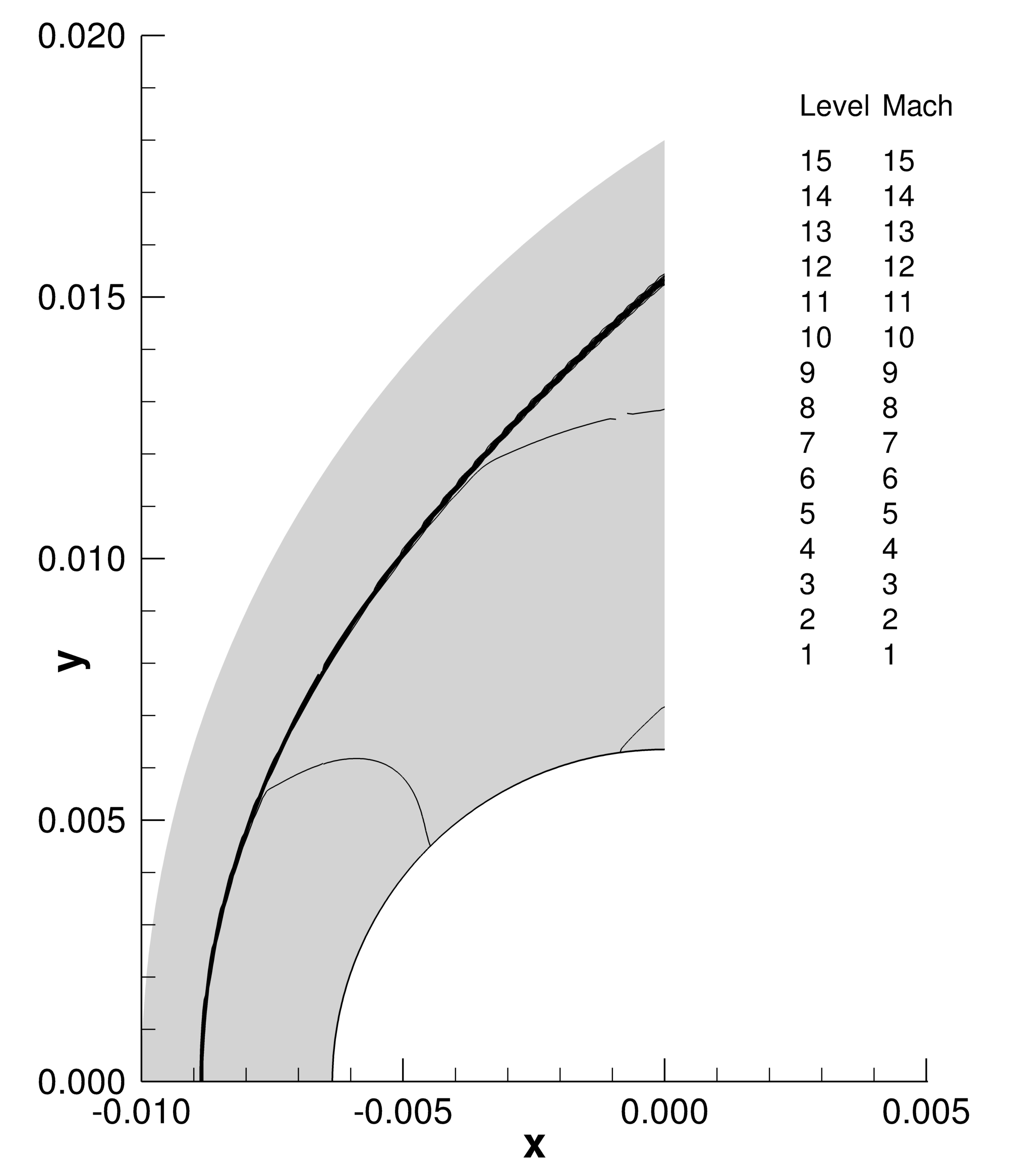,width=3.1cm} \hspace{-0.16cm}}
\subfloat{\epsfig{figure=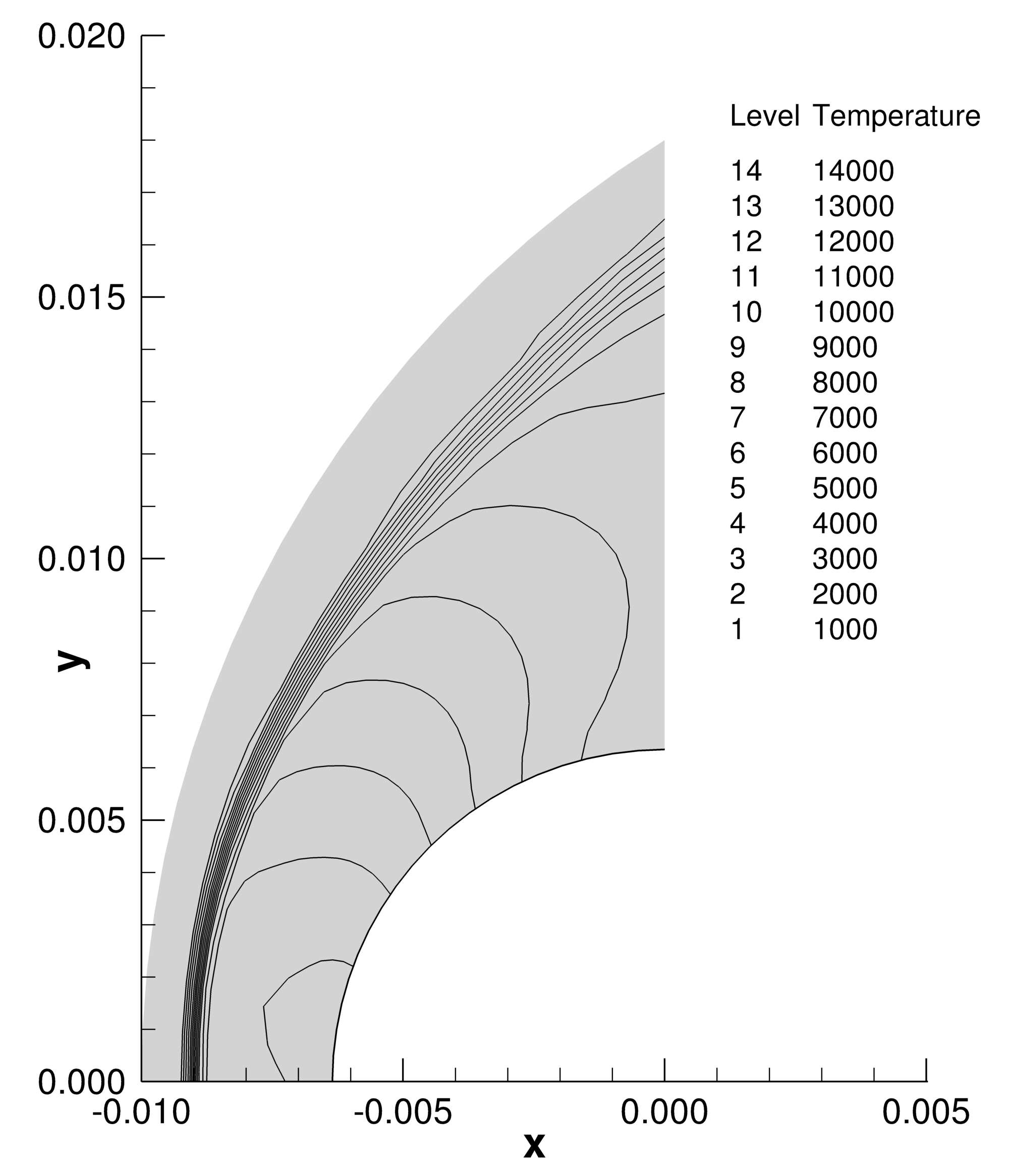,width=3.1cm} \hspace{-0.16cm}}
\subfloat{\epsfig{figure=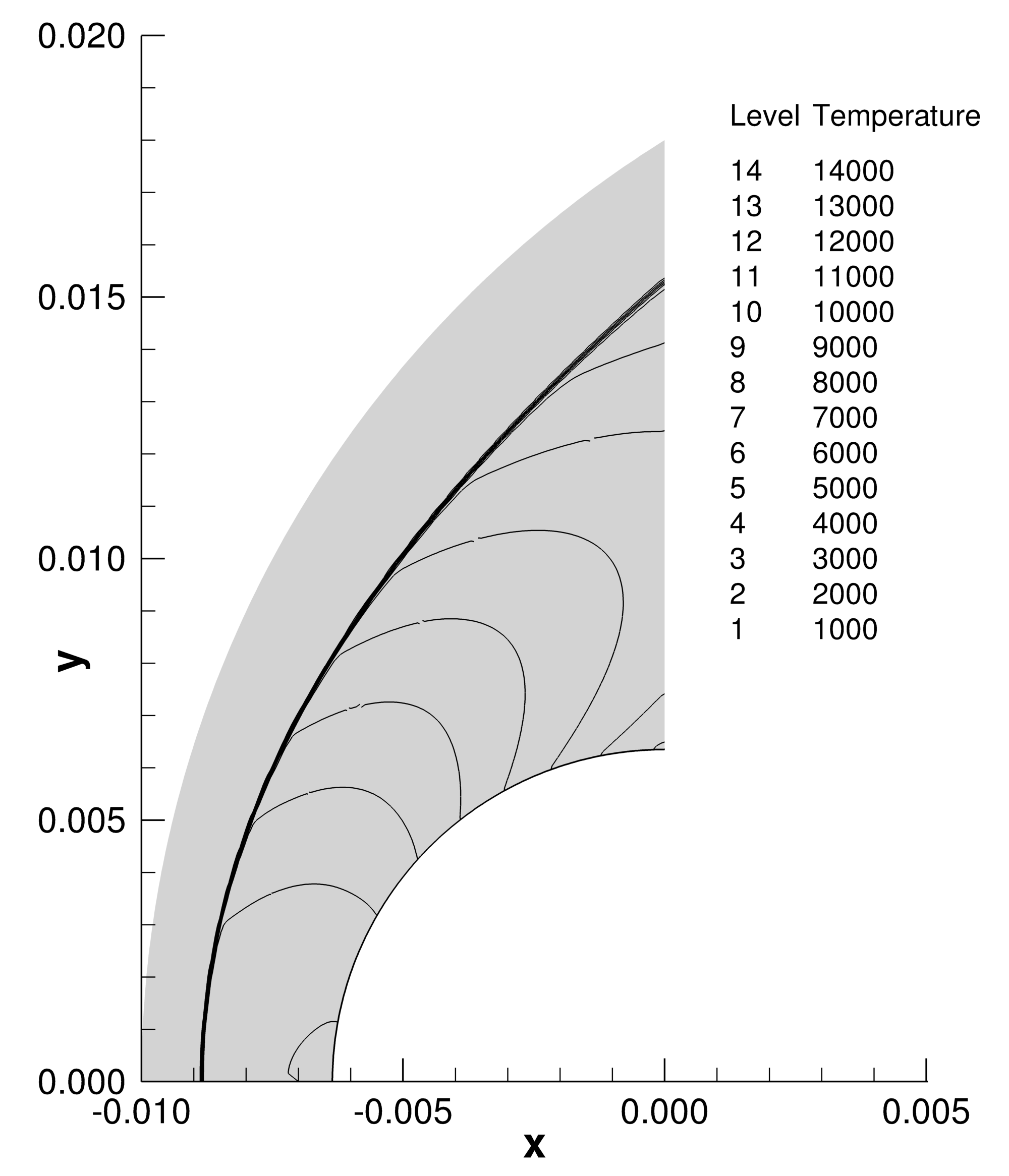,width=3.1cm}}\\ 
\setcounter{subfigure}{0}
\subfloat{\epsfig{figure=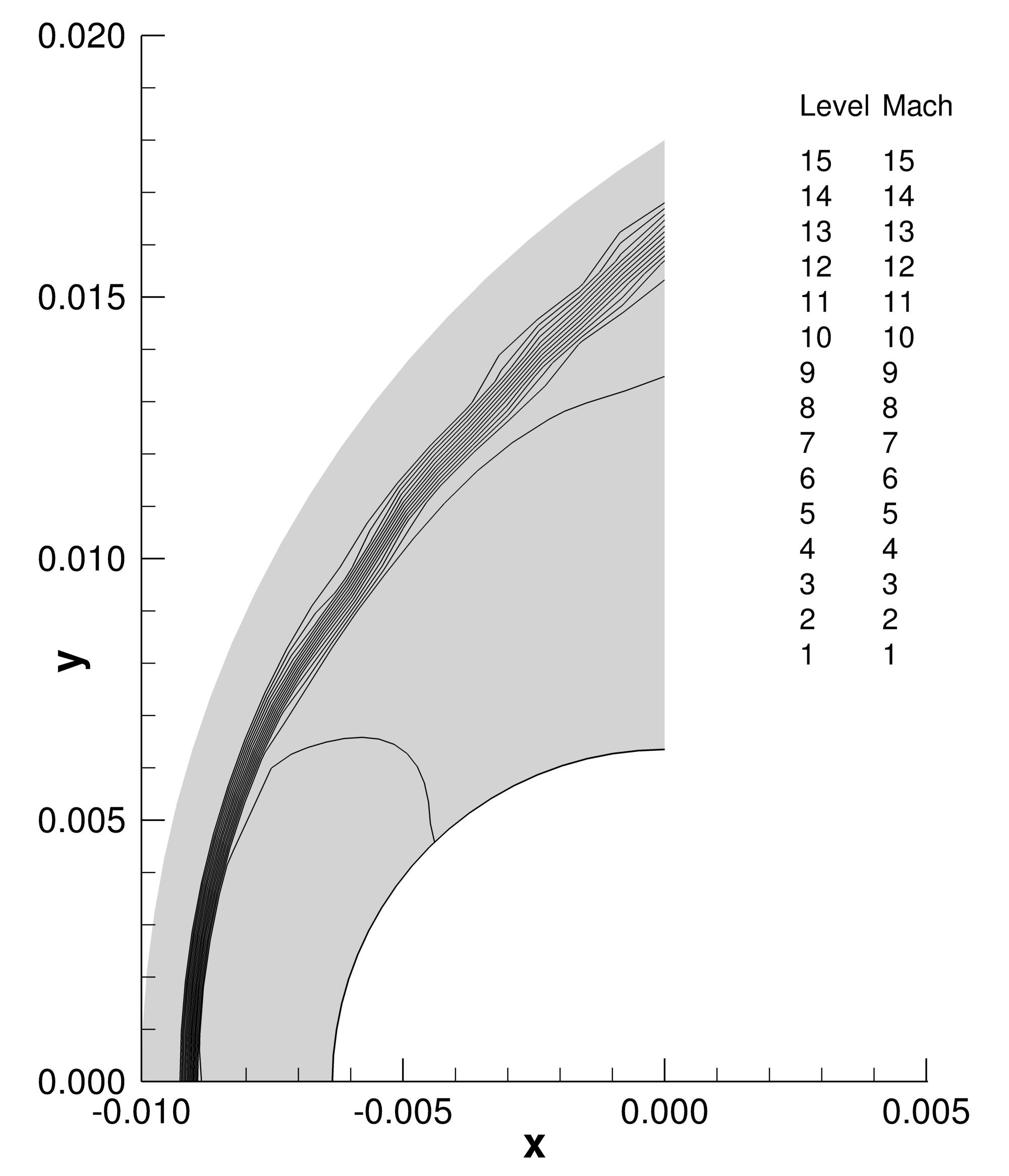,width=3.1cm} \hspace{-0.16cm}}
\subfloat{\epsfig{figure=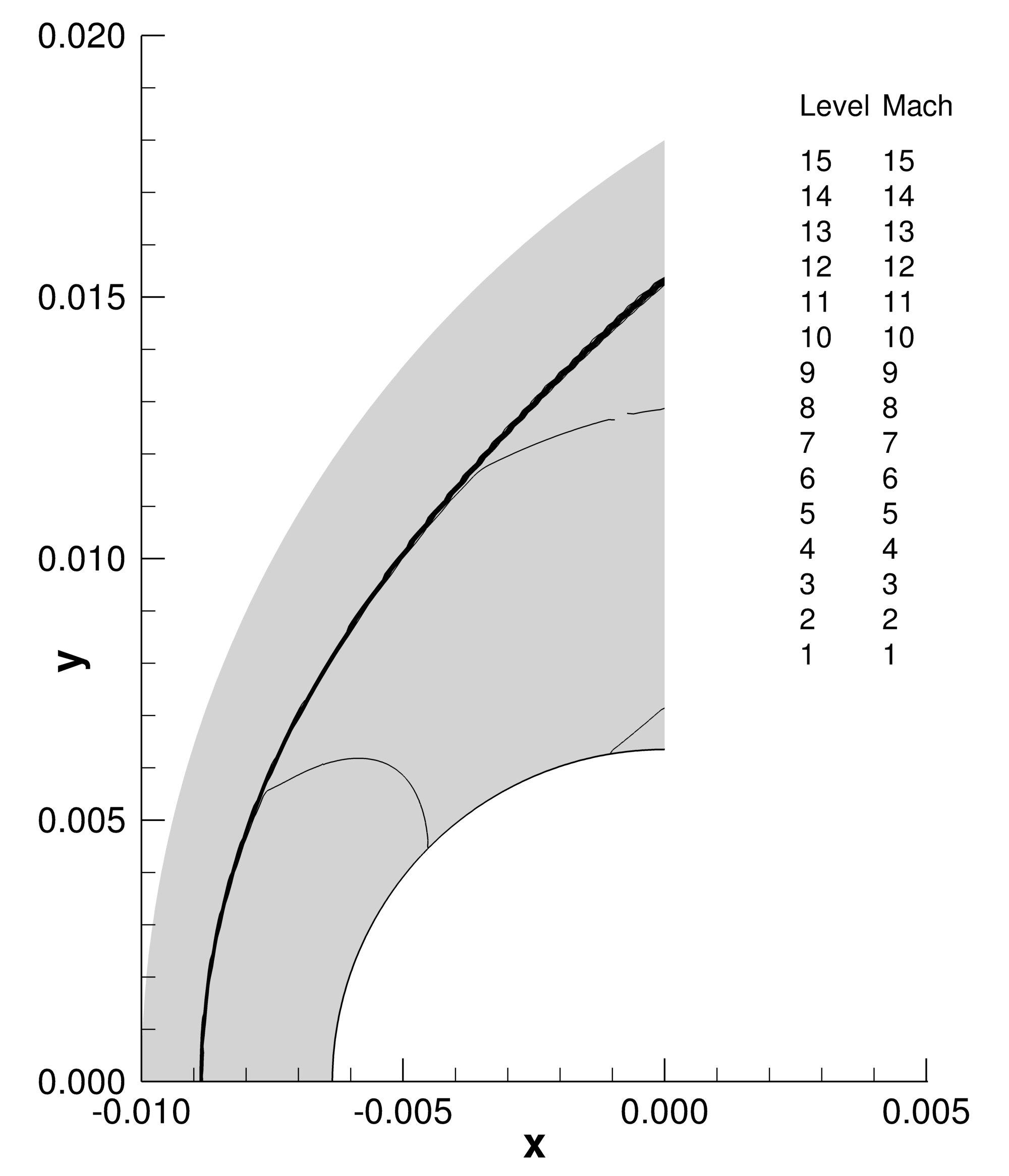,width=3.1cm} \hspace{-0.16cm}}
\subfloat{\epsfig{figure=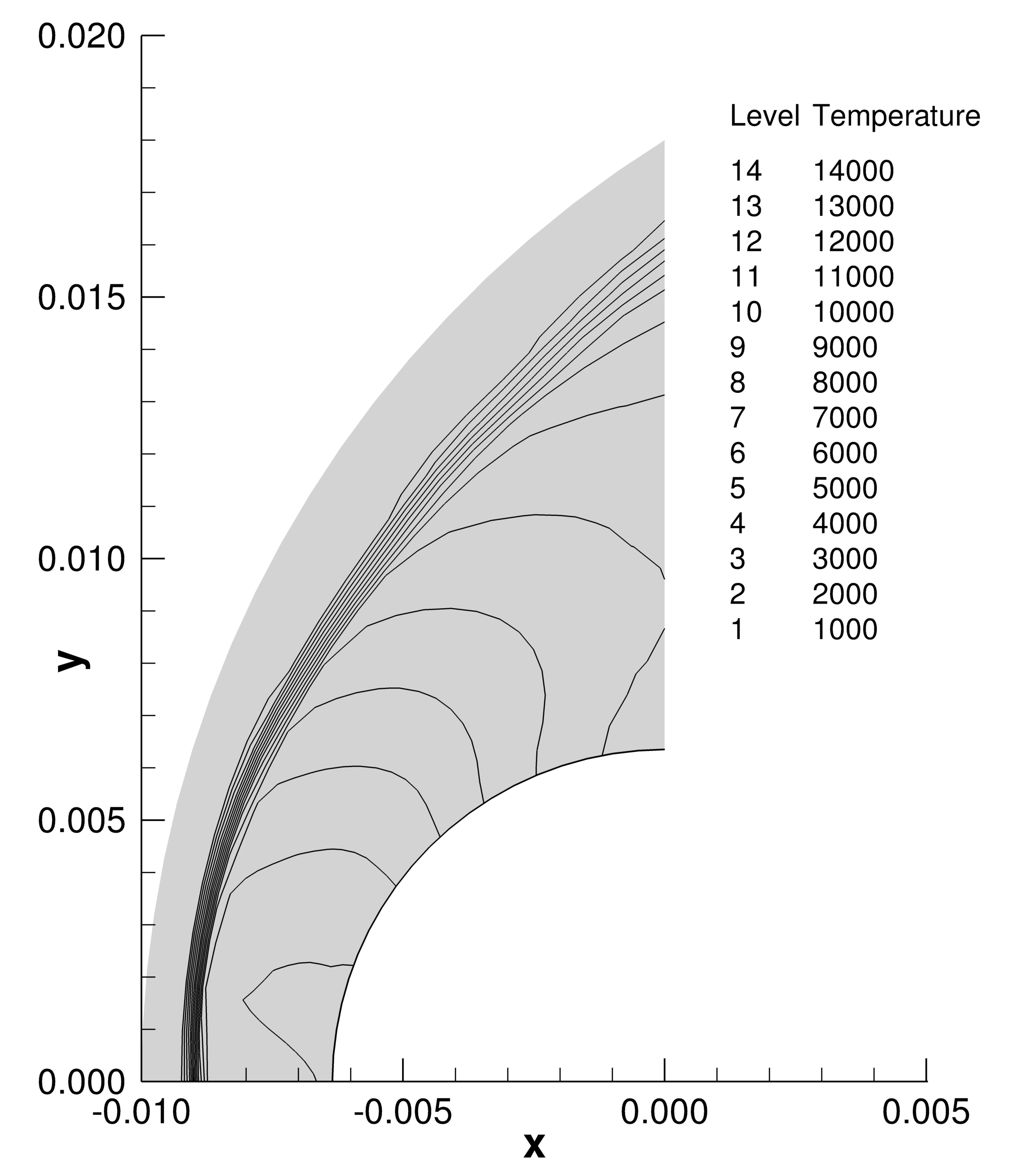,width=3.1cm} \hspace{-0.16cm}}
\subfloat{\epsfig{figure=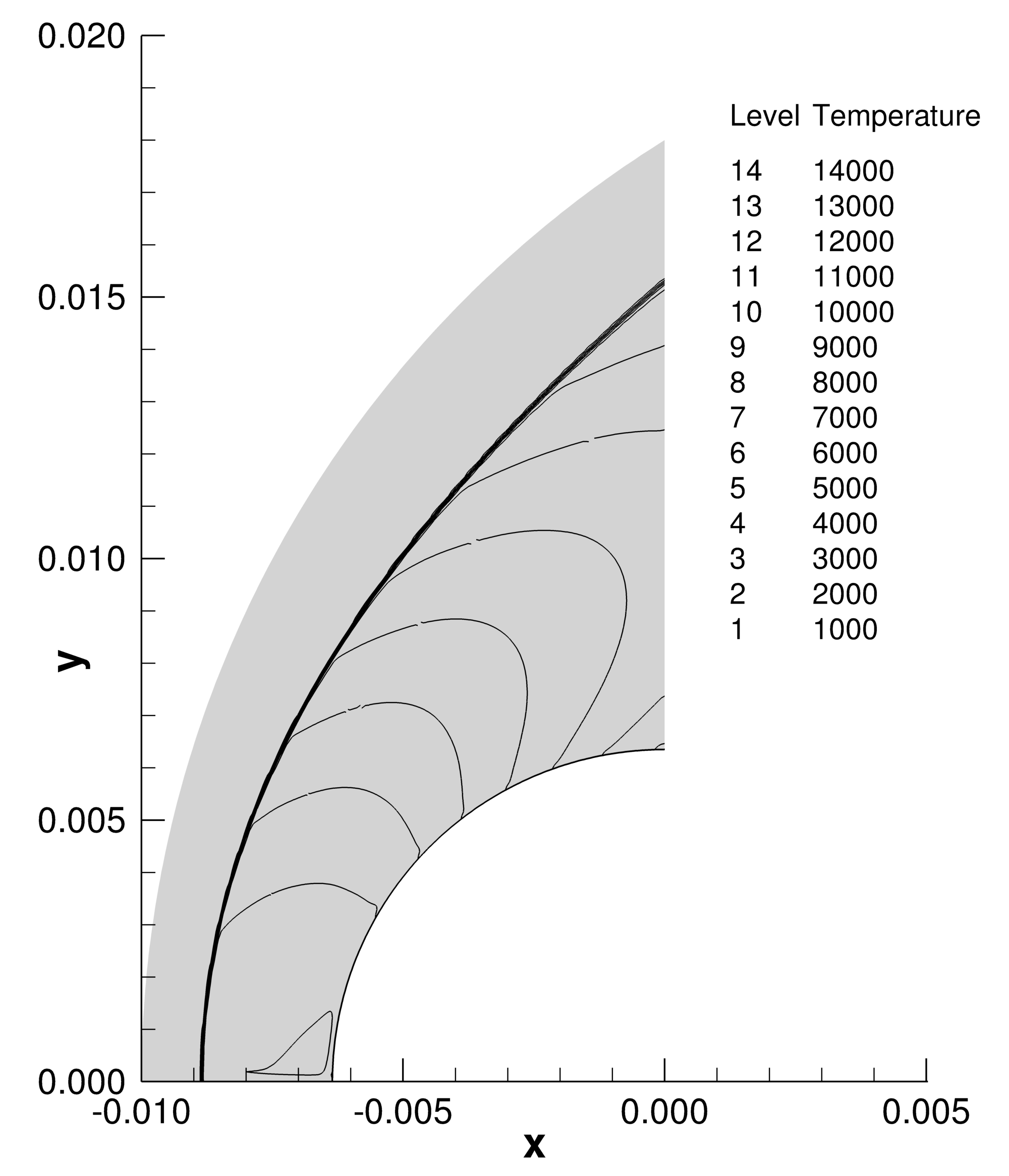,width=3.1cm}}\\ 
\caption{Hypersonic flow over Lobb's sphere: Mach number $M$ and temperature $\T$ contours obtained with the numerical fluxes \cref{eq:bouchut-flux} (top row), \cref{eq:hll_flux} (middle row), \cref{eq:godunov_flux} (bottom row), on two grids with $N=20\times20$ and $160\times160$ elements.}
\label{fig:solution_lobb_M15}
\end{center}
\end{figure}

\begin{figure}
\begin{center}
\setcounter{subfigure}{0}
\subfloat{\epsfig{figure=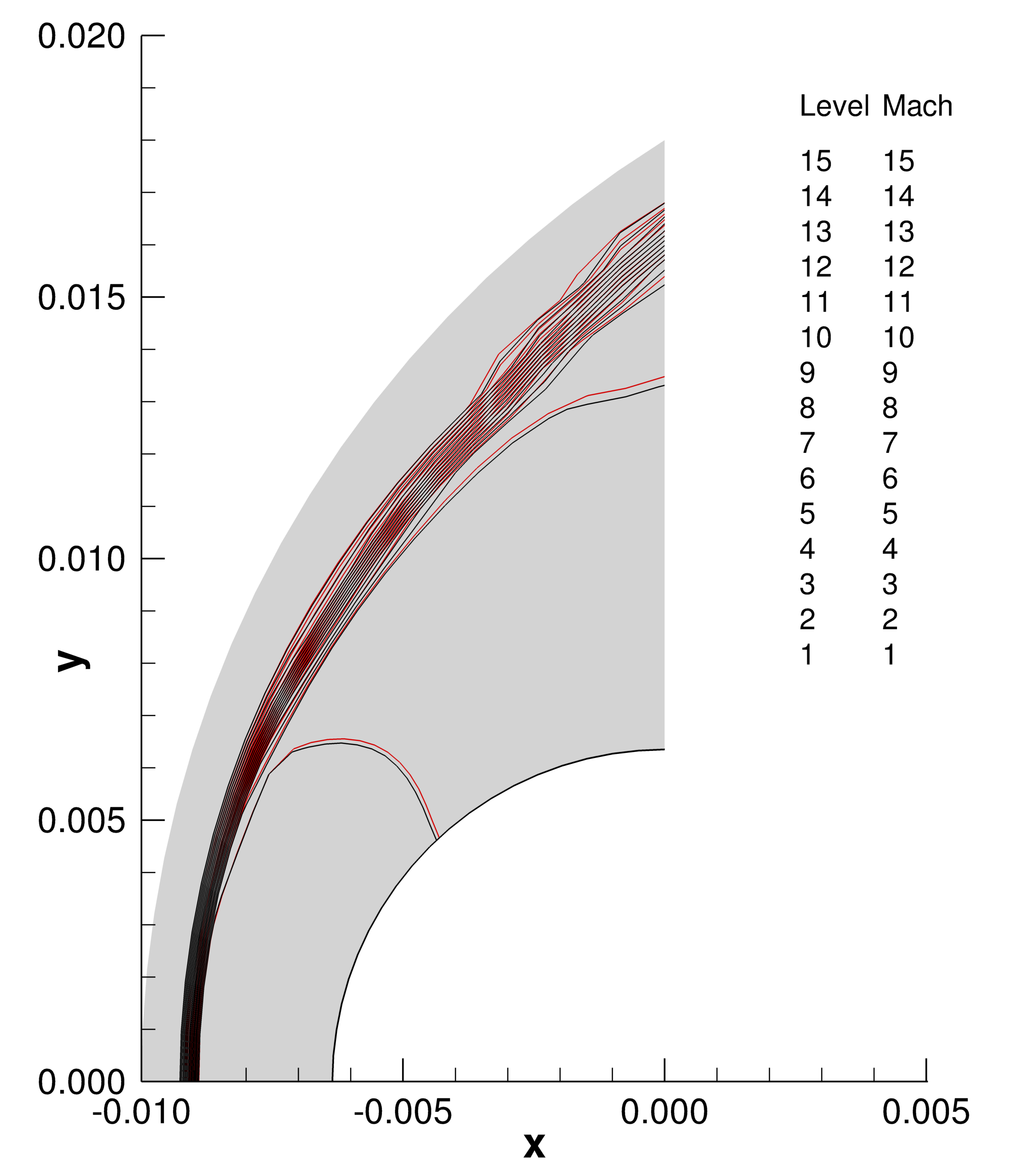,width=3.1cm} \hspace{-0.16cm}}
\subfloat{\epsfig{figure=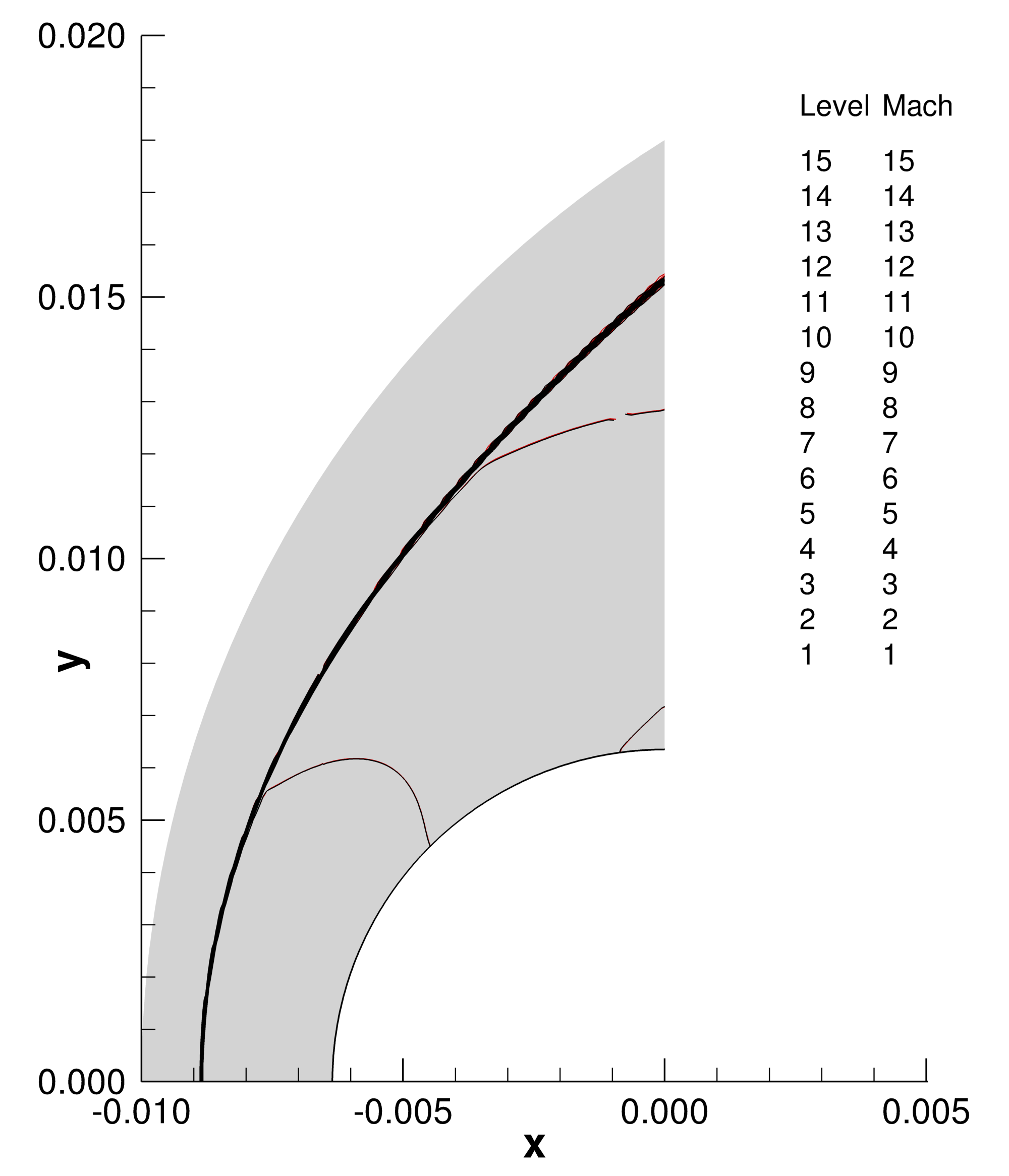,width=3.1cm} \hspace{-0.16cm}}
\subfloat{\epsfig{figure=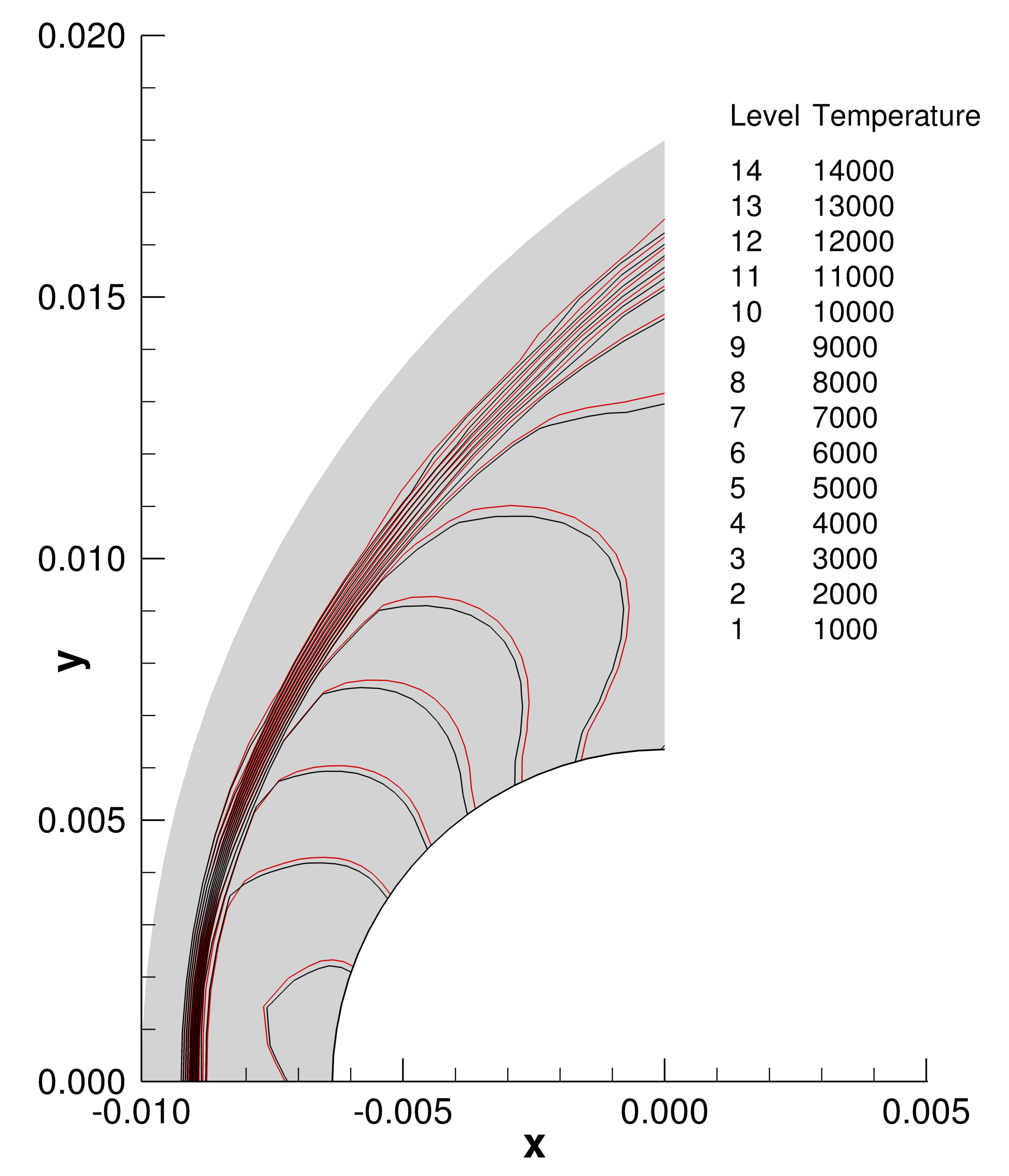,width=3.1cm} \hspace{-0.16cm}}
\subfloat{\epsfig{figure=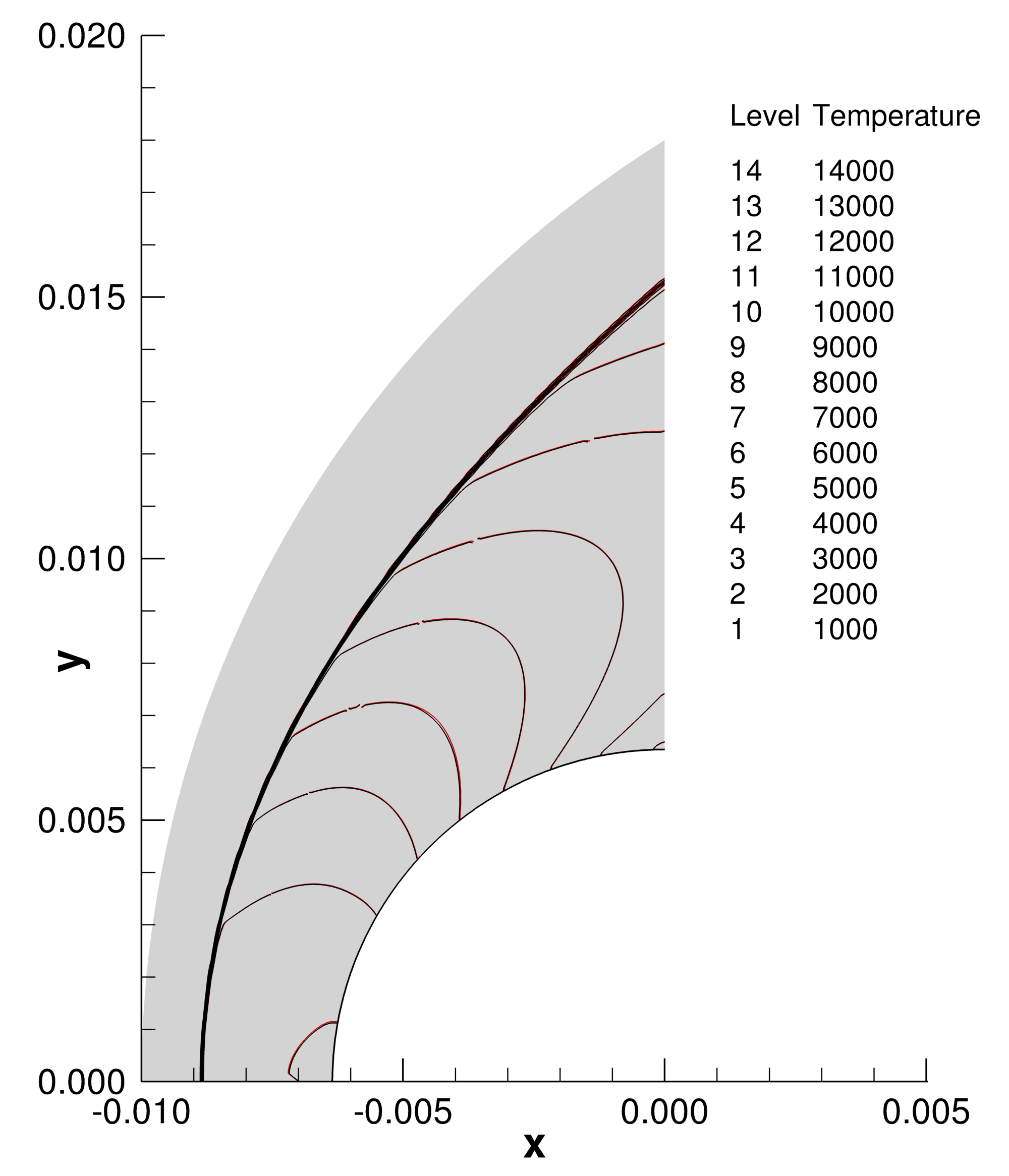,width=3.1cm}}\\ 
\caption{Hypersonic flow over Lobb's sphere: Mach number $M$ and temperature $\T$ contours obtained with the HLL numerical flux \cref{eq:hll_flux} (red), and the HLL-PG flux considering an equivalent monocomponent perfect gas (black) on two grids with $N=20\times20$ and $160\times160$ elements.}
\label{fig:solution_lobb_hll_hllms}
\end{center}
\end{figure}

\begin{figure}
\begin{center}
\subfloat{\epsfig{figure=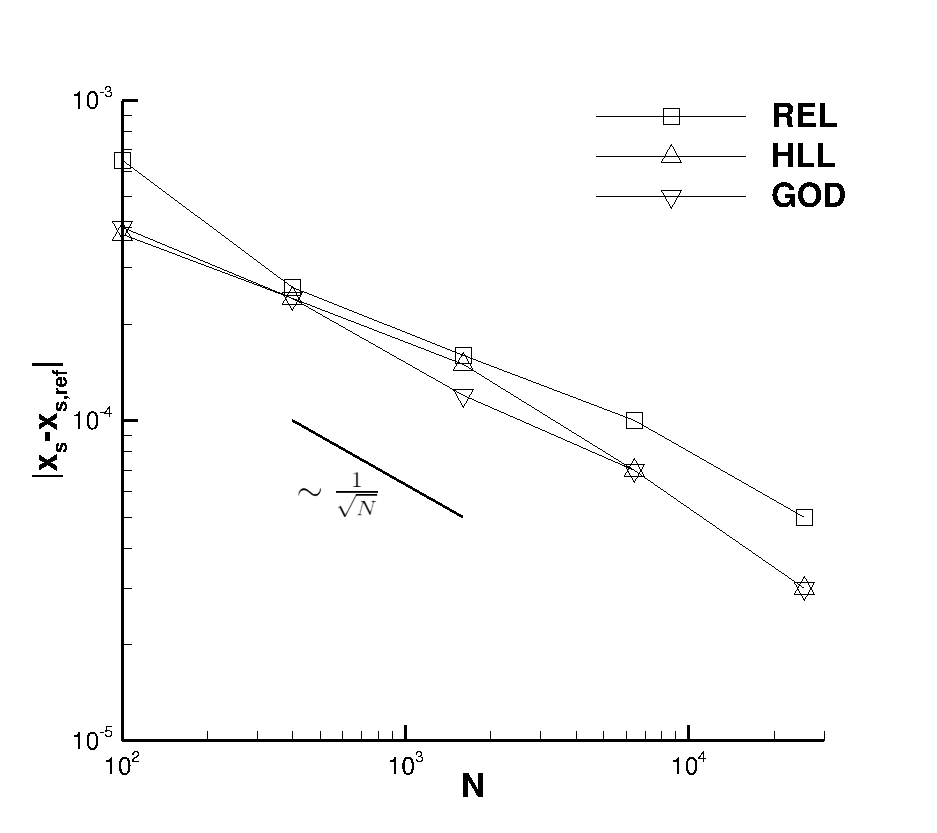,width=5cm}}
\caption{Hypersonic flow over Lobb's sphere: convergence of the shock position in the symmetry plane $y=0$ under mesh refinement.}
\label{fig:shock_position_lobb_M15}
\end{center}
\end{figure}

\subsection{Hypersonic flow over a double cone}

We finally consider the 2D hypersonic flow over a double cone with angles $25$  and $55$ deg. adapted from \cite{Druguet_et_al_05,KNIGHT20128} and made of molecular and atomic nitrogen with mass fractions $Y_{N_2}=0.99$, $Y_{N}=0.01$. The freestream Mach number is $M_\infty=11.3$ with $\rho_\infty=1.34\times10^{-3}$kg/m$^3$ and $\T_\infty=303$K. The freestream vibration temperature of the molecular nitrogen is taken at $\T_{N2_\infty}=3085$K. We use a series of five unstructured grids (see \cref{fig:2D_meshes}). A symmetry condition is imposed at the bottom boundary.

\begin{figure}
\begin{center}
%
\subfloat{\epsfig{figure=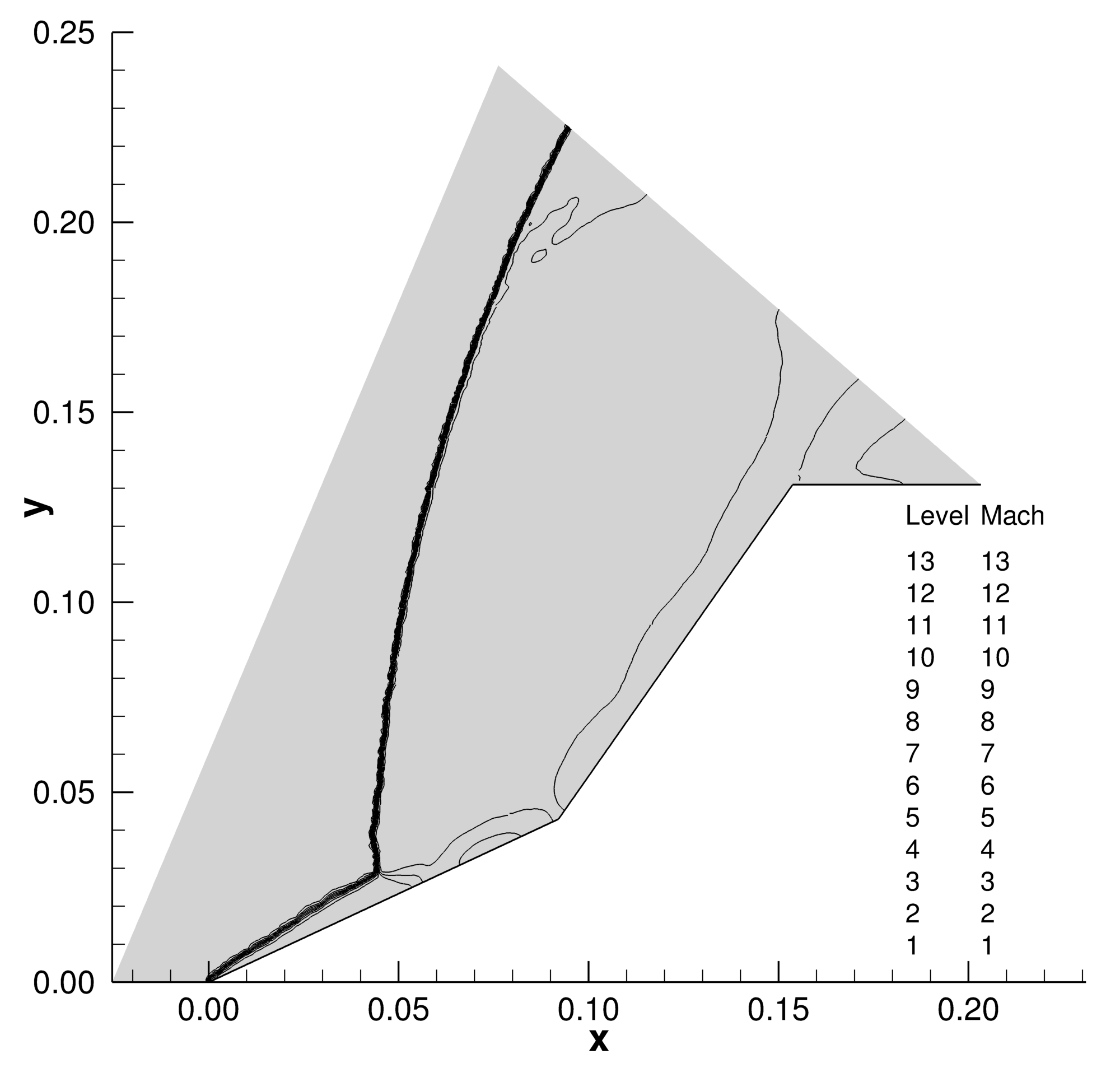,width=4.0cm}} 
\subfloat{\epsfig{figure=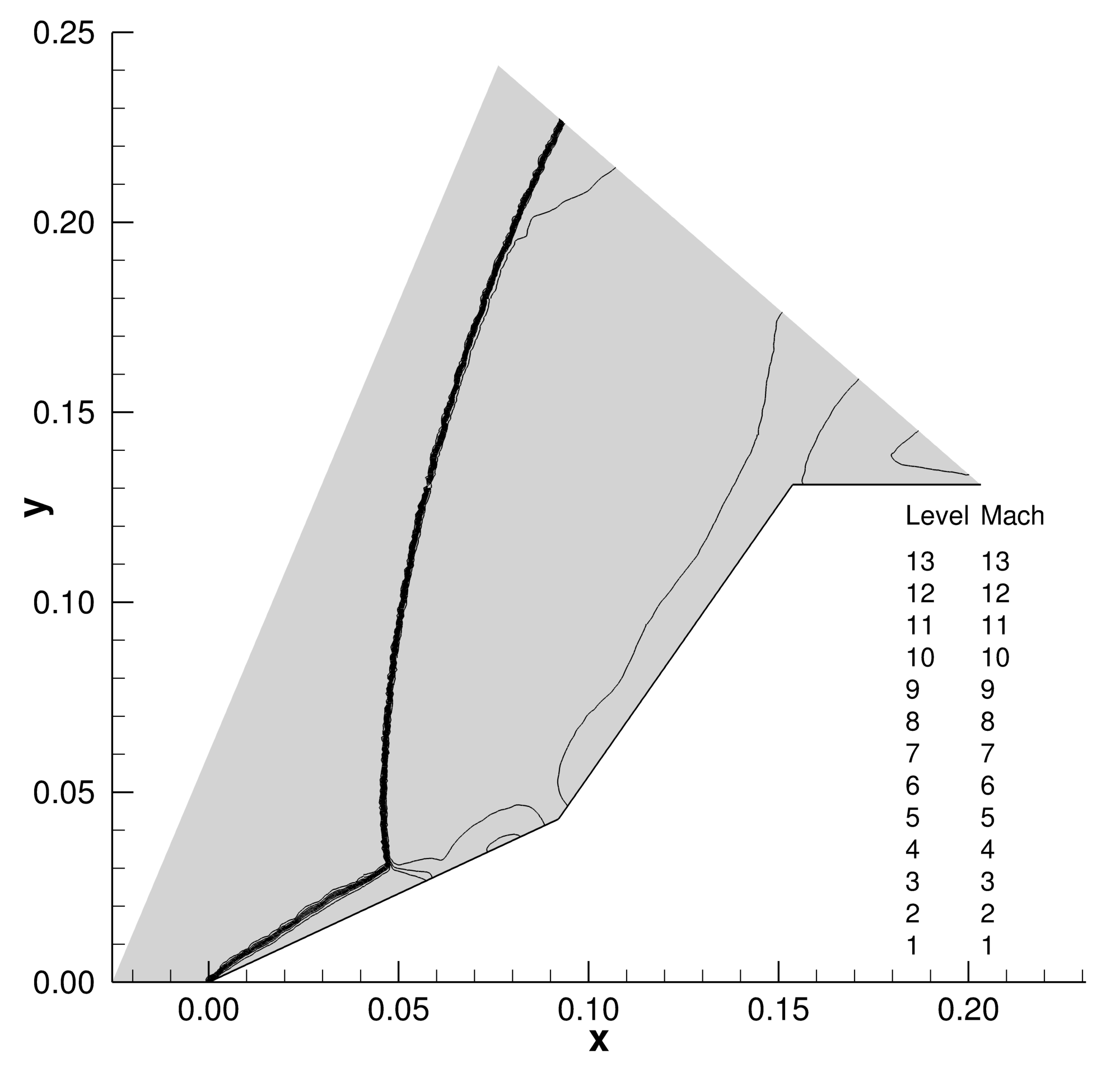,width=4.0cm}} 
\subfloat{\epsfig{figure=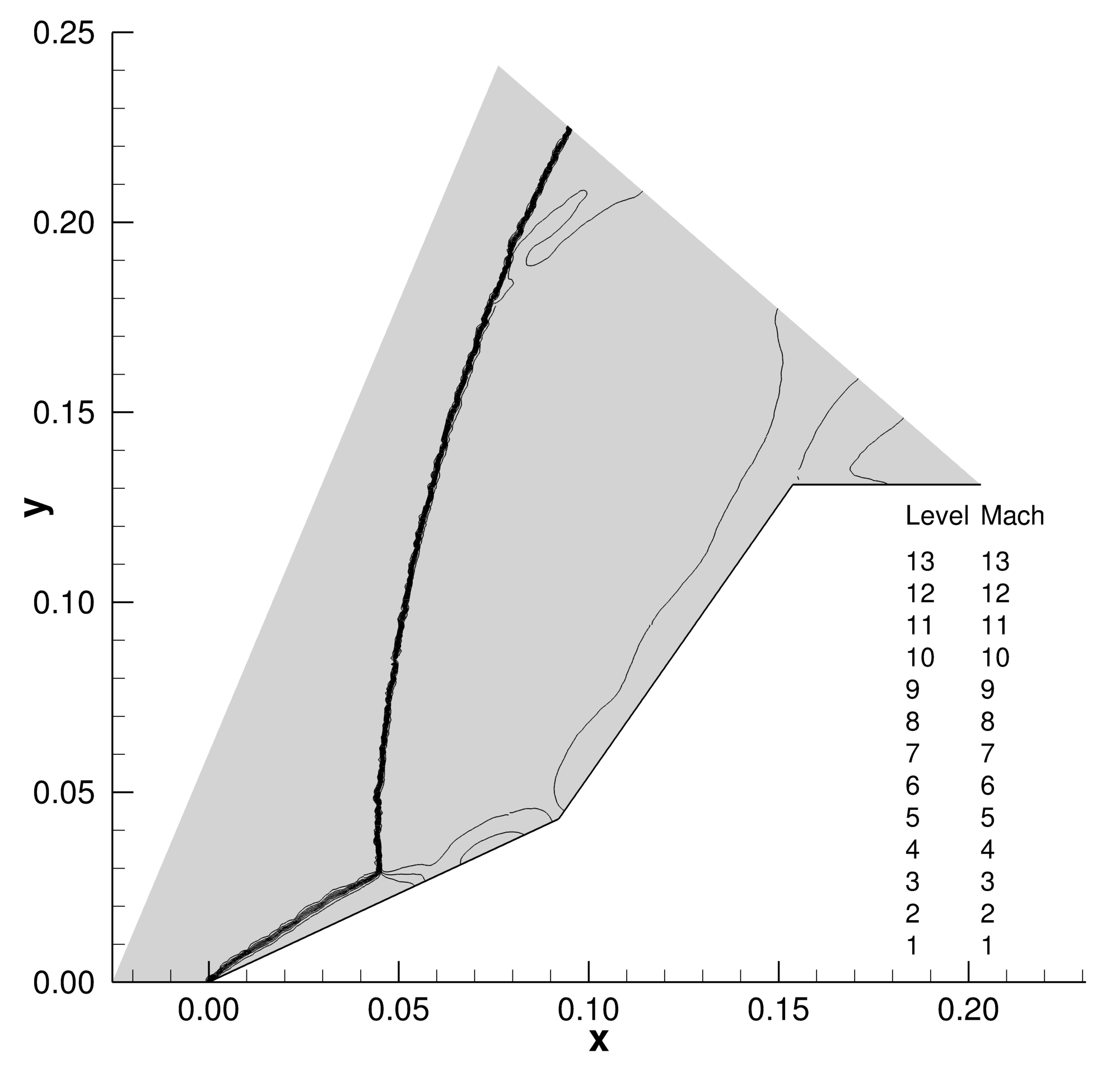,width=4.0cm}} \\
\setcounter{subfigure}{0}
\subfloat[REL]{\epsfig{figure=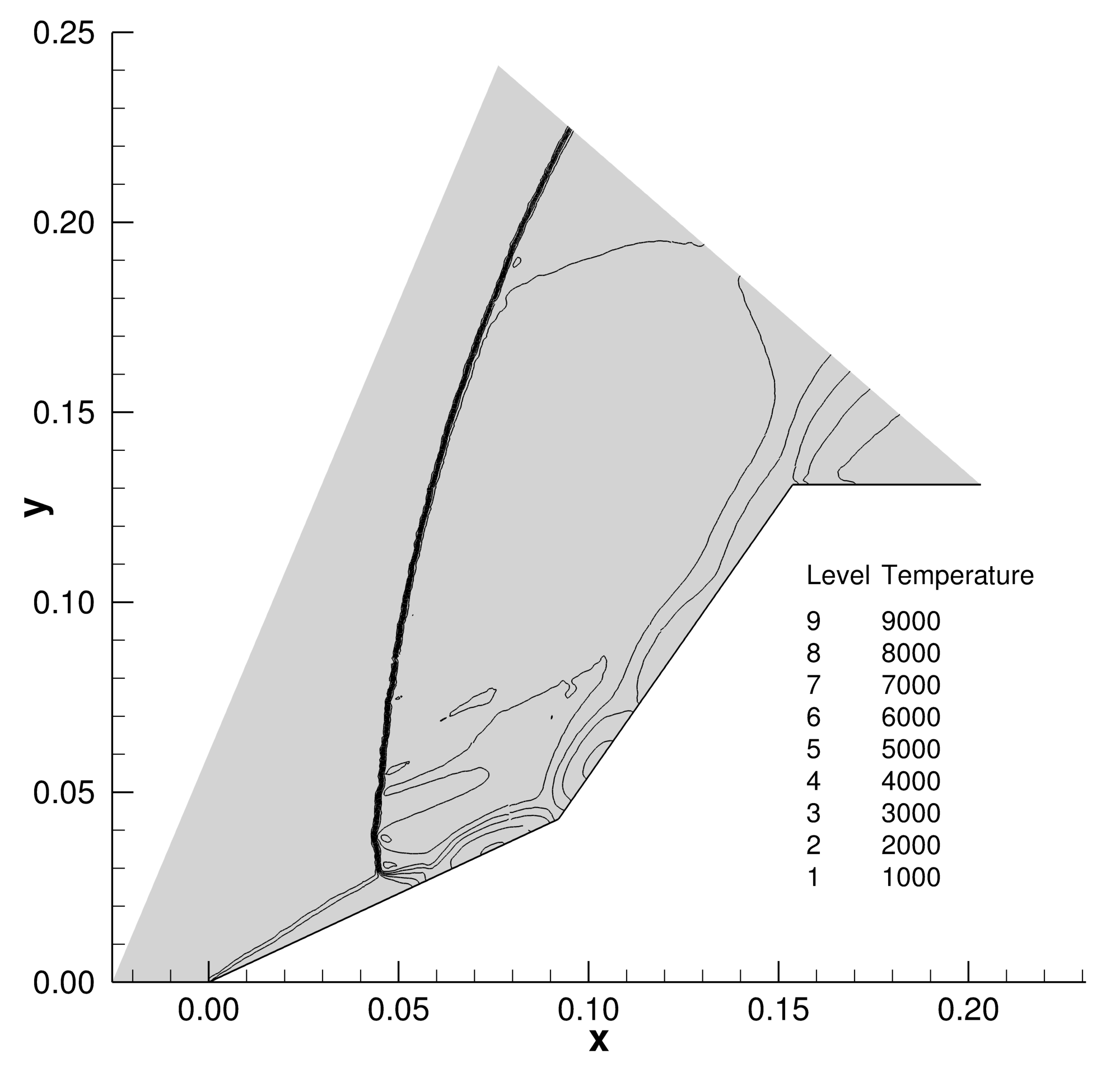,width=4.0cm}}
\subfloat[HLL]{\epsfig{figure=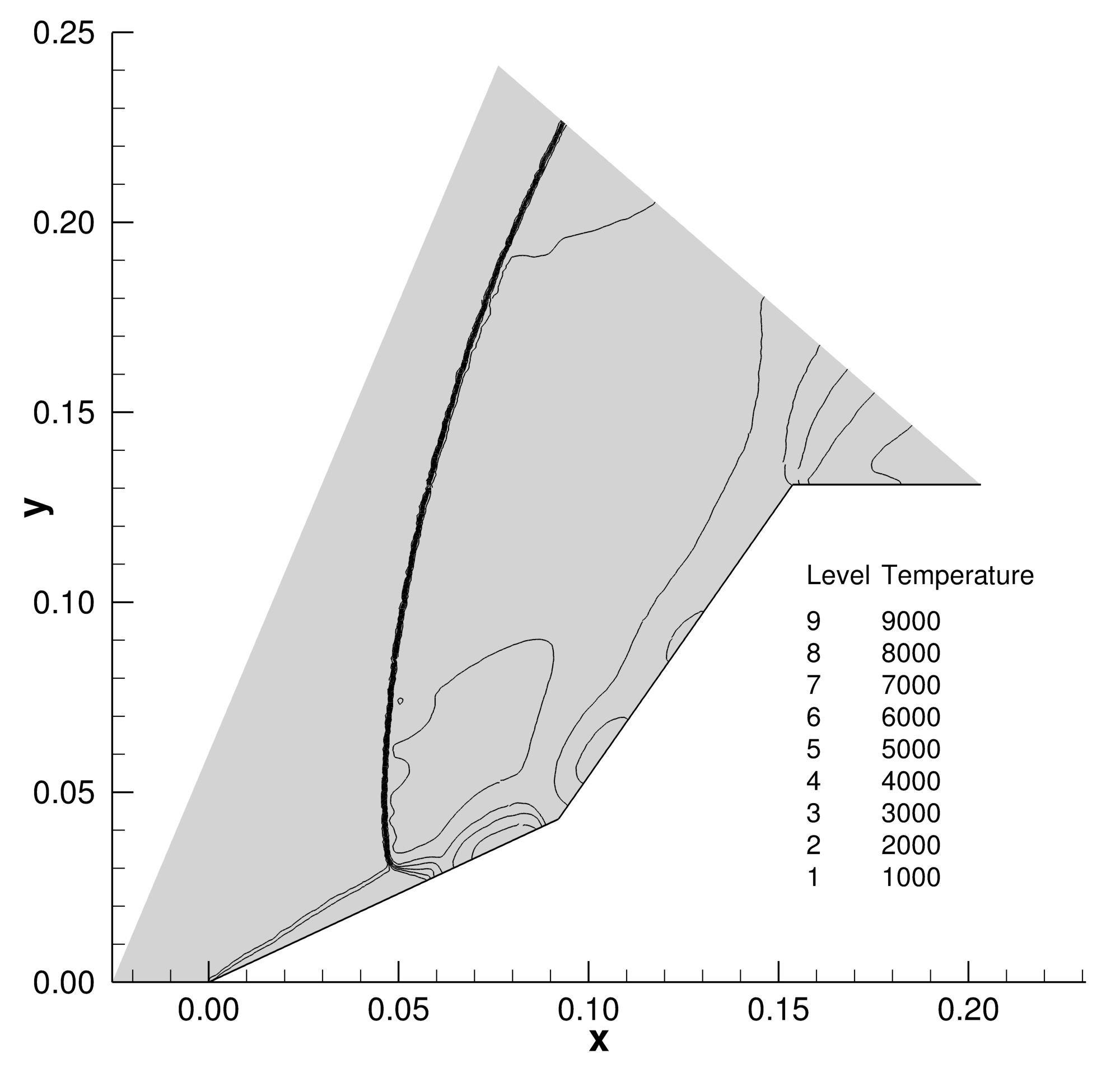,width=4.0cm}}
\subfloat[GOD]{\epsfig{figure=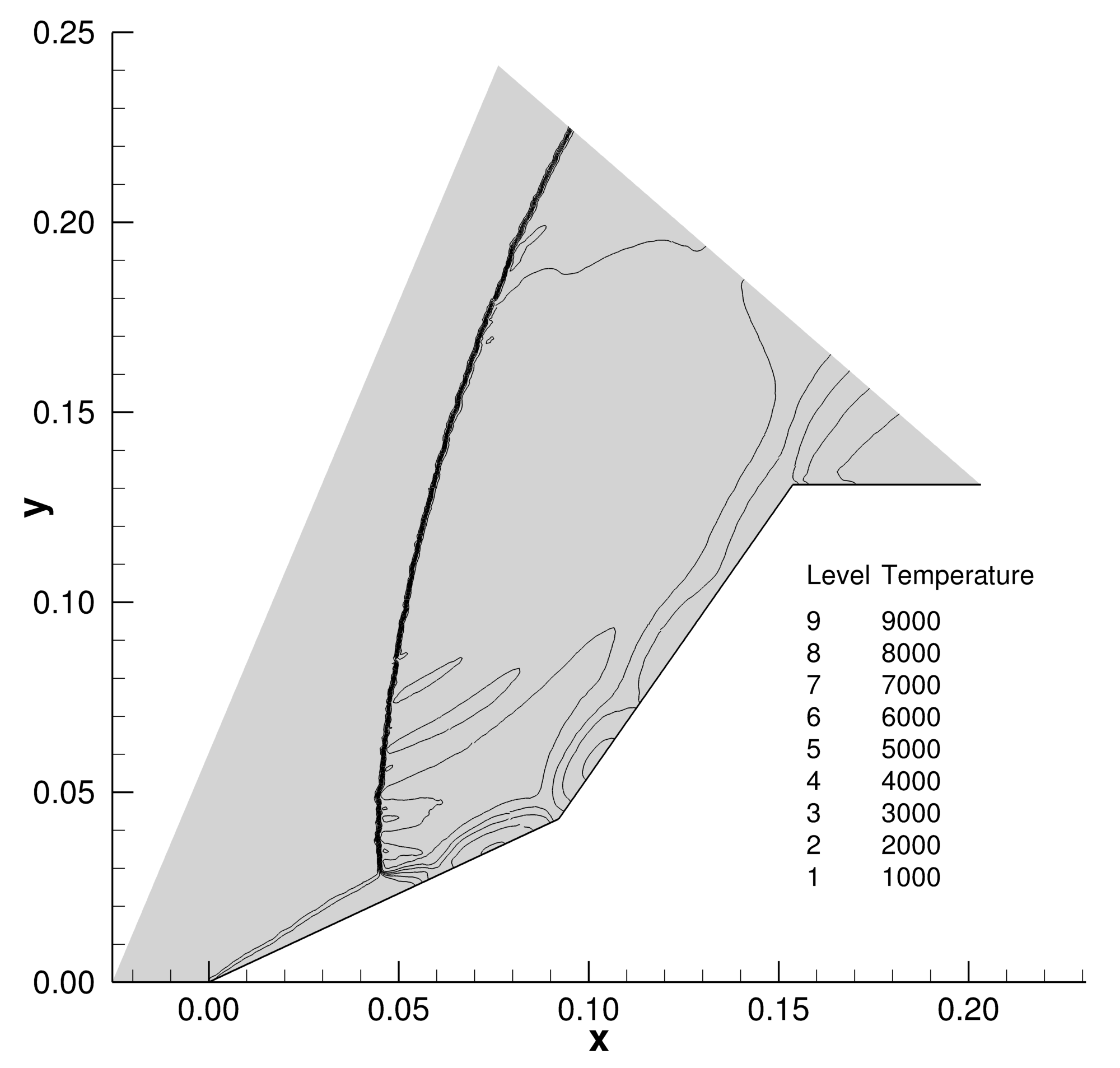,width=4.0cm}}
\caption{Hypersonic flow over a double cone: Mach number $M$ (top row) and temperature $\T$ (bottom row) contours obtained with the numerical fluxes \cref{eq:bouchut-flux} (REL), \cref{eq:hll_flux} (HLL), and \cref{eq:godunov_flux} (GOD) on a fine grid with $N=52,237$ elements.}
\label{fig:solution_dble_cone}
\end{center}
\end{figure}

\begin{figure}
\begin{center}
\setcounter{subfigure}{0}
\subfloat{\epsfig{figure=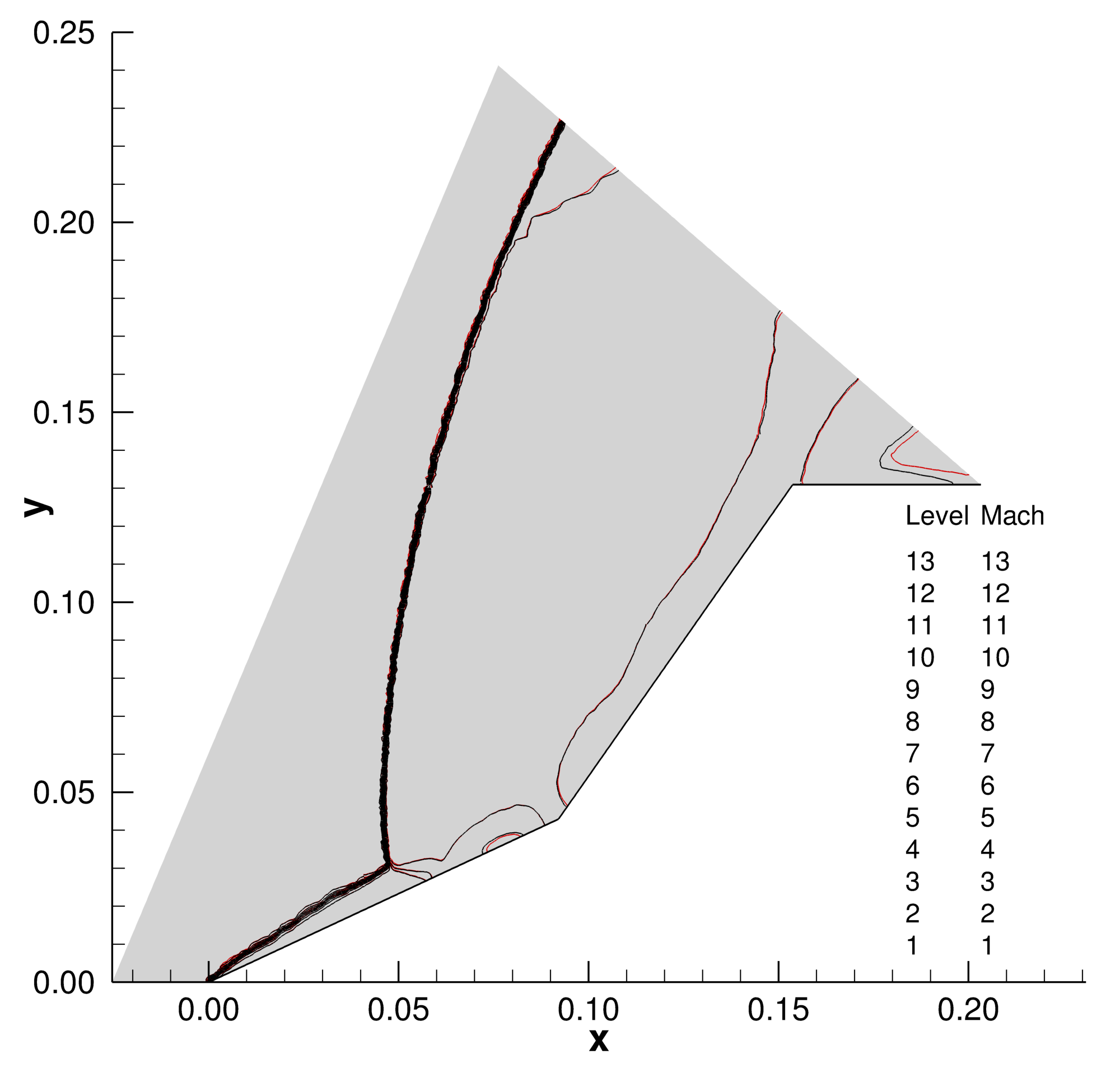,width=4.0cm}}
\subfloat{\epsfig{figure=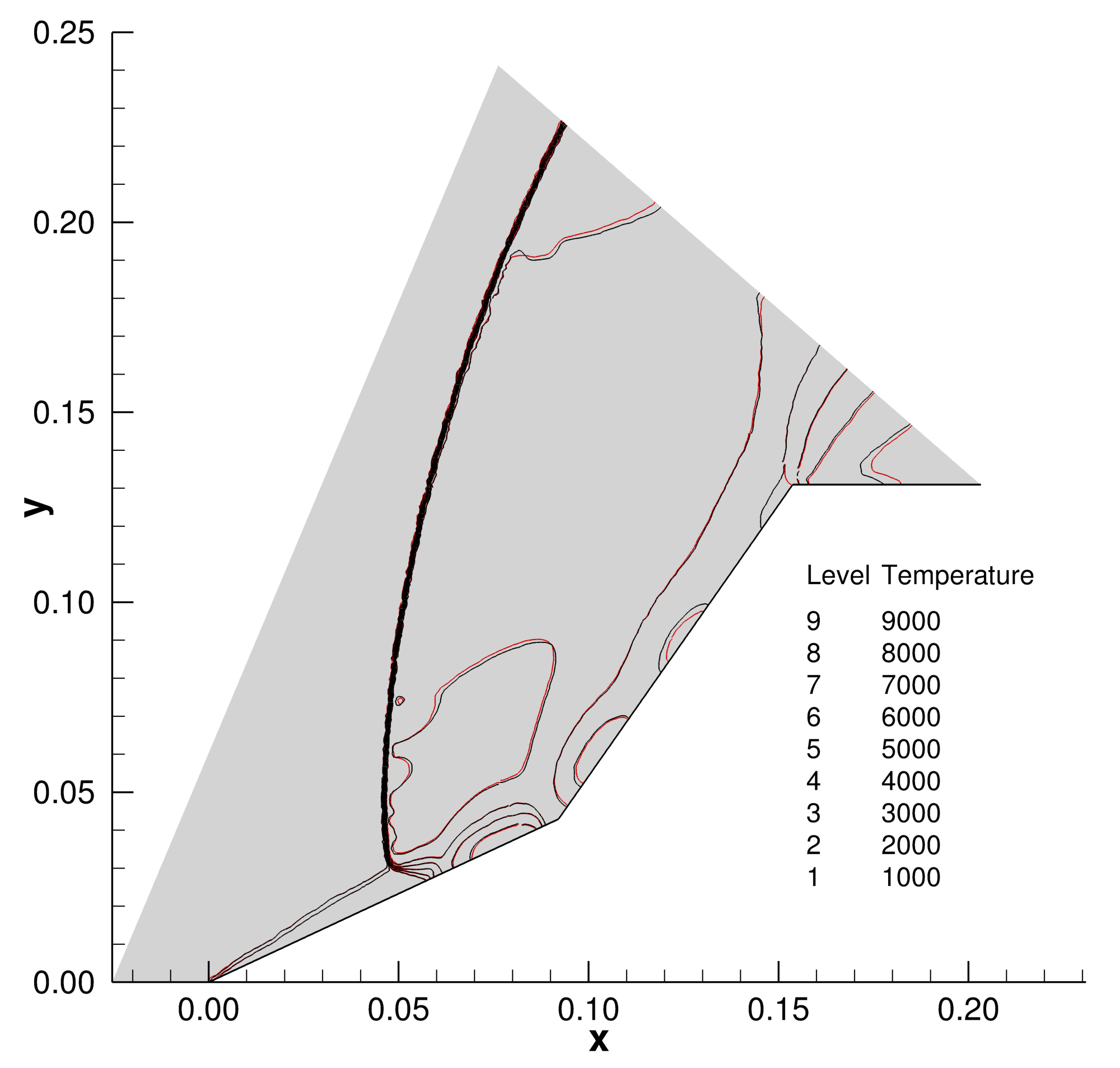,width=4.0cm}}
\caption{Hypersonic flow over a double cone: Mach number $M$ (left) and temperature $\T$ (right) contours obtained with the numerical flux \cref{eq:hll_flux} (red), and the HLL-PG flux considering an equivalent perfect gas (black) on a fine grid with $N=52,237$ elements.}
\label{fig:solution_dble_cone_hll_hllms}
\end{center}
\end{figure}

Contours of Mach number and translation-rotation temperature obtained with the three different schemes on the second finest mesh are displayed in \cref{fig:solution_dble_cone}. Compared to references \cite{Druguet_et_al_05,KNIGHT20128}, we observe a strong overestimation of the distance of the bow shock to the wall due to the absence of chemical reactions. However, the results with the three schemes are in good agreement. As done for the previous configuration, we compare in \cref{fig:solution_dble_cone_hll_hllms} the obtained solution to that corresponding to the use of an equivalent perfect gas with adiabatic exponent $\gamma = 1.4032$ (HLL-PG) on a fine grid. Once again a very good agreement is obtained. Finally in \cref{fig:press_dble_cone}, we display the pressure distribution at the wall obtained with the schemes on the five grids. The first pressure peak corresponds to the reflexion of the separated shock at the wall, while the second peak corresponds to rapid pressure variations due to the geometrical transition between the cones. We observe convergence of the solution as the mesh is refined and a close agreement between results from the three schemes on the finest grids.

\begin{figure}
\begin{center}
\subfloat[REL]{\epsfig{figure=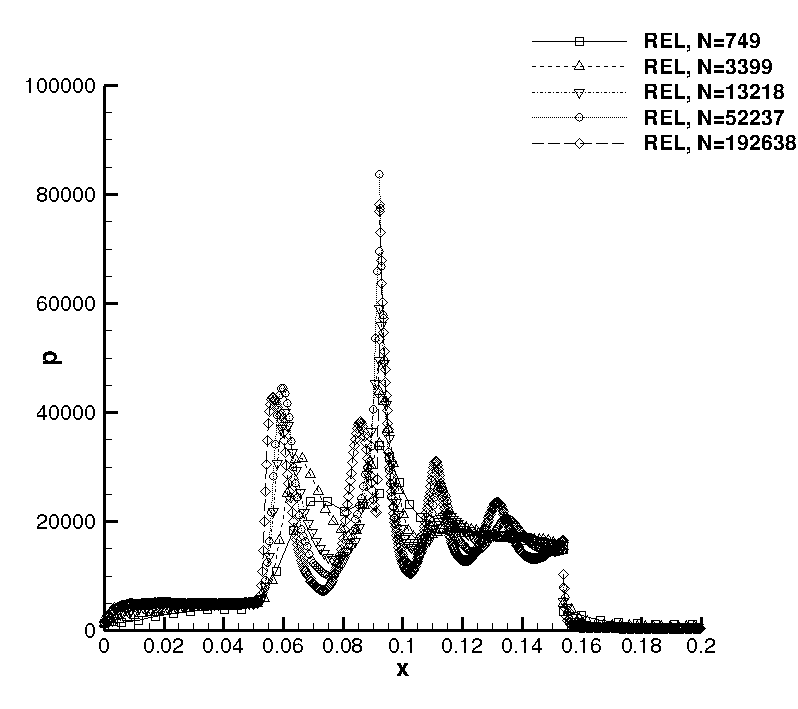,width=4.0cm}}
\subfloat[HLL]{\epsfig{figure=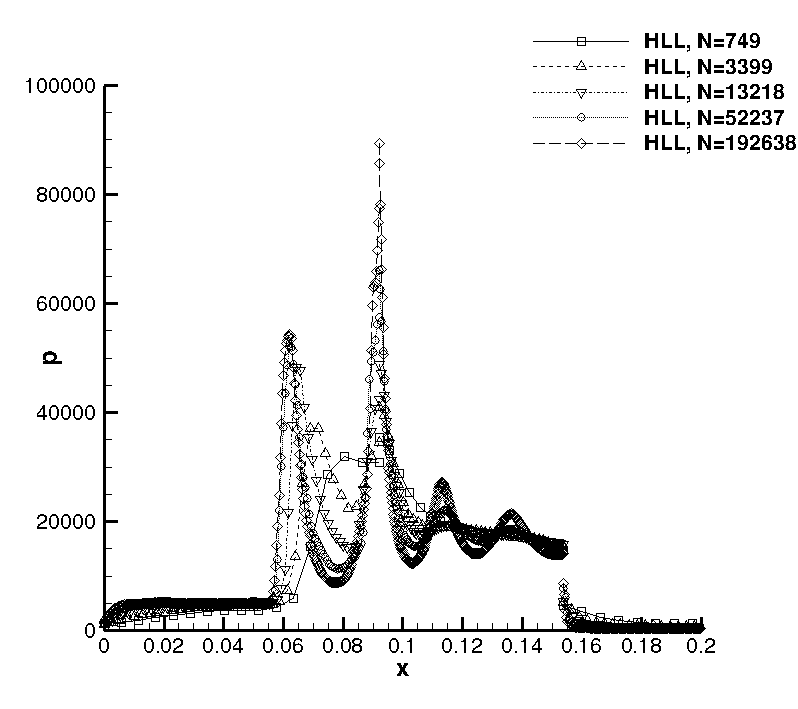,width=4.0cm}}
\subfloat[GOD]{\epsfig{figure=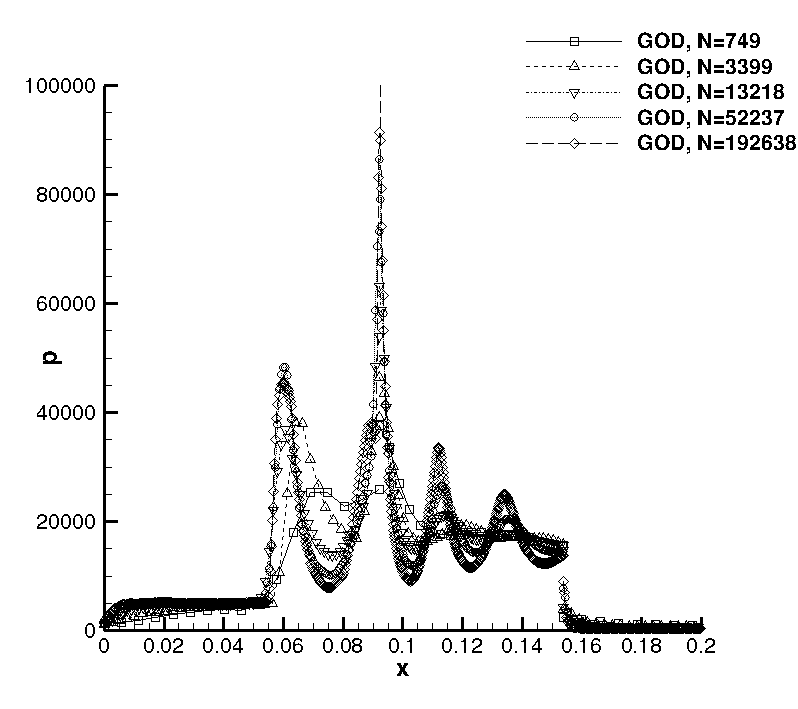,width=4.0cm}}
\caption{Hypersonic flow over a double cone: pressure distribution at the wall obtained with the numerical fluxes \cref{eq:bouchut-flux} (REL), \cref{eq:hll_flux} (HLL), and \cref{eq:godunov_flux} (GOD) on a series of five grids with $N$ elements.}
\label{fig:press_dble_cone}
\end{center}
\end{figure}

%
%
\section{Concluding remarks}\label{sec:conclusions}

We introduce a general framework to design finite volume schemes for the compressible multicomponent Euler equations in thermal nonequilibrium. The framework allows to define a numerical scheme for its discretization from a scheme for the discretization of the monocomponent polytropic gas dynamics through a simple linear formula. Moreover, the numerical scheme inherits the properties of the scheme for the polytropic gas dynamics under a subcharacteristic condition on the adiabatic exponent of the polytropic gas.

This framework relies on the extension of the relaxation of energy for the gas dynamics equations \cite{coquel_perthame_98} to the model under consideration in this work. Three different numerical fluxes are constructed with this framework the polytropic Godunov exact Riemann solver \cite{godunov_59}, HLL numerical flux \cite{hll_83}, and pressure-based relaxation solver \cite{bouchut_04}. They are assessed through numerical simulations of flows in one and two space dimensions with discontinuous solutions and complex wave interactions. The results highlight robustness, nonlinear stability, convergence of the present method, as well as similar performances of the three schemes.

Other numerical fluxes may be deduced from this framework. We also stress that the numerical fluxes designed in this framework can be used as building blocks in the general framework of conservative elementwise flux differencing schemes \cite{fisher_carpenter_13} and future work will consider the use of discontinuous Galerkin schemes for the discretization of the compressible multicomponent Euler equations in thermal nonequilibrium.

%
%
\appendix

\section{Convexity of the entropy for the energy relaxation system}\label{sec:appendix_relax_entropy_convexity}

The object of this appendix is the proof of \cref{th:entropy-convexity}. Without loss of generality we define $n_s$ in the mapping \cref{eq:mapping_Y_Yb} as the one corresponding to one species that satisfies $r_{n_s}=\min_\alpha r_\alpha$. To prove that $\rho\zeta({\bf w})$ is convex it is sufficient to prove that $\zeta(\Yb,\tau,e_r,e_s,{\bf e}_v)$ is convex from \cref{th:equiv_convexity} and, from \cref{eq:zeta-delta-zeta}, we rewrite $\zeta$ as
\begin{align*}
  \zeta(\Yb,\tau,e_r,e_s,{\bf e}_v) &= C_{v_t}(\Yb)\ln\big((\gamma-\gamma(\Yb))C_{v_t}(\Yb)\big) + \sum_{\alpha=1}^{n_s}r_\alpha Y_\alpha\ln Y_\alpha + \tfrac{r(\Yb)}{\gamma-1}\ln\tfrac{\gamma(\Yb)-1}{\gamma-\gamma(\Yb)} \\
	&-r(\Yb)\ln\tau - \tfrac{r(\Yb)}{\gamma-1}\ln e_r - \tfrac{(\gamma-\gamma(\Yb))C_{v_t}(\Yb)}{\gamma-1}\ln e_s + l(\Yb) - \s_v({\bf Y},{\bf e}_v),
\end{align*}
\noindent with $l(\Yb)=-\sum_{\alpha}Y_\alpha C_{v_\alpha}^t\ln C_{v_\alpha}^t - C_{v_t}(\Yb)\ln(\gamma-1)$, with $Y_{n_s}=1-\sum_{\alpha<n_s}Y_\alpha$, linear in $\Yb$. 

Introducing the short notations $\partial_kr\equiv\partial_{Y_k}r(\Yb)$, $\partial_kC_{v_t}\equiv\partial_{Y_k}C_{v_t}(\Yb)$, and $\partial_k\gamma\equiv\partial_{Y_k}\gamma(\Yb)$, the Hessian $\boldsymbol{\cal H}_\zeta(\Yb,\tau,e_r,e_s,{\bf e}_v)$ of $\zeta$ reads
\begin{equation}\label{eq:hessian_zeta}
  \begin{pmatrix} \big(\partial_{kl}^2\zeta-\delta_{k,l}\psi_k\tfrac{e_k^{v^2}\s_k^{v''}\!\!(e_k^v)}{Y_k}\big)_{kl} & \big(\tfrac{-\partial_kr}{\tau}\big)_{k} & \big(\tfrac{-\partial_kr}{(\gamma-1)e_r}\big)_{k} & \big(\tfrac{\partial_kr-(\gamma-1)\partial_kC_{v_t}}{(\gamma-1)e_s}\big)_{k} & \big(\tfrac{e_k^{v}\s_k^{v''}\!\!(e_k^v)}{Y_k}\big)_{1\leq k\leq n_d}\\
    \big(\tfrac{-\partial_lr}{\tau}\big)_{l} & \tfrac{r(\Yb)}{\tau^2} & 0 & 0 & 0\\
    \big(\tfrac{-\partial_lr}{(\gamma-1)e_r}\big)_{l} & 0 & \tfrac{r(\Yb)}{(\gamma-1)e_r^2} & 0 & 0 \\
    \big(\tfrac{\partial_lr-(\gamma-1)\partial_lC_{v_t}}{(\gamma-1)e_s}\big)_{l} & 0 & 0 & \tfrac{\gamma-\gamma(\Yb)}{\gamma-1}\tfrac{C_{v_t}(\Yb)}{e_s^2} & 0 \\
		\big(\tfrac{e_l^{v}\s_l^{v''}(e_l^v)}{Y_l}\big)_{1\leq l\leq n_d} & 0 & 0 & 0 & \big(-\delta_{k,l}\tfrac{\s_k^{v''}\!\!(e_k^v)}{Y_k}\big)_{k,l}
 \end{pmatrix}
\end{equation}
\noindent with $\delta_{k,l}$ the Kronecker symbol and $\psi_k=1$ if $1\leq k\leq n_d$ and $\psi_k=0$ if $n_d<k< n_s$. Unless stated otherwise, the subscripts are in the range $1\leq k,l < n_s$, $k$ corresponding to a row index and $l$ corresponding to a column index. Likewise
\begin{equation}\label{eq:dr_dCv_dgam}
  \partial_kr \overset{\cref{eq:equiv-gamma}}{=} C_{v_t}(\Yb)\partial_k\gamma+(\gamma(\Yb)-1)\partial_kC_{v_t},
\end{equation}
\noindent and
\begin{equation}\label{eq:d2ZetadYY}
\partial_{kl}^2\zeta = \frac{r_{n_s}}{Y_{n_s}} + \frac{r_k}{Y_k}\delta_{k,l} + \frac{\partial_kC_{v_t}\partial_lC_{v_t}}{C_{v_t}(\Yb)} + \frac{C_{v_t}(\Yb)\partial_k\gamma\partial_l\gamma}{\big(\gamma-\gamma(\Yb)\big)\big(\gamma(\Yb)-1\big)}, \quad 1\leq k,l< n_s.
\end{equation}

We now prove that $\boldsymbol{\cal H}_\zeta$ is symmetric positive definite. Let ${\bf x}=(x_{1\leq i<n_s},x_\tau,x_r,x_s,x_{1\leq i\leq n_d}^v)^\top$ in $\mathbb{R}^{n_s+n_d+3}$ non zero and use the notation $\sum\equiv\sum_{k=1}^{n_s-1}$, we get
\begin{align}
  \sum_{k,l=1}^{n_s-1} x_k\partial_{kl}^2\zeta x_l &\overset{\cref{eq:d2ZetadYY}}{=} \frac{r_{n_s}(\sum x_k)^2}{Y_{n_s}} + \sum \frac{r_kx_k^2}{Y_k} + \frac{(\sum x_k\partial_kC_{v_t})^2}{C_{v_t}(\Yb)} + \frac{C_{v_t}(\Yb)(\sum x_k\partial_k\gamma)^2}{\big(\gamma-\gamma(\Yb)\big)\big(\gamma(\Yb)-1\big)} \nonumber \\
  &\overset{\cref{eq:dr_dCv_dgam}}{=} \frac{r_{n_s}(\sum x_k)^2}{Y_{n_s}} + \sum \frac{r_kx_k^2}{Y_k} + \frac{(\sum x_k\partial_kC_{v_t})^2}{C_{v_t}(\Yb)} \nonumber\\
	&+ \frac{\Big(\sum x_k\big(\partial_kr-(\gamma(\Yb)-1)\partial_kC_{v_t}\big)\Big)^2}{\big(\gamma-\gamma(\Yb)\big)r(\Yb)} \nonumber \\
  &= \frac{r_{n_s}(\sum x_k)^2}{Y_{n_s}} + \sum \frac{r_kx_k^2}{Y_k} + \frac{(\sum x_k\partial_kC_{v_t})^2}{C_{v_t}(\Yb)} \nonumber\\
	&+ \frac{(\sum x_k\partial_kr)^2}{\big(\gamma-\gamma(\Yb)\big)r(\Yb)} + \frac{(\gamma(\Yb)-1)(\sum x_k\partial_kC_{v_t})^2}{\big(\gamma-\gamma(\Yb)\big)C_{v_t}(\Yb)} - \frac{2(\sum x_k\partial_kr)(\sum x_k\partial_kC_{v_t})}{\big(\gamma-\gamma(\Yb)\big)C_{v_t}(\Yb)} \nonumber \\
	&\overset{\cref{eq:dr_dCv_dgam}}{=} \frac{r_{n_s}(\sum x_k)^2}{Y_{n_s}} + \sum \frac{r_kx_k^2}{Y_k} + \frac{(\sum x_k\partial_kr)^2}{(\gamma-1)r(\Yb)} + \frac{(\sum x_k(\gamma-1)\partial_kC_{v_t}-x_k\partial_kr)^2}{\big(\gamma-\gamma(\Yb)\big)(\gamma-1)C_{v_t}(\Yb)}, \label{eq:SPD_d2ZetadYY}
\end{align}

\noindent so using \cref{eq:hessian_zeta,eq:SPD_d2ZetadYY} we obtain
\begin{align*}
  {\bf x}^\top\boldsymbol{\cal H}_\zeta{\bf x} &= \frac{r_{n_s}(\sum x_k)^2}{Y_{n_s}} + \sum \frac{r_kx_k^2}{Y_k} + \Big(1+\frac{1}{\gamma-1}-1\Big)\frac{(\sum x_k\partial_kr)^2}{r(\Yb)} - \sum_{k=1}^{n_d}\tfrac{\s_k^{v''}\!(e_k^v)}{Y_k}(x_k^v-e_k^vx_k)^2 \\
	&+ \frac{(\sum x_k(\gamma-1)\partial_kC_{v_t}-x_k\partial_kr)^2}{\big(\gamma-\gamma(\Yb)\big)(\gamma-1)C_{v_t}(\Yb)} + r(\Yb)\frac{x_\tau^2}{\tau^2} + \frac{r(\Yb)}{\gamma-1}\frac{x_r^2}{e_r^2} \\
  &+ \frac{\big(\gamma-\gamma(\Yb)\big)C_{v_t}(\Yb)}{\gamma-1}\frac{x_s^2}{e_s^2} -2\sum x_k\Big(\partial_kr\frac{x_\tau}{\tau} + \frac{\partial_kr}{\gamma-1}\frac{x_r}{e_r} + \frac{(\gamma-1)\partial_kC_{v_t}-\partial_kr}{\gamma-1}\frac{x_s}{e_s}\Big) \\
  &= Q({\bf x}) + \frac{\big(\sum x_k\partial_kr-r(\Yb)\tfrac{x_\tau}{\tau}\big)^2}{r(\Yb)} + \frac{\big(\sum x_k\partial_kr-r(\Yb)\tfrac{x_r}{e_r}\big)^2}{(\gamma-1)r(\Yb)} \\
	&+ \frac{\big(\sum x_k((\gamma-1)\partial_kC_{v_t}-\partial_kr)-(\gamma-\gamma(\Yb))C_{v_t}(\Yb)\tfrac{x_s}{e_s}\big)^2}{(\gamma-\gamma(\Yb))(\gamma-1)C_{v_t}(\Yb)} - \sum_{k=1}^{n_d}\tfrac{\s_k^{v''}\!(e_k^v)}{Y_k}(x_k^v-e_k^vx_k)^2,
\end{align*}

\noindent with $\s_k^{v''}\!(e_k^v)=-\tfrac{r_k}{e_k^v(e_k^v+r_k\vartheta_k^v)}<0$, so the four last terms are non-negative, and
\begin{equation*}
 Q({\bf x}) = \frac{r_{n_s}(\sum x_k)^2}{Y_{n_s}} + \sum \frac{r_kx_k^2}{Y_k} - \frac{(\sum x_k\partial_kr)^2}{r(\Yb)}.
\end{equation*}

Using (\ref{eq:def2_r}), we get $\partial_kr=r_k-r_{n_s}\geq0$ since by assumption $r_{n_s}=\min_\alpha r_\alpha$, so we rewrite
\begin{equation*}
 \sum \frac{r_kx_k^2}{Y_k} = \sum \frac{r_{n_s}+\partial_kr}{Y_k}x_k^2 = \sum \frac{r_{n_s}}{Y_k}x_k^2 + \sum \frac{r_{n_s}+\sum_lY_l\partial_lr}{r(\Yb)}\frac{\partial_kr}{Y_k}x_k^2,
\end{equation*}
\noindent and hence obtain
\begin{align*}
 Q({\bf x}) &= \frac{r_{n_s}(\sum x_k)^2}{Y_{n_s}} + \sum \frac{r_{n_s}}{Y_k}\Big(1+\frac{\partial_kr}{r(\Yb)}\Big)x_k^2 + \sum_{k,l=1}^{n_s-1} \frac{\partial_kr\partial_lr}{r(\Yb)}\Big(\frac{Y_l}{Y_k}x_k^2 - x_kx_l\Big) \\
 &= \frac{r_{n_s}(\sum x_k)^2}{Y_{n_s}} + \sum \frac{r_{n_s}}{Y_k}\Big(1+\frac{\partial_kr}{r(\Yb)}\Big)x_k^2 + \frac{1}{2}\sum_{k,l=1}^{n_s-1} \frac{\partial_kr\partial_lrY_kY_l}{r(\Yb)}\Big(\frac{x_k}{Y_k}-\frac{x_l}{Y_l}\Big)^2,
\end{align*}
\noindent which is positive, providing that the $x_{1\leq k<n_s}$ are not all zero since $\partial_kr\geq0$, so ${\bf x}^\top\boldsymbol{\cal H}_\zeta{\bf x}>0$ and we conclude that $\zeta(\Yb,\tau,e_r,e_s,{\bf e}_v)$ is strictly convex. \qed

\begin{acknowledgements}
Part of this work was funded under the Onera research project JEROBOAM. One of the authors (F. Renac) gratefully acknowledges this support.
\end{acknowledgements}

\bibliographystyle{spmpsci}      
\bibliography{biblio_generale}

\end{document}